\documentclass{article}

\usepackage[english]{babel}
\usepackage{csquotes}
\usepackage[letterpaper,top=2cm,bottom=2cm,left=3cm,right=3cm,marginparwidth=1.75cm]{geometry}

\usepackage{amsmath, mathtools}
\usepackage{amsthm}
\usepackage{amsfonts}
\usepackage{amssymb}
\usepackage{bbm}
\RequirePackage[numbers]{natbib}
\usepackage{stmaryrd}
\usepackage{graphicx}
\usepackage[colorlinks=true, allcolors=blue, breaklinks]{hyperref}
\usepackage{pdfpages}
\usepackage{comment}
\usepackage{dsfont}
\usepackage{ulem}
\usepackage{caption, subcaption}
\mathtoolsset{showonlyrefs} 
\numberwithin{equation}{section}

\theoremstyle{plain}

\newtheorem{theorem}{Theorem}[section]
\newtheorem{lemma}[theorem]{Lemma}
\newtheorem{corollary}[theorem]{Corollary}
\newtheorem{proposition}[theorem]{Proposition}

\theoremstyle{remark}
\newtheorem{definition}[theorem]{Definition}
\newtheorem{remark}[theorem]{Remark}

\newcommand{\E}{\mathbb{E}}
\newcommand{\1}{\mathbbm{1}} 
\newcommand{\mP}{\mathbb{P}}

\title{Genetic Composition of Supercritical Branching Populations under Power Law Mutation Rates}
\author{Vianney Brouard\thanks{ENS de Lyon, UMPA, CNRS UMR 5669, 46 Allée d’Italie, 69364 Lyon Cedex 07, France; E-mail: vianney.brouard@ens-lyon.fr}}

\begin{document}
\maketitle
\begin{abstract}
We aim to understand the evolution of the genetic composition of cancer cell populations. To achieve this, we consider an individual-based model representing a cell population where cells divide, die and mutate along the edges of a finite directed graph $(V,E)$. The process starts with only one cell of trait $0$. Following typical parameter values in cancer cell populations we study the model under \textit{power law mutation rates}, in the sense that the mutation probabilities are parametrized by negative powers of a scaling parameter $n$ and the typical sizes of the population of interest are positive powers of $n$. Under a \textit{non-increasing growth rate condition}, we describe the time evolution of the first-order asymptotics of the size of each subpopulation in the $\log(n)$ time scale, as well as in the random time scale at which the wild-type population, resp. the total population, reaches the size  $n^{t}$. In particular, such results allow for the perfect characterization of evolutionary pathways. Without imposing any conditions on the growth rates, we describe the time evolution of the order of magnitude of each subpopulation, whose asymptotic limits are positive non-decreasing piecewise linear continuous functions.
\end{abstract}
\textbf{MSC2020 subject classifications:} 60J80, 60J27, 60F99, 92D15, 92D25. 

\section{Introduction and presentation of the model}
\label{Introduction}
Consider a population of cells characterized by a phenotypic trait, where the trait space $V$ is finite. For all $v \in V$ denote by $(Z_{v}(t))_{t \in \mathbb{R}^{+}}$ the number of cells of trait $v$ at time $t$ in the population, and $\left(\mathcal{Z}(t):=\left(Z_{v}(t)\right)_{v \in V}\right)_{t \in \mathbb{R}^{+}}$ the global process. Assume that $0 \in V$ and
\begin{align}
    \forall v \in V, Z_{v}(0)=\1_{\{v=0\}}, \text{ almost surely.}
\end{align}
Cells with trait $0$ are called \textit{wild-type cells}, and all cells with trait $v \in V\backslash\{0\}$ are called \textit{mutant cells}. The population dynamics will follow a continuous-time branching process on $\mathbb{N}_{0}^{V}$. More precisely, cells divide (giving birth to two daughter cells) and die with rates depending only on their phenotypic trait. The birth, death and growth rate functions are respectively
\begin{align}
    &\alpha : V \longrightarrow \mathbb{R}^{+}, \beta : V \longrightarrow \mathbb{R}^{+} \text{ and } \lambda:=\alpha-\beta.
\end{align}
We use the words "division" and "birth" synonymously. During a division event of a cell of trait $v \in V$, the two daughter cells may independently mutate. The mutation landscape across traits is encoded via a directed graph structure $(V,E)$ on the trait space, where $E\subset \{(v,u), \forall v,u \in V^{2}\}$ is a set of ordered pairs over $V$ such that for all $v \in V$, $(v,v) \cap E=\emptyset$, and there exists a path from $0$ to $v$ within $E$. In other words, $(V,E)$ represents a finite directed graph without self-loops, with each vertex belonging to the connected component of $0$. Mutation directly from trait $v$ to trait $u$ is possible if and only if $(v,u) \in E$. Let $\mu: E \longrightarrow [0,1]$ be a mutation kernel satisfying 
\begin{align}
\label{total mutation proba}
    \forall v \in V, \overline{\mu}(v):=\sum_{u\in V: (v,u) \in E}\mu(v,u)\leq 1.
\end{align}
A daughter cell mutates from its mother trait $v$ to trait $u$ with probability $\mu(v,u)$, meaning that $\overline{\mu}(v)$ is its total mutation probability. Notice that backward mutations are permitted in this model. 

Finally the exact transition rates from a state $z=(z_v)_{v \in V} \in \mathbb{N}_{0}^{V}$ of the process $\mathcal{Z}$ are
\begin{align}
\label{processus markov tableau des taux}
z \mapsto 
\left\{ 
\begin{array}{l l l} 
&z-\delta_{v}, \mbox{ at rate } z_{v} \beta(v), \\ 
&z-\delta_{v}+\delta_{u}+\delta_{w}, \mbox{ at rate } 2z_{v}\alpha(v)\mu(v,u)\mu(v,w)\1_{\{(v,u) \in E\}}\1_{\{(v,w) \in E\}}\1_{\{u \neq w\}}, \\
&z-\delta_{v}+2\delta_{u}, \mbox{ at rate } z_{v}\alpha(v)\mu(v,u)^{2}\1_{\{(v,u) \in E\}}, \\
&z+\delta_{v}, \mbox{ at rate } z_{v}\alpha(v)\left(1-\overline{\mu}(v)\right)^{2}+2\underset{u \in V: (u,v) \in E}{\sum} z_{u}\alpha(u)\mu(u,v)\left(1-\overline{\mu}(u)\right),
\end{array} 
\right.
\end{align}
where $\forall v \in V, \delta_{v}=\left(\1_{\{u=v\}}\right)_{u \in V}$. Throughout the paper, the growth rate of the wild-type subpopulation $\lambda(0)$ is assumed to be strictly positive, to ensure that the wild-type subpopulation survives with positive probability.  

The biological motivation for this model is to capture the time dynamics of the genetic composition of a cell population during carcinogenesis. tumors are typically detected when they reach a large size, around $10^9$ cells. The mutation rates per base pair per cell division are generally estimated to be of order $10^{-9}$, see \cite{jones2008comparative,bozic2014timing}. Thus, the framework of a \textit{power law mutation rates limit} naturally arises. A parameter $n \in \mathbb{N}$ is used to quantify both the decrease of the mutation probabilities, expressed as a negative power of $n$, and the typical population size, expressed as a positive power of $n$, at which we are interested in understanding the genetic composition. The aim is to obtain asymptotic results on the sizes of all the mutant subpopulations when $n$ goes to infinity. This is a classical stochastic regime studied in particular in  \cite{durrett2011traveling, smadi2017effect, cheek2018mutation, bovier2019crossing, cheek2020genetic, champagnat2021stochastic, coquille2021stochastic, esser2024general,blath2021stochastic, paul2023canonical, gamblin2023bottlenecks, esser2025effective}. Such a regime is referred to in \cite{cheek2018mutation, cheek2020genetic} as the \textit{large population rare mutations limit}. However, we have chosen the more precise term \textit{power law mutation rates} to distinguish this regime from the classical \textit{rare mutations limit}, which is generally used in the context of adaptive dynamics to separate evolutionary and ecological scales, where the mutation probabilities $\mu^{(n)}$ typically scale as $ e^{-Cn} \ll\mu^{(n)}\ll\frac{1}{n \log(n)}$. Indeed, under the power law mutation rates limit, the mutation probabilities are of a higher order compared to those under the rare mutations limit if for instance $\mu^{(n)} \propto n^{-\alpha}$ with $\alpha \in (0,1]$. 

To be more precise, let $L:=\{\ell(v,u)\in \mathbb{R}_{+}^{*}, \forall (v,u) \in E\}$ be a set of strictly positive labels on the edges of the graph, where $\mathbb{R}_+^*:=\{x\in \mathbb{R}, x>0\}$. Introduce a sequence of  models $\left(\mathcal{Z}^{(n)}\right)_{n \in \mathbb{N}}$, where for each $n \in \mathbb{N}$, $\mathcal{Z}^{(n)}$ corresponds to the process described above with the mutation kernel $\mu^{(n)}:E \longrightarrow [0,1]$ satisfying
\begin{equation}
\label{mutation label regime}
    \forall (v,u) \in E, n^{\ell(v,u)} \mu^{(n)}(v,u) \underset{n \to \infty}{\longrightarrow} \mu(v,u)\in \mathbb{R}^{+}.
\end{equation}
For all $t \in \mathbb{R}_{+}^{*}$, the stopping times corresponding to the first time that the wild-type subpopulation $Z_{0}^{(n)}$, respectively the total population $Z^{(n)}_{tot}:=\sum_{v \in V}Z_{v}^{(n)}$, reaches the level $n^t$, are defined as
\begin{align}
\label{Equation: Stopping times}
    &\eta_{t}^{(n)}:=\inf \left\{u \in \mathbb{R}^{+}: Z_{0}^{(n)}(u) \geq n^t \right\} \text{ and } \sigma_{t}^{(n)}:= \inf \left\{u \in \mathbb{R}^{+}: Z^{(n)}_{tot}(u)\geq n^t \right\}.
\end{align}
These are motivated by two different biological interpretations in different scenarios. For instance, when considering metastasis the wild-type subpopulation $Z_{0}^{(n)}$ may represent the primary tumor, and the mutant subpopulations $Z_{v}^{(n)},$ for all $v \in V \backslash \{0\}$, may correspond to secondary tumors. As clinicians typically have access to the size rather than the age of a tumor, it is biologically relevant to estimate the genetic composition of the secondary tumors when the primary one has reached a given size. This is mathematically encoded by examining the first-order asymptotics of $Z_{v}^{(n)}\big{(}\eta_{t}^{(n)}\big{)}$ for all $v \in V\backslash \{0\}$. Another biological scenario involves the total population $Z^{(n)}_{tot}$ representing a single tumor. It is appropriate to obtain theoretical results about the size of the mutant subpopulations $Z_{v}^{(n)}$ for all $v \in V \backslash \{0\}$ when the tumor has reached a given size. This corresponds exactly to looking at the first-order asymptotics of $Z_{v}^{(n)}\big{(}\sigma_{t}^{(n)}\big{)}$. Every time that results can be stated either with $\eta_{t}^{(n)}$ or $\sigma_{t}^{(n)}$, the following notation will be used
\begin{align}
\label{rho}
    \rho_{t}^{(n)}:=\eta_{t}^{(n)} \text{ or } \sigma_{t}^{(n)}.
\end{align}
In the present work the cell population will be studied on different time scales: the random time scale
\begin{align}
\label{Equation: random time scale}
    \Big{(}\rho_{t}^{(n)}+s\Big{)}_{(t,s)\in \mathbb{R}^{+}\times \mathbb{R}};
\end{align}
and the following deterministic approximation 
\begin{align}
\label{Equation: deterministic time scale}
     \Big{(}\mathfrak{t}^{(n)}_t+s\Big{)}_{(t,s) \in \mathbb{R}^{+}\times \mathbb{R}}, \text{ with } \mathfrak{t}^{(n)}_t:=t \frac{\log(n)}{\lambda(0)}.
\end{align}
Intuitively, the lineage of wild-type cells generated from the cancer-initiating cell constitutes the first subpopulation that will generate mutations. Understanding its growth, therefore, provides the natural time scale to consider for observing mutations. The birth and death rates of this lineage are $\alpha(0)\left(1-\overline{\mu}^{(n)}(0)\right)^2$ and $\beta(0)+\alpha(0)\left( \overline{\mu}^{(n)}(0)\right)^2$, respectively. Due to the power law mutation rates regime specified in Equation \eqref{mutation label regime}, these rates converge to $\alpha(0)$ and $\beta(0)$ when $n$ grows to $\infty$. Consequently, this lineage should therefore behave asymptotically as a birth and death process with rates $\alpha(0)$ and $\beta(0)$. Indeed, such a result emerges from the natural martingale associated to a birth and death process, see Lemma \ref{Lem:control primary pop}. In particular the growth rate of this lineage is close to $\lambda(0)$, thus this population reaches a size of order $n^{t}$ approximately at the deterministic time $\mathfrak{t}_{t}^{(n)}$, see Lemma \ref{convergence temps arret}. 

For any finite directed labeled graph $(V,E,L)$, under the following \textit{non-increasing growth rate condition} 
\begin{align}
\label{Assumption:non increasing growth rate condition}
    \forall v \in V, \lambda(v)\leq \lambda(0),
\end{align}
the first-order asymptotics of the mutant subpopulation sizes $Z_v^{(n)}$ are obtained on both random and deterministic time scales \eqref{Equation: random time scale} and \eqref{Equation: deterministic time scale}, see Theorem \ref{Theorem: non increasin growth rate graph}. Assumption \eqref{Assumption:non increasing growth rate condition} can be biologically motivated. Historically, tumor dynamics has been seen under the prism of clonal expansion of selective mutations, i.e. $\lambda(v)>\lambda(0)$. Nevertheless, the paradigm of neutral cancer evolution has recently been considered, see \cite{sottoriva2015big, ling2015extremely, williams2016identification, venkatesan2016tumor, davis2017tumor}. This means that all the selective mutations are already present in the cancer-initiating cell, and any mutations that occur subsequently are neutral (i.e. $\lambda(v)=\lambda(0)$). With Assumption \eqref{Assumption:non increasing growth rate condition}, deleterious mutations (i.e. $\lambda(v)<\lambda(0)$) are also permitted. This paradigm has been introduced because the genetic heterogeneity inside a tumor could be explained by considering neutral mutations only. Various statistical methods have been developed to infer the evolutionary history of tumors, including test of neutral evolution, see \cite{zeng2006statistical, achaz2009frequency, gunnarsson2021exact} for details.

Without any assumption on the growth rate function $\lambda$, we study the system on the deterministic time scale of Equation \eqref{Equation: deterministic time scale}. As in \cite{durrett2011traveling,bovier2019crossing,champagnat2021stochastic,coquille2021stochastic, esser2024general,blath2021stochastic, paul2023canonical, esser2025effective}, we obtain the asymptotic behavior of the \textit{stochastic exponent} processes
\begin{align}
\label{X}
     \forall v \in V, X_{v}^{(n)}(t):=\frac{\log^{+}\big{(}Z_{v}^{(n)}\big{(}\mathfrak{t}^{(n)}_t\big{)}\big{)}}{\log(n)/\lambda(0)}.
\end{align}
These results are presented in Theorem \ref{Theorem: general finite graph}. Here we are tracking the exponent of $n$ for each subpopulation, whereas Theorem \ref{Theorem: non increasin growth rate graph} is a more refined result that gives the size directly in terms of $n$. To our knowledge, it is the first model capturing this level of refinement on the asymptotic behaviors under the power law mutation rates regime \eqref{mutation label regime}. Two significant new conclusions emerge. 

First, Theorem \ref{Theorem: non increasin growth rate graph} shows the remarkable result that under Assumption \eqref{Assumption:non increasing growth rate condition} the randomness in the first-order asymptotics of the size of any mutant subpopulation is fully described by the stochasticity of only one random variable $W$, which encodes the long-time randomness of the lineage of wild-type cells issued from the cancer-initiating cell. More precisely, the stochasticity for any mutant subpopulation size is fully driven, at least to first order, by the randomness in the growth of the wild-type subpopulation and not by the dynamics of any lineage of a mutant cell nor by the stochasticity generating the mutations.

Second, Theorem \ref{Theorem: non increasin growth rate graph} characterizes the exact effective evolutionary pathways, in the sense of the pathways that asymptotically contribute to the growth of the mutant subpopulations. More precisely, if the length of a pathway is defined as the sum of the labels of its edges, asymptotic results on the stochastic exponent give that for any trait $v$, among the pathways from $0$ to $v$, only those of minimal length can asymptotically contribute to the growth of trait $v$. However, having results on the first-order asymptotics of the size of the mutant subpopulations allows us to see which of those minimal length pathways actually contribute to the dynamics of trait $v$. More specifically, among the minimal length pathways only those with the maximal number of neutral mutations on their edges asymptotically contribute to the growth of trait $v$. Indeed, for each neutral mutation in a pathway, an additional multiplicative factor of order $\log(n)$ appears in the first-order asymptotics. Such a theoretical result opens the door for developing new statistical methods to infer the underlying graph structure from data, i.e. to infer the evolutionary history of tumors, as well as for designing new statistical estimators for biologically relevant parameters, alongside new neutral (and deleterious) cancer evolution tests.

Moreover it is, to our knowledge, the first time that this power law mutation rates limit has been studied on the random time scale of Equation \eqref{Equation: random time scale}. From a biological point of view, it is more interesting to obtain results on such a random time scale rather than a deterministic one. We find that the randomness in the first-order asymptotics of any mutant subpopulation size is fully described by the stochasticity in the survival of the lineage of wild-type cells issued from the cancer-initiating cell. 

In \cite{cheek2018mutation, cheek2020genetic}, Cheek and Antal study a model that can be seen as an application of the model of the present work via a specific finite directed labeled graph $(V,E,L)$, the finite-dimensional hypercube. Among their results, they fully characterize, in distribution, the asymptotic sizes of all the mutant subpopulations around the random time at which the wild-type subpopulation reaches the typical size allowing mutations to occur. In their setting, it corresponds to $\big{(}\eta_{1}^{(n)}+s\big{)}_{s \in \mathbb{R}}$. In particular, they obtain that the asymptotic sizes of all the mutant subpopulations around this random time $\eta_{1}^{(n)}$ are finite almost surely, following generalized Luria-Delbrück distributions, see \cite[Theorem 5.1]{cheek2020genetic}. The original Luria and Delbrück model, introduced in \cite{luria1943mutations}, has generated many subsequent works, see in particular \cite{lea1949distribution,kendall1960birth,hamon2012statistics,keller2015mutant,kessler2015scaling, cheek2018mutation, cheek2020genetic}. Two major features explain the latter result. The first one is that asymptotically only a finite number of mutant cells are generated from the wild-type subpopulation until time $\eta_{1}^{(n)}$, following a Poisson distribution. The second one is that all the lineages of the mutant cells generated from mutational events of the wild-type subpopulation have, up to time $\eta_{1}^{(n)}$, only an asymptotically finite random time to grow, which is exponentially distributed. We extend their results to larger times, typically when the total mutation rate from the subpopulation of a trait $v$ to the subpopulation of a trait $u$ is growing as a positive power of $n$, instead of remaining finite.

In \cite{durrett2011traveling}, Durrett and Mayberry study the exponentially growing Moran model. They consider the same mutation regime; their total population size grows exponentially fast at a fixed rate, and new individuals in the population choose their trait via a selective frequency-dependent process. In Theorem \ref{Theorem: general finite graph}, a similar result is obtained for the case of a multitype branching population. In particular, for this setting, the exponential speed of the total population (and of the dominant subpopulations) growth evolves over time. More specifically, we show that the speed is a non-decreasing piecewise constant function going from $\lambda(0)$ to $\underset{v \in V}{\max}\lambda(v)$, and taking values only from the set $\left\{\lambda(v), \forall v \in V \right\}$, see Theorem \ref{Theorem: general finite graph}. 

In \cite{cheek2018mutation, smadi2017effect,bovier2019crossing,champagnat2021stochastic,coquille2021stochastic, esser2024general,blath2021stochastic, paul2023canonical, esser2025effective}, the authors consider the power law mutation rates limit of Equation \eqref{mutation label regime} in the special case where all different traits mutate with the same scaling of a fixed order of a negative power of $n$. In contrast, in the present work, the power law mutation rates are more general by allowing traits to mutate with different scalings, as in \cite{cheek2020genetic, gamblin2023bottlenecks}. 

As in \cite{cheek2018mutation,cheek2020genetic}, compared to the different models in \cite{durrett2011traveling, smadi2017effect,champagnat2021stochastic,coquille2021stochastic, esser2024general,blath2021stochastic,gamblin2023bottlenecks,esser2025effective}, the initial population $\mathcal{Z}^{(n)}(0)$ is not assumed to have a macroscopic size. This introduces an additional randomness in how the wild-type subpopulation stochastically grows to reach a macroscopic size. However, contrary to \cite{cheek2018mutation,cheek2020genetic}, we condition neither on the survival of the wild-type subpopulation nor on the finiteness of the stopping times of Equation \eqref{rho}. 

In \cite{nicholson2019competing}, Nicholson and Antal study a similar model under a slightly less general non-increasing growth rate condition. More precisely, in their case, all the growth rates of the mutant populations are strictly smaller than the growth rate of the wild-type population: $\forall v \in V \backslash \{0\}, \lambda(v)<\lambda(0)$. However, the main difference remains the mutation regime. In their case, only the last mutation is in the power law mutation rates regime, while all other mutations have a fixed probability independent of $n$. In Theorem \ref{Theorem: non increasin growth rate graph} the case where all mutations are in the power law mutation rates regime is analyzed. Additionally, Nicholson and Antal were interested in obtaining the distribution of the first time that a mutant subpopulation gets a mutant cell, whereas in the present work, the first-order asymptotics of the sizes of the mutant subpopulations are studied over time. 

In \cite{nicholson2023sequential}, Nicholson, Cheek and Antal study the case of a mono-directional graph where time tends to infinity with fixed mutation probabilities. In particular, they obtain the almost sure first-order asymptotics of the mutant subpopulation sizes. Under a non-increasing growth rate condition, they are able to characterize the distribution of the random variables they obtain in the limit. Without any condition on the growth rates, they study the distribution of the random limit under the small mutation probabilities limit, using the hypothesis of an approximating model with less stochasticity. Note that the mutation regime they study is not the power law mutation rates limit of Equation \eqref{mutation label regime} as considered in the present work. Under the latter regime, both the size of the population goes to infinity and the mutation probabilities to $0$, through the parameter $n$. 

In \cite{gunnarsson2025limit}, Gunnarsson, Leder and Zhang study a similar model to the one in the present work and are also interested in capturing the time-evolution of the genetic diversity of a cell population, using in their case the well-known summary statistic called the \textit{site frequency spectrum} (SFS for short). The main difference lies in the considered mutation regime which is not the power law mutation rates limit. In their case, the mutation probabilities are fixed. Additionally, they restrict the study to the neutral cancer evolution case. In particular, as in the present work, they capture the first-order asymptotics of the SFS at a fixed time and at the random time at which the population first reaches a certain size. Two noticeable similarities in the results are that the first-order asymptotics of the SFS converge to a random limit when evaluated at a fixed time and to a deterministic limit when evaluated at the previous stochastic time. One could argue that in the present work the correct convergence in the latter case is actually a stochastic limit. But the randomness is fully given by the survival of the wild-type lineage of the cancer-initiating cell, so conditioned on such an event, in the end, the limit is a deterministic one. In particular the results of Gunnarson, Leder and Zhang are all conditioned on the non extinction of the population. 

In \cite{gamblin2023bottlenecks}, Gamblin, Gandon, Blanquart and Lambert study a model of an exponentially growing asexual population that undergoes cyclic bottlenecks under the power law mutation rates limit. Their trait space is composed of 4 subpopulations $00, 10, 01$ and $11$, where two pathways of mutations are possible: $00 \mapsto 10 \mapsto 11$ and $00 \mapsto 01 \mapsto 11$. They study the special case where one mutation ($10$) has a high rate but is a weakly beneficial mutation whereas the other mutation ($01$) has a low rate but is a strongly beneficial mutation. In particular they show the notable result that due to cyclic bottlenecks only a unique evolutionary pathway unfolds, but modifying their intensity and period implies that all pathways can be explored. Their work relies on a deterministic approximation of the wild-type subpopulation $00$ and some parts of the analysis of the model's behavior are obtained only through heuristics. The present work, and more specifically Theorem \ref{Theorem: general finite graph}, because it considers selective mutations, can be used and adapted to consider the case of cyclic bottlenecks in order to prove rigorously their results, both in the specific trait space that they consider and in a general finite directed trait space. 

The rest of the paper is organized as follows. In Section \ref{Main results}, we give the results and their biological interpretations. Sections \ref{Section: proof mono directional graph} and \ref{Section: generalisation of the proof} are dedicated to proving Theorem \ref{Theorem: non increasin growth rate graph}, which assumes Equation \eqref{Assumption:non increasing growth rate condition}. In Section \ref{Section: proof mono directional graph}, we provide the mathematical construction of the model for an infinite mono-directional graph using Poisson point measures, as well as the proof in this particular case. The generalization of the proof from an infinite mono-directional graph to a general finite directed graph is given in Section \ref{Section: generalisation of the proof}.

\section{Main results and biological interpretation} 
\label{Main results}
In Subsection \ref{result growth rate condition} the first-order asymptotics of the size of each mutant subpopulation on the time scales \eqref{Equation: random time scale} and \eqref{Equation: deterministic time scale} are provided under the non-increasing growth rate condition \eqref{Assumption:non increasing growth rate condition}. In Subsection \ref{result general case}, the asymptotic result on the stochastic exponent of each mutant subpopulation is presented without any assumption on the growth rate function $\lambda$. In each subsection, biological interpretations of the results are provided.

\subsection{First-order asymptotics of the size of the mutant subpopulations under the non-increasing growth rate condition}
\label{result growth rate condition}
In this subsection, we assume that the graph $(V,E,L)$ satisfies the non-increasing growth rate condition given by Equation \eqref{Assumption:non increasing growth rate condition}. 

\subsubsection{Heuristics for a general finite graph under the non-increasing growth rate condition}
The next definitions, notations and results are initially motivated by heuristics for the simplest possible graph: a wild-type and a mutant population where only mutations from wild-type to mutant cells are considered. Specifically, we consider the graph $(V,E,L)=\left(\{0,1\}, \{(0,1)\},\{\ell(0,1)\}\right)$, as illustrated in Figure~\ref{Figure: 2 traits}. 
\begin{figure}
\centering
\includegraphics[width=0.7 \linewidth]{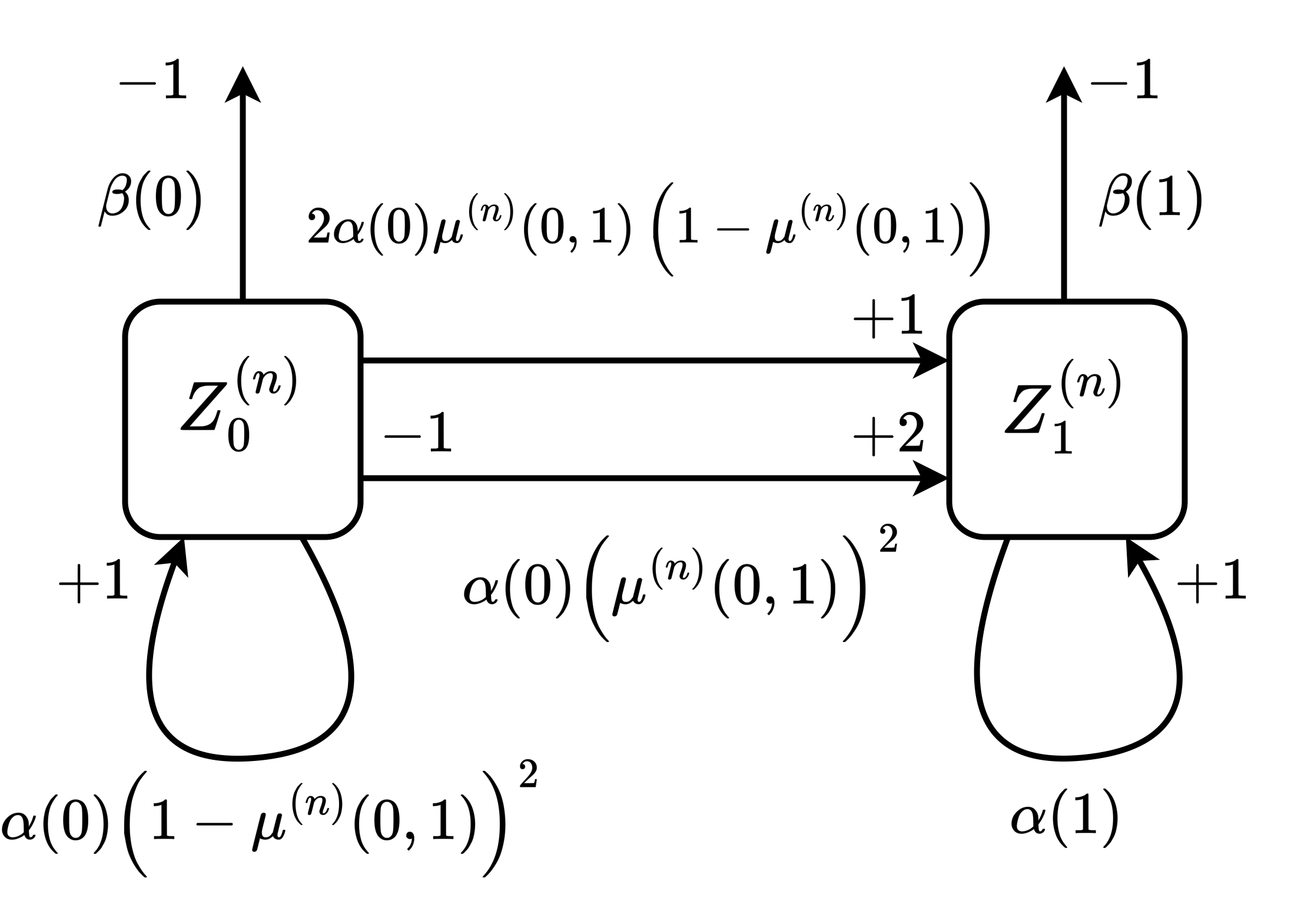}\caption{Graphical representation of the model with two traits and without backward mutation}
\label{Figure: 2 traits}
\end{figure} 

Under the power law mutation rates regime, the intrinsic birth and death rates of the wild-type subpopulation, $\alpha(0)\left(1-\mu^{(n)}(0,1)\right)^{2}$ and $\beta(0)+\alpha(0)\left(\mu^{(n)}(0,1)\right)^2$, respectively, are so close to $\alpha(0)$ and $\beta(0)$ that its natural martingale asymptotically behaves like that of a birth and death process with rates $\alpha(0)$ and $\beta(0)$ (see Lemma \ref{Lem:control primary pop}). This allows us to approximate the growth of the wild-type subpopulation as an exponential growth with parameter $\lambda(0)$. Then, if it survives, at time $\mathfrak{t}^{(n)}_{t}$ (see \eqref{Equation: deterministic time scale}), its size is of order $\Theta\left(n^{t}\right)$ (see Lemma \ref{convergence temps arret}), where we use the standard Landau notation for $\Theta$. Given this, we understand why it is necessary to wait until time $\mathfrak{t}^{(n)}_{\ell(0,1)}$ before observing any mutations. Indeed, with a mutation probability scaling as $n^{-\ell(0,1)}$, the total mutation probability up to time $\mathfrak{t}^{(n)}_{t}$ scales as $$\int_0^{t} n^{u}n^{-\ell(0,1)}d\big{(}u\frac{\log(n)}{\lambda(0)}\big{)}=\frac{n^{-\ell(0,1)}}{\lambda(0)}\left(n^{t}-1\right),$$ which starts to be of order $1$ for $t \geq \ell(0,1)$. This is formalized by D. Cheek and T. Antal in \cite{cheek2018mutation, cheek2020genetic}. An illustration is provided in Figure~\ref{Figure: Heuristics1}. 
\begin{figure}
\centering
\includegraphics[width=1 \linewidth]{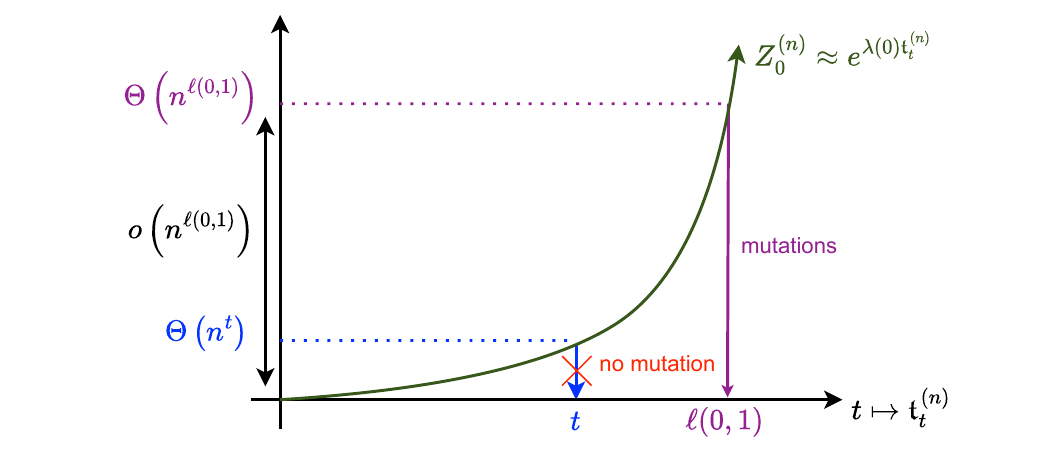}\caption{Heuristics for the first-occurrence time of mutant cells}
\label{Figure: Heuristics1}
\end{figure} 

Some heuristics for the size of the mutant subpopulation at time $\mathfrak{t}^{(n)}_{t},$ for $t \geq \ell(0,1)$, can also be derived. For $\ell(0,1) \leq u \leq t$, the number of new mutations generated at time $\mathfrak{t}_{u}^{(n)}$ scales as $\exp\big{(}\lambda(0)(u-\ell(0,1))\frac{\log(n)}{\lambda(0)}\big{)}$. The remaining time for these new mutant cells to grow exponentially at rate $\lambda(1)$ until time $\mathfrak{t}_{t}^{(n)}$ is $\mathfrak{t}_{t-u}^{(n)}$. This implies that their lineages reach a size at time $\mathfrak{t}_{t}^{(n)}$ of order 
\begin{align}
\label{size order mut pop}
    \Theta \Big{(}\exp \Big{(}\left[\lambda(1)t+(\lambda(0)-\lambda(1))u-\lambda(0)\ell(0,1)\right]\frac{\log(n)}{\lambda(0)}\Big{)}\Big{)}.
\end{align}
Two scenarios are then possible: 
\begin{itemize}
    \item If $\lambda(1)<\lambda(0)$: Equation \eqref{size order mut pop} is maximized for $u=t$ and equals $n^{t-\ell(0,1)}$. This means that the dynamics of the mutant subpopulation is driven by mutations from the wild-type subpopulation rather than by its intrinsic growth. More precisely, its size order at time $\mathfrak{t}_{t}^{(n)}$ is determined entirely by the mutations generated at that time -and so is of order $n^{t-\ell(0,1)}$-  and not by the lineages arising from mutations at earlier times. Biologically, these mutations are termed \textit{deleterious}.
    \item If $\lambda(1)=\lambda(0)$: Equation \eqref{size order mut pop} is independent of $u$ and equals $\Theta\left(n^{t-\ell(0,1)}\right)$ for any $\ell(0,1) \leq u \leq t$. This indicates that lineages of mutant cells generated from mutations at any time between $\mathfrak{t}_{\ell(0,1)}^{(n)}$ and $\mathfrak{t}_{t}^{(n)}$ have the same order of size at time $\mathfrak{t}_{t}^{(n)}$. In other words, there is a balance in the dynamics of the mutant subpopulation between the contributions of mutations and its intrinsic growth. This is a consequence of assuming $\lambda(1)=\lambda(0)$. These mutations are referred to as \textit{neutral mutation}, even though biologically speaking, this would more precisely mean the restrictive condition $\alpha(1)=\alpha(0)$ and $\beta(1)=\beta(0)$. Therefore, to capture the total size of the mutant subpopulation at time $\mathfrak{t}^{(n)}_{t}$, one must integrate all lineages resulting from mutational events over the time $\mathfrak{t}_{u}^{(n)},$ for $\ell(0,1) \leq u \leq t$. This gives exactly the order $\Theta\left((t-\ell(0,1)) \log(n) n^{t-\ell(0,1)}\right)$.
\end{itemize}
To summarize, for this simple graph, the size of the mutant subpopulation after time $\mathfrak{t}_{\ell(0,1)}^{(n)}$ scales as
\begin{align}
\label{Equation: heuristic size mutant pop}
    \Theta\Big{(}n^{t-\ell(0,1)}\left[\1_{\{\lambda(0)>\lambda(1)\}}+\1_{\{\lambda(0)=\lambda(1)\}}(t-\ell(0,1))\log(n)\right]\Big{)}.
\end{align}
Notice, in particular, that in any case, the mutant subpopulation exhibits exponential growth at rate $\lambda(0)$ after time $\mathfrak{t}^{(n)}_{\ell(0,1)}$, as indicated by the factor $n^{t-\ell(0,1)}$. An illustration of this heuristic can be found in Figure~\ref{Figure: schema taille pop}, which visually represents the growth dynamics of the mutant subpopulation over time.
\begin{figure}
\centering
\includegraphics[width=1 \linewidth]{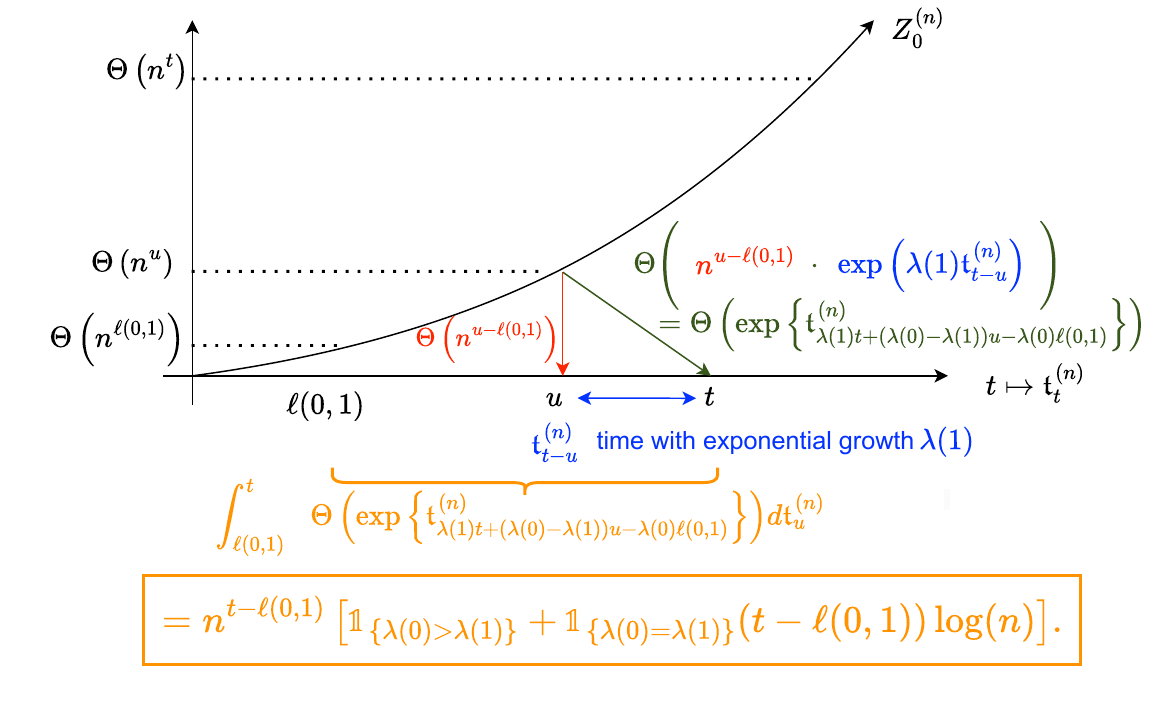}\caption{Heuristics for the size of the mutant subpopulation after time $\mathfrak{t}_{\ell(0,1)}^{(n)}$}
\label{Figure: schema taille pop}
\end{figure} 

These heuristics on this simple graph can be used as an elementary brick for developing heuristics on a general finite graph. Considering a vertex $v \in V\backslash \{0\}$, there may be multiple mutational pathways from the initial vertex $0$ to $v$. It is important to understand which pathways actually contribute to the size order of the mutant subpopulation of trait $v$. Using both the previous heuristics on the time required for mutations to occur and the fact that after this time, the mutant subpopulation grows exponentially at rate $\lambda(0)$, along with an additional $\log(n)$ factor if the mutation is neutral, it seems natural to iteratively apply this reasoning to a mutational pathway, encoded via a mono-directional graph. In the following, we will use the term \textit{'walk'} instead of \textit{'pathway'}, favoring the nomenclature of graph theory over the biological terminology. For any given walk from $0$ to $v$, the needed time $u$, in the time scale $\mathfrak{t}^{(n)}_{u}$, to observe a cell of trait $v$ generated via this specific walk is the sum of the labels of the edges along this walk, which is referred to as the \textit{length} of the walk. After this time, this subpopulation of cells of trait $v$ grows exponentially at rate $\lambda(0)$. Moreover, as observed in \eqref{Equation: heuristic size mutant pop}, for each neutral mutation along the walk, an additional multiplicative factor of order $\log(n)$ is included in the size order. This leads to three key observations about the total mutant subpopulation of trait $v$:
\begin{itemize}
    \item First occurrence of cells: Cells of trait $v$ first appear after a time equal to the minimum of the lengths of all walks from $0$ to $v$.
    \item Effective evolutionary pathways: After this time, only walks whose lengths equal this minimum might contribute to the size order of the mutant subpopulation of trait $v$. This is because any time delay creates an exponential delay in the size order. This fact is captured asymptotically in Theorem \ref{Theorem: general finite graph}.
    \item Neutral mutation factor: The additional multiplicative factor of $\log(n)$ due to neutral mutations implies that, among the walks from $0$ to $v$ with lengths equal to the aforementioned minimum, only those with the maximal number of neutral mutations actually contribute to the size order of the mutant subpopulation of trait $v$. Specifically, these walks contribute with a factor of $\log(n)$ raised to the power given by this maximal number of neutral mutations. This fact is asymptotically captured in Theorem \ref{Theorem: non increasin growth rate graph}. Additionally, for each of these admissible walks, an additional time integral is obtained at each neutral mutation, as observed in \eqref{Equation: heuristic size mutant pop}. 
\end{itemize}
An illustration of this reasoning is provided with an example in Figure~\ref{Figure: Heuristics2}.
\begin{figure}
\centering
\includegraphics[width=0.8 \linewidth]{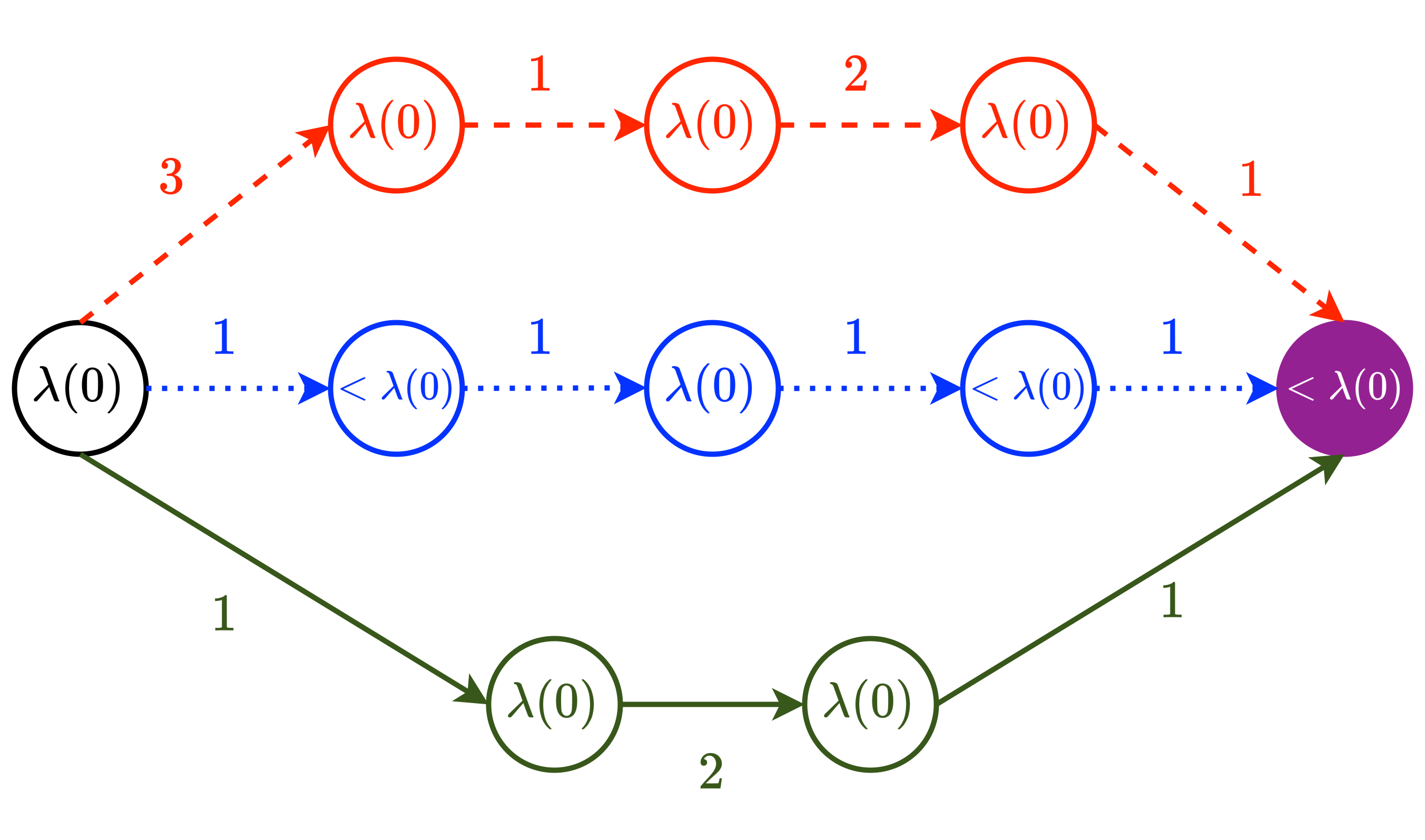}\caption{Heuristics for the contribution of walks to the size order of the plain purple mutant subpopulation: in this example, the dashed red walk has a length of $7$, while the dotted blue and plain green walks have a length of $4$. Therefore, only the two latter walks may contribute to the size order of the plain purple mutant subpopulation, making them sub-admissible walks. However, the dotted blue walk has only one neutral mutation, whereas the plain green walk has two neutral mutations. As a result, only the plain green walk will ultimately contribute to the size order of the purple mutant subpopulation. For $t \geq 4$, at time $\mathfrak{t}^{(n)}_{t}$, it will grow as $\log^{2}(n)e^{\lambda(0)\mathfrak{t}^{(n)}_{t-4}}$. Notice, in particular, that the dashed red walk has the maximal number of neutral mutations, which is $3$. However, since it is not a sub-admissible walk, the multiplicative factor of $\log(n)$ remains $2$ instead of $3$.} 
\label{Figure: Heuristics2}
\end{figure}

\subsubsection{Notations and definitions:} 

Now, the natural definitions derived from these heuristics are formally established before presenting the results.

\begin{definition}[Deleterious and neutral vertices] 
    A vertex $v \in V$ is called a neutral vertex if $\lambda(v)=\lambda(0)$, and a deleterious vertex if $\lambda(v)<\lambda(0)$. 
\end{definition}
\begin{remark}
    In the previous definition, the terms "neutral" or "deleterious" for a mutation are based on comparing its growth rate to that of the wild-type subpopulation. However, it is possible to have a mutation from a vertex $v$ to a vertex $u$ where $\lambda(v)<\lambda(u) \leq \lambda(0)$. Although such a mutation could theoretically be considered selective, since $\lambda(u)>\lambda(v)$, the previous definition categorizes it as either neutral or deleterious, depending on the value of $\lambda(u)$ relative to $\lambda(0)$. This nomenclature emerges from the fact that, under Assumption \eqref{Assumption:non increasing growth rate condition}, any mutant subpopulation grows exponentially at rate $\lambda(0)$, as developed in the earlier heuristics. Thus, this legitimates the previous definition, assuming \eqref{Assumption:non increasing growth rate condition} holds. 
\end{remark}

The following definition provides a structured framework to analyze the contribution of evolutionary pathways to the growth of mutant subpopulations. It does so by introducing the adapted vocabulary, for the neutral and deleterious evolutionary context of the model, associated with walks in labeled graphs. We use the term \textit{'walk'} here according to the standard nomenclature of graph theory.
\begin{definition} [Walk in the graph]
    A walk $\gamma=(v(0),\cdots,v(k))$ in the graph $(V,E)$ is defined as a sequence of vertices linking $v(0)$ to $v(k)$ such that for all $0 \leq i \leq k, v(i) \in V$, and for all $0 \leq i \leq k-1, (v(i),v(i+1)) \in E$. We will sometimes use the term \textit{'path'} to refer to a walk that visits only distinct vertices. Given a walk $\gamma=(v(0),v(1),\cdots v(k))$ in the labeled graph $(V,E,L)$, we define: 
\begin{itemize}
    \item The sum of the labels of the edges and the sum over the first $i$ edges of the walk $\gamma$, respectively:
    \begin{align}
        t(\gamma):= \sum_{i=0}^{k-1}\ell(v(i),v(i+1)) \text{ and for all } i\leq k, t_{\gamma}(i):=\sum_{j=0}^{i-1}\ell(v(j),v(j+1)).
    \end{align}
    \item The subset of neutral heads of the edges of the walk $\gamma$ and its cardinality:
    \begin{align}
        \gamma_{neut}=\{v(i), 1 \leq i \leq k: \lambda(v(i))=\lambda(0)\} \text{ and } \theta(\gamma):= \vert \gamma_{neut}\vert.
    \end{align}
    \item The weights $w_{neut}(\gamma)$ and $w_{del}(\gamma)$ associated with the neutral and deleterious vertices of the walk $\gamma$, respectively:
    \begin{align}
        &w_{neut}(\gamma):=\prod_{1 \leq i \leq k, \lambda(v(i))=\lambda(0)}\frac{2\alpha(v(i-1))\mu(v(i-1),v(i))}{\lambda(0)},\\
        &w_{del}(\gamma):=\prod_{1 \leq i \leq k, \lambda(v(i))<\lambda(0)}\frac{2\alpha(v(i-1))\mu(v(i-1),v(i))}{\lambda(0)-\lambda(v(i))}.
    \end{align}
    Along a walk, the constant of the asymptotic contribution of a vertex- depending on its parameters and those of the upstream vertex- takes a distinct form based on whether the vertex is neutral or deleterious. This distinction motivates the use of the separate weights $w_{neut}(\gamma)$ and $w_{del}(\gamma)$. 
    \item The time dependence associated with the neutral vertices: Let $\sigma$ be an increasing function from $\{1, \cdots,\theta(\gamma)\}$ to $\{1, \cdots, k \}$, such that $v(\sigma_i)$ is the i-th neutral vertex of the walk $\gamma$. For all $t>0$, define the multiple integral $I_{\gamma}(t)$ as
\begin{align}
    I_{\gamma}(t):=\int_{t_{\gamma}(\sigma_{\theta(\gamma)})}^{t \vee t_{\gamma}(\sigma_{\theta(\gamma)})}\int_{t_{\gamma}(\sigma_{\theta(\gamma)-1})}^{u_1} \cdots \int_{t_{\gamma}(\sigma_{\theta(\gamma)-k})}^{u_k}\cdots \int_{t_{\gamma}(\sigma_{1})}^{u_{\theta(\gamma)-1}}du_{\theta(\gamma)}\cdots du_1.
\end{align}
Along a walk, for each neutral vertex that is visited, an additional integral over the time parameter appears in the asymptotic limit, as described in the heuristics. This motivates the definition of $I_{\gamma}(t)$. 
\item The weight of the walk $\gamma$ at time $t$:
\begin{align} \label{weight}
    w_{\gamma}(t):=w_{del}(\gamma) w_{neut}(\gamma) I_{\gamma}(t).
\end{align}
This expression captures the total weight of a walk $\gamma$ at time $t$, accounting for both the deleterious and neutral visited vertices, and the integrals over the time parameters associated with these neutral vertices. 
\end{itemize}
\end{definition}
The next remark provides a recursive formula for computing the weight of a walk $\gamma$ at a given time $t$.
\begin{remark} \label{Remark:recursive structure weight path}
    The weight $w_{\gamma}(t)$ of the walk $\gamma=(v(0), \cdots, v(k))$ at time $t$ can be recursively expressed in terms of the weight $w_{\overset{\leftarrow}{\gamma}}(t)$ associated with the walk $\overset{\leftarrow}{\gamma}:=(v(0),\cdots,v(k-1))$, which is the same walk as $\gamma$ up to the second-to-last vertex (i.e. without the final vertex $v(k)$). The recursive equation, which considers whether the last vertex $v(k)$ is deleterious or neutral, is given by
    \begin{align}
        w_{\gamma}(t)= &2\alpha(v(k-1)) \mu(v(k-1),v(k)) \\
        &\cdot\Big{(}\1_{\{\lambda(k)<\lambda(0)\}} \frac{1}{\lambda(0)-\lambda(v(k))}w_{\overset{\leftarrow}{\gamma}}(t)+\1_{\{\lambda(k)=\lambda(0)\}} \frac{1}{\lambda(0)} \int_{t(\overset{\leftarrow}{\gamma})}^{t \vee t(\overset{\leftarrow}{\gamma})} w_{\overset{\leftarrow}{\gamma}}(s)ds\Big{)}.
    \end{align}
\end{remark}

\begin{definition}[Admissible walks] \label{Def: Admissible path} 
For all $v \in V$, let $P(v)$ denote the set of all walks $\gamma$ in the graph $(V,E)$ that link the vertex $0$ to the vertex $v$. We define the:
\begin{itemize}
    \item The minimum total label sum among all walks from vertex $0$ to vertex $v$: $$t(v):=\min_{\gamma \in P(v)} t(\gamma).$$
    \item The maximum number of neutral vertices among the shortest walks from vertex $0$ to vertex $v$: $$\theta(v):=\max_{\gamma \in P(v), t(\gamma)=t(v)} \theta(\gamma).$$
    \item The set of admissible walks from vertex $0$ to vertex $v$: $$A(v):= \{\gamma \in P(v): t(\gamma)=t(v) \text{ and } \theta(\gamma) =\theta(v)\}.$$
\end{itemize}
\end{definition}

\begin{remark}
    In the previous definition, the set $A(v)$ is referred to as the set of admissible walks because, as indicated by the heuristics, only walks belonging to $A(v)$ contribute to the growth dynamics of the mutant subpopulation of trait $v$. This is formally established in Theorem \ref{Theorem: non increasin growth rate graph}.
\end{remark}

\subsubsection{First-order asymptotic results} 

Under Assumption \eqref{Assumption:non increasing growth rate condition}, the more refined result can now be formally stated. The model is mathematically constructed in Section \ref{Section: generalisation of the proof} (see \eqref{Equ6}, \eqref{Equ7}, \eqref{ModelConstructionK}, \eqref{ModelConstuctionH} and \eqref{ModelConstructionZ}) using independent Poisson Point Measures. The following theorem provides the asymptotic results for this specific mathematical construction of the model. The convergences are, in particular, obtained in probability. For any mathematical construction of the model other than the one given in Section \ref{Section: generalisation of the proof}, the convergences hold at least in distribution in the appropriate Skorokhod space, see Remark \ref{Rq Thm 1}. A motivation for the normalizing term $d_{v}^{(n)}(t,s)$, introduced in the the following theorem, is provided below in Remark \ref{Rq Thm 1}.
\begin{theorem}
\label{Theorem: non increasin growth rate graph}
Assume that the general finite directed labeled graph $(V,E,L)$ satisfies both the power law mutation rates regime described in \eqref{mutation label regime} and the non-increasing growth rate graph condition given in \eqref{Assumption:non increasing growth rate condition}. Let $h_{n}=\frac{\log(n)}{\log(\log(n))\theta_{\max}+\varphi_n}$, where $\varphi_n \underset{n \to \infty}{\to} \infty$ such that $h_n \underset{n \to \infty}{\to}\infty$ and where $\theta_{max}:=\max_{v \in V \backslash \{0\}} \theta(v)$. Let also $\psi_{n}$ such that $\sqrt{\log(n)}=o(\psi_{n})$. Define for all $(t,s) \in \mathbb{R}^{+}\times \mathbb{R}$,
\begin{align}
\label{Equation: denominateur}
    d_{v}^{(n)}(t,s):=&\1_{\left\{t \in [0,t(v)-h_{n}^{-1})\right\}}+\1_{\left\{t \in [t(v)-h_{n}^{-1},t(v))\right\}}\psi_{n}\log^{\theta(v)-1}(n)\\
    &\hspace{4cm}+\1_{\left\{t \in [ t(v),\infty)\right\}}n^{t-t(v)}\log^{\theta(v)}(n)e^{\lambda(0) s}.
\end{align}
Let $(T,M) \in \left(\mathbb{R}_{+}^{*}\right)^{2}$ and $0<T_1<T_2$. Using the mathematical definition of the model given in Section \ref{Section: generalisation of the proof} (see \eqref{Equ6} and \eqref{Equ7}), there exists a random variable $W$, properly defined in \eqref{W def general finite graph}, satisfying 
\begin{align}
    W\overset{law}{:=}Ber \Big{(}\frac{\lambda(0)}{\alpha(0)}\Big{)} \otimes Exp \Big{(}\frac{\lambda(0)}{\alpha(0)}\Big{)},
\end{align}
such that for all $v \in V\backslash \{0\}$, we obtain the convergence results in probability in $L^{\infty}([0,T] \times [-M,M])$ for Equation \eqref{Equation: conv neutral case} and in $L^{\infty}\left([T_{1},T_{2}]\times [-M,M]\right)$ for Equations \eqref{Equation: convergence deleterious case}, \eqref{Equation: conv neutral case random scale} and \eqref{Equation: convergence deleterious case in random scale}:
\begin{itemize}
    \item \textbf{Deterministic time scale \eqref{Equation: deterministic time scale}:} \\
    If $\lambda(v)=\lambda(0)$, then
\begin{align}
\label{Equation: conv neutral case}
    & \frac{Z_{v}^{(n)}\Big{(} \mathfrak{t}^{(n)}_{t}+s\Big{)}}{d^{(n)}_{v}(t,s)}  \underset{n \to \infty}{\longrightarrow} W \sum_{\gamma \in A(v)} w_{\gamma}(t).
\end{align}
If $\lambda(v)<\lambda(0)$, then
\begin{align}
\label{Equation: convergence deleterious case}
    &\frac{Z_{v}^{(n)}\Big{(}\mathfrak{t}^{(n)}_{t(v)+t}+s\Big{)}}{n^{t}\log^{\theta(v)}(n)e^{\lambda(0) s}} \underset{n \to \infty}{\longrightarrow}W \sum_{\gamma \in A(v)} w_{\gamma}(t(v)+t).
\end{align}
\item \textbf{Random time scale \eqref{Equation: random time scale}:} Consider $\big{(}\rho_{t}^{(n)}\big{)}_{t \in \mathbb{R}^{+}}$ as defined in \eqref{rho}. \\
If $\lambda(v)=\lambda(0)$, then 
\begin{align}
\label{Equation: conv neutral case random scale}
    & \frac{Z_{v}^{(n)}\Big{(} \rho_{t}^{(n)}+s\Big{)}}{d_{v}^{(n)}(t,s)} \underset{n \to \infty}{\longrightarrow}\1_{\{W>0\}}\sum_{\gamma \in A(v)}w_{\gamma}(t).
\end{align}
If $\lambda(v)<\lambda(0)$, then 
\begin{align}
\label{Equation: convergence deleterious case in random scale}
    & \frac{Z_{v}^{(n)}\Big{(} \rho_{t(v)+t}^{(n)}+s\Big{)}}{n^{t}\log^{\theta(v)}(n)e^{\lambda(0) s}} \underset{n \to \infty}{\longrightarrow} \1_{\{W>0\}}\sum_{\gamma \in A(v)} w_{\gamma}(t(v)+t).
\end{align}
\end{itemize}  
\end{theorem}
The proof of Theorem \ref{Theorem: non increasin growth rate graph} relies on a martingale approach using Doob's and Maximal Inequalities. The initial step involves controlling the growth of the lineage of wild-type cells originated from the initial cell, for both the deterministic and random time scales \eqref{Equation: deterministic time scale} and \eqref{Equation: random time scale} (see Lemma \ref{Lem:control primary pop general finite graph} and \ref{Lem: deterministic approx tau general finite graph}). For any vertex $v \in V \backslash \{0\}$, there may be several mutational walks in the graph $(V,E)$ that start from $0$ and lead to $v$. Understanding the contribution of each of these walks to the first-order asymptotics of the size of the mutant subpopulation of trait $v$ is essential. The proof proceeds in 2 steps:
\begin{itemize}
    \item Consider an infinite mono-directional graph under Assumption \eqref{Assumption:non increasing growth rate condition} and establish the result for this specific graph, see Section \ref{Section: proof mono directional graph}. Performing this step for an infinite graph is particularly helpful in handling cycles (such as backward mutations) in a general finite directed graph.
    \item Identify and exclude walks from the initial vertex $0$ to $v$ that do not contribute to the first-order asymptotics of the size of the mutant subpopulation of trait $v$, see Section \ref{Section: generalisation of the proof}. 
\end{itemize}  

\begin{remark}\label{Rq Thm 1}
\begin{itemize}
\item[1.] \textbf{Mathematical construction:} For any mathematical construction other than the one given in Section \ref{Section: generalisation of the proof}, the convergences hold at least in distribution in $\mathbb{D}\left([0,T]\times[-M,M]\right)$ for Equation \eqref{Equation: conv neutral case} and in $\mathbb{D}\left([T_{1},T_{2}]\times [-M,M]\right)$ for Equations \eqref{Equation: convergence deleterious case}, \eqref{Equation: conv neutral case random scale} and \eqref{Equation: convergence deleterious case in random scale}.
\item[2.] \textbf{An additional $\log(n)$ factor:} Notice that a multiplicative factor of $\log^{\theta(v)}(n)$ is captured after time $\mathfrak{t}^{(n)}_{t(v)}$, see Equations \eqref{Equation: denominateur}, \eqref{Equation: conv neutral case}, \eqref{Equation: convergence deleterious case}, \eqref{Equation: conv neutral case random scale} and \eqref{Equation: convergence deleterious case in random scale}. Obtaining a result on the stochastic exponents (see \eqref{X}) does not capture such a factor. For instance, with the model of Figure \ref{Figure: 2 traits}, if $\lambda(1)=\lambda(0)$, Theorem \ref{Theorem: non increasin growth rate graph} gives that after time $\ell(0,1)$, $Z_{1}^{(n)}\big{(}\mathfrak{t}^{(n)}_{t}\big{)}$ behaves approximately as $\log(n) e^{\lambda(0) \mathfrak{t}^{(n)}_{t-\ell(0,1)}}$. However, what is captured with $X_{1}^{(n)}(t)$ after time $\ell(0,1)$ is asymptotically $\lambda(0)(t-\ell(0,1))$, see Theorem \ref{Theorem: general finite graph}. 
\item[3.] \textbf{Stochasticity of the limits:} The random variable $W$ is explicitly defined as the almost sure limit of the natural positive martingale associated to a specific birth and death branching process with rates $\alpha(0)$ and $\beta(0)$; see \eqref{W def general finite graph}. The martingale associated to the lineage of wild-type cells issued from the initial cell behaves similarly to the one associated to the aforementioned birth and death branching process (see Lemma \ref{Lem:control primary pop general finite graph}). Thus, $W$ quantifies the randomness of this lineage over the long time. Due to the power law mutation rates regime, mutations arise after a long time, so the stochasticity of this lineage is already captured by $W$. Notice that under Assumption \eqref{Assumption:non increasing growth rate condition}, the randomness in the first-order asymptotics of any mutant subpopulation size is described completely by $W$. This means that the stochasticity of these subpopulations is driven primarily by the randomness in the growth of the wild-type subpopulation rather than by the one of the mutational process or of any lineage of mutant cells. In particular, if the process starts with a large number of wild-type cells instead of just one, the first-order asymptotics of the size of the mutant subpopulations would be entirely deterministic.
\item[4.] \textbf{Selective cancer evolution:} It seems quite natural not to obtain such a result when considering selective mutation ($\lambda(v)>\lambda(0)$). Indeed, a selective mutation imply that any time advantage translates directly into a growth advantage. Thus, the stochasticity of the mutational process, as well as the randomness in the lineages of the mutant cells, cannot be ignored. Therefore, expecting to control the stochasticity of the mutant subpopulation solely by controlling the randomness in the wild-type subpopulation, without also accounting for the randomness in the mutational process and the mutant lineages, is vain. More precisely, using a martingale approach to derive the first-order asymptotics cannot be successful for a selective mutation. Technically, this is because the expected size of the selective mutant subpopulation is of a higher order than its typical asymptotic size. Indeed, the rare event of the initial cell mutating to the selective trait extremely quickly, an event that asymptotically vanishes, is responsible for this discrepancy between the expected value and the typical asymptotic size of the selective mutant subpopulation. Nevertheless, when examining the stochastic exponent \eqref{X}, the martingale approach allows us to obtain convergence results as given in Theorem \ref{Theorem: general finite graph}. This is because the aforementioned rare event contributes only a factor proportional to its probability to the expected value of the stochastic exponent, meaning it actually asymptotically neither contributes to the typical size nor to the expected value of the stochastic exponent of the selective mutant subpopulation. Generalization to derive the first-order asymptotics when considering selective mutations is a work in progress.
\item[5.] \textbf{Definition of neutral mutation:} In view of Theorem \ref{Theorem: non increasin growth rate graph}, the mathematical definition of neutral mutation, $\lambda(v)=\lambda(0)$, is well-understood, as opposed to the more restrictive but biologically meaningful condition of having both $\alpha(v)=\alpha(0)$ and $\beta(v)=\beta(0)$. Indeed, maintaining the same growth rate $\lambda(v)=\lambda(0),$ while changing the birth and death rates $\alpha(v)$ and $\beta(v)$ alters the distribution of any lineage of mutant cells. Consequently, one might naturally expect that this would alter the stochasticity of the mutant subpopulation size. However, this is not the case. The randomness in the first-order asymptotics is fully summed up by the random variable $W$. Thus, it is entirely consistent that, under the neutral assumption, the condition pertains only to the growth rate function rather than to the birth and death rate functions.
\item[6.] \textbf{Motivation of $d_{v}^{(n)}(t,s)$:} Considering the time scale $\mathfrak{t}^{(n)}_{t}$, the result slightly differs depending on whether the vertex is neutral or deleterious. Indeed, when looking at the asymptotic behavior for a deleterious vertex $v$, the result holds strictly after time $t(v)$, whereas, in the case of a neutral vertex, the entire trajectory from the initial time can be analyzed. Mathematically, this difference arises from the additional multiplicative factor of $\log(n)$ in the first-order asymptotics when considering a neutral mutation. This factor allows us to control the quadratic variation at time $t(v)$ for the martingale associated to the mutant subpopulation. Three distinct regimes are obtained, as indicated by \eqref{Equation: denominateur} and \eqref{Equation: conv neutral case} :
\begin{itemize}
    \item Up to time $t(v)-h_n^{-1}$: with high probability, no mutational pathway from 0 to $v$ has generated a mutant cell of trait $v$. Since $h_{n} \to \infty$ and satisfies $h_{n}=o(\log(n))$, $t(v)$ can be interpreted as the first time -when considering the time scale accelerated by $\log(n)$- at which it becomes asymptotically possible to observe the first occurrence of a mutant cell of trait $v$. This result is also true for deleterious mutations, see Lemma \ref{Lem: no mutant cell}.
    \item For $t \in \left[t(v)-h_n^{-1},t(v)\right)$: in this time interval, some mutant cells of trait $v$ are produced, but the interval's length is insufficient to achieve any power of $n$ for the size of the mutant subpopulation of trait $v$. We succeed to dominate its growth by $\psi_n \log^{\theta(v)-1}(n)$, with a well-chosen $\psi_n$. Heuristically, the total number of mutant cells of trait $v$ resulting from a mutational event up to time $t$ is of order $\Theta\big{(}\log^{\theta(v)-1}(n)\big{)}$. With the remaining time for these mutant cells' lineages to grow, we manage to control the size of the mutant subpopulation of trait $v$ by at most $\sqrt{\log(n)}\log^{\theta(v)-1}(n)$. Consequently, dividing by any function $\psi_n$ satisfying $\sqrt{\log(n)}=o(\psi_n)$ results in an asymptotic limits of $0$. This result also holds for deleterious mutations, see Lemma \ref{Lem: control before first mut}. The $\sqrt{\log(n)}$ factor in the growth control comes from a mathematical analysis using a martingale approach, particularly considering the time scale accelerated by $\log(n)$. With further refinement, we conjecture that the actual size of the mutant subpopulation at time $t(v)$ is of order $\Theta \big{(}\left(\1_{\{\lambda(0)=\lambda(v)\}}\log(\log(n))+\1_{\{\lambda(0)>\lambda(v)\}}\right)\log^{\theta(v)-1}(n)\big{)}$. 
    \item For $t \in  [t(v),\infty)$: with high probability, the number of mutant cells of trait $v$ grows exponentially at rate $\lambda(0)$. A supplementary multiplicative factor $\log^{\theta(v)}(n)$ is present due to the neutral mutations on the walks in $A(v)$. Thus, the growth scales globally as $n^{(t-t(v))}\log^{\theta(v)}(n)w_{v}(t)$.
\end{itemize}
\item[7.] \textbf{Differences between the time scales:} When comparing point (i) and (ii) of Theorem \ref{Theorem: non increasin growth rate graph}, notice that the result transitions from the deterministic time scale $\mathfrak{t}^{(n)}_{t}$ to the random time scale $\rho_{t}^{(n)}$ merely by switching $W$ to $\1_{\{W>0\}}$. This seemingly surprising fact can be explained by the essential role of $W$. As mentioned in Remark \ref{Rq Thm 1} 2., $W$ encodes the long-term stochasticity of the lineage of wild-type cells originating from the initial cell. By showing that the time scale $\mathfrak{t}^{(n)}_{t}$ serves as the correct deterministic approximation of $\rho_{t}^{(n)}$ (see Lemma \ref{Lem: deterministic approx tau general finite graph}), it follows that obtaining an asymptotic result on time scale $\mathfrak{t}^{(n)}_{t}$ also yields a result for the time scale $\rho_{t}^{(n)}$. This idea is formalized using a technique similar to that in \cite[Lemma 3]{foo2013dynamics}. The switch from $W$ to $\1_{\{W>0\}}$ in the result occurs because the time scale $\rho_{t}^{(n)}$ inherently carries the stochasticity of the random variable $W$. Consequently, the only remaining randomness that needs to be considered is the survival of the lineage from the initial cell, which is asymptotically given by $\1_{\{W>0\}}$.
\end{itemize}
\end{remark}

\subsection{Result for a general finite directed labeled graph}
\label{result general case}
This subsection does not require the non-increasing growth rate condition of Equation \eqref{Assumption:non increasing growth rate condition}. Without this assumption, a martingale approach fails to obtain the first-order asymptotics of the mutant subpopulation sizes. However, the stochastic exponents of the mutant subpopulations, as defined in \eqref{X}, can be uniformly tracked over time. In particular, we show that, under the event $\{W>0\}$, the limits are positive deterministic non-decreasing piecewise linear continuous functions. Such limits are defined via a recursive algorithm tracking their slopes over time. More precisely, we show that the slopes can only increase and take values from the growth rate function. 

In the tracking algorithm, two different kinds of updates can be made:
\begin{itemize}
    \item \textbf{Birth of a new trait:} The first update is the birth of a new trait which takes as its slope the maximum between its inner growth rate and the slope of the subpopulation that gave birth to it. In fact, it could also happen that many subpopulations give birth to it at the same time; in this case it is the maximum of their slopes that is compared to the inner growth rate of the born trait. Such a comparison on the growth rates indicates which mechanism is driving the subpopulation growth: either its inner growth if this subpopulation is selective compared to the subpopulation(s) that is/are giving birth to it, or conversely the mutational process if it is deleterious. The neutral case corresponds to a balance of these two mechanisms, as previously mentioned in Theorem \ref{Theorem: non increasin growth rate graph}.
    \item \textbf{Growth driven by another trait:} The second kind of update is when a live trait $v$ increases its slope because another live trait $u$ among its incoming neighbors, with a higher slope, has reached its typical size so that the mutational contribution from trait $u$ now drives the growth of trait $v$. Consequently trait $v$ now takes the slope of trait $u$. Again potentially many traits $u$ among the incoming neighbors of trait $v$ can reach at the same time the typical size for the mutational contribution to drive the growth of trait $v$; in this case the growth of trait $v$ is driven by the trait $u$ with the maximal slope. This kind of update encodes the possibility in the evolutionary process that the driving mechanism of a subpopulation can change over time, always triggering an increase in the actual growth of the subpopulation.
\end{itemize}
 How these two different kinds of updates happen in the tracking algorithm is made formal in the following theorem. Moreover, they can happen at the same time for different vertices. The complexity of such an algorithm comes mostly from the generality both on the growth rate function and on the trait structure. Under the non-increasing growth rate condition \eqref{Assumption:non increasing growth rate condition}, the limiting functions $(x_{v})_{x \in V}$ have an explicit form, see Corollary \ref{Cor: theorem 2.2 non-increasing growth rate condition}; this is also true when the graph structure is mono-directional, see Corollary \ref{Cor: theorem 2.2 mono directional graph}.

\begin{theorem}
\label{Theorem: general finite graph}
Let $0<T_{1}<T_{2}$. The stochastic exponents defined in \eqref{X} satisfy
\begin{align}
    \Big{(}\big{(}X^{(n)}_{v}(t)\big{)}_{v \in V}\Big{)}_{t\in [T_1,T_2]} \underset{n \to \infty}{\longrightarrow} \1_{\{W>0\}} \Big{(}\big{(}x_{v}(t)\big{)}_{v \in V}\Big{)}_{t \in [T_1,T_2]}
\end{align}
in probability in $L^{\infty}[T_1,T_2]$. For each $v \in V$, $x_v$ is a positive deterministic non-decreasing piecewise linear continuous function obtained via a recursive approach tracking its slope over time. In particular there exist $k^{*} \in \mathbb{N}$ and $0=\Delta_0<\Delta_1<\cdots<\Delta_{k^{*}}<\infty$ such that the slopes of $(x_v)_{v \in V}$ change only at the times $(\Delta_j)_{j \in \{0,\cdots,k^*\}}$. For $j \in \{0,\cdots,k^{*}\}$, at time $\Delta_{j}$ two kinds of updates in the slopes can occur: (i) either a new trait starts to grow or (ii) an already growing trait increases its slope due to a growth driven now by another more selective trait. The algorithm tracks the following quantities for all $j \in \{0,\cdots,k^{*}\}$ at time $\Delta_j$:
\begin{itemize}
    \item the set of alive traits, $A_{j}$,
    \item the set of not-yet-born traits, $U_{j}$, 
    \item the slope of $x_{v}$, $\lambda_{j}(v)$,
    \item and the set of traits whose growth is driven by trait $v$, $C_{j}(v)$. 
\end{itemize}
\textbf{Initialization:} Set $A_{0}=\{0\}$, $U_{0}=V\ \backslash \{0\}$ and for all $v \in V$ 
\begin{align}
    x_{v}(0)=0, \lambda_{0}(v)=\lambda(0)\1_{\{v=0\}}, \text{ and } C_{0}(v)=\emptyset.
\end{align}
\textbf{Induction:} Let $j \in \{0, \cdots, k^{*}-1\}$. Assume that there exist times $0=\Delta_0 < \Delta_{1}<\cdots <\Delta_{j}<\infty$ such that $\left(x_{v}\right)_{v \in V}$ are positive deterministic non-decreasing piecewise linear continuous functions defined on $[0,\Delta_j]$, where changes of slopes occur only on the discrete set $\{\Delta_1,\cdots,\Delta_{j}\}$. Also assume that there exist $\lambda_{j}(v)$, $A_j$, $U_{j}$, and $C_{j}(v)$, respectively the slope of $x_{v}$, the set of alive vertices and not-yet-born vertices, and the set of vertices whose growth is driven by $v$, everything at time $\Delta_{j}$. 

Then there exists $\Delta_{j+1} \in (\Delta_j,\infty)$ such that $\left(x_v\right)_{v \in V}$ are constructed during the time period $[\Delta_j, \Delta_{j+1}]$ according to the following. For all $v \in V$ and for all $t \geq \Delta_{j}$ let
\begin{align}
    y_{v}(t)=(t-\Delta_{j}) \lambda_{j}(v)+x_{v}(\Delta_j).
\end{align}
For all $v \in U_{j}$ define 
\begin{align}
    &\forall u\in A_{j} \text{ such that } (u,v) \in E, \delta_{u,v}:= \inf \{t \geq \Delta_j: y_{u}(t) \geq  \lambda(0) \ell(u,v)\}, \\
    &\delta_{v}:=\inf_{u \in A_{j}: (u,v)\in E}\delta_{u,v}, \\
    &\nu(v):= \{u \in A_j: (u,v) \in E \text { and } \delta_{u,v}=\delta_{v}\}.\\
\end{align}
For all $v \in A_{j}$ define 
\begin{align} 
     & B_{j}(v):= \{u \in A_j : (v,u) \in E \text{ and } \lambda_{j}(v)>\lambda_{j}(u) \}, \\ 
     &\forall u \in B_{j}(v), \delta_{v,u}:= \inf\{t \geq \Delta_j: y_{v}(t) \geq y_{u}(t)+\lambda(0)\ell(v,u)\}, \\
     &\delta_{v}:= \inf_{u \in B_{j}(v)} \delta_{v,u}, \\
     &\nu(v):= \{u \in B_{j}(v): \delta_{v,u}=\delta_{v}\}.
\end{align}
Then define $\Delta_{j+1}:=\inf_{v \in V} \delta_{v}$ and $\nu_{j+1}:=\{v \in V:\delta_{v}=\Delta_{j+1}\}$. Then proceed to the following updates: 
\begin{itemize}
    \item Let $A_{j+1}:=A_j \cup \left(\nu_{j+1}\cap U_{j}\right) \text{ and } U_{j+1}=U_{j} \backslash \left( \nu_{j+1} \cap U_{j}\right).$ Also let $\forall v \in U_{j+1},$ $\lambda_{j+1}(v)=\lambda_{j}(v)=0$, $C_{j+1}(v)=C_{j}(v)=\emptyset$.
    \item For all $v \in \nu_{j+1} \cap A_{j}$, introduce the set $\nu^{(-)}(v):=\{u \in \nu(v): \exists w \in \nu_{j+1} \cap A_{j}, \lambda_{j}(w)>\lambda_{j}(v), \text{ and } u \in \nu(w) \}$. 
    
    Then let $C_{j+1}(v):=C_{j}(v) \cup\bigcup_{u \in \nu(v) \backslash \nu^{(-)}(v)} \left(\{u\} \cup C_{j}(u)\right).$ For all $u \in \nu(v) \backslash \nu^{(-)}(v)$ and $w \in C_{j}(u)$, $\lambda_{j+1}(u)=\lambda_{j+1}(w)=\lambda_{j}(v)$.
    \item For all $v \in A_j$ whose slope has not been updated yet, let $\lambda_{j+1}(v)=\lambda_{j}(v)$. And for all $v \in A_j$ whose set $C_{j}(v)$ has not been updated yet, let $C_{j+1}(v):=C_{j}(v)$.
    \item For all $v \in \nu_{j+1} \cap U_j$, let $\lambda_{j+1}(v):=\max \left(\lambda(v), \max_{u \in \nu(v)}\lambda_{j+1}(u)\right)$, and \\
    $C_{j+1}(v)=C_{j}(v)=\emptyset$. If $\lambda_{j+1}(v)\geq \lambda(v)$, introduce the following set $\nu^{+}(v):=\{u \in \nu(v): \lambda_{j+1}(u)=\max_{w \in \nu(v)}\lambda_{j+1}(w)\}$, and for all $u \in \nu^{+}(v)$,  $C_{j+1}(u):=C_{j+1}(u) \cup \{v\}$.
\end{itemize}
For any mathematical construction other than the one given in Section \ref{Section: generalisation of the proof} (see \eqref{Equ6}, \eqref{Equ7}, \eqref{ModelConstructionK}, \eqref{ModelConstuctionH} and \eqref{ModelConstructionZ}), the convergences are at least in distribution in $\mathbb{D}\left([T_1,T_2]\right)$ . 
\end{theorem}
The proof of Theorem \ref{Theorem: general finite graph} is a minor adaptation of the proofs found in \cite{durrett2011traveling}. Specifically, by adapting the arguments from \cite[Propositions 2 and 4]{durrett2011traveling} to the context of the present model, Theorem \ref{Theorem: general finite graph} follows. For this reason, and in the interest of brevity, we do not provide an explicit proof of Theorem \ref{Theorem: general finite graph}. 

The only notable difference is that the process does not start from a macroscopic state. However, it can be easily shown that, conditioned on $\{W=0\}$, no mutant cells are generated asymptotically, since with high probability the wild-type subpopulation can't survive in the $\log(n)$-accelerated time scale. Additionally, conditioned on $\{W>0\}$, the first phase, corresponding to the growth of the wild-type subpopulation leading to the macroscopic state that allows for the generation of the first mutant cell, is straightforward to capture.

When considering a(n) (infinite) mono-directional graph, the structure of such a graph is sufficiently simple to allow for an explicit form of the limiting functions $(x_v)_{v \in V}$, see the next corollary. In particular, there is only one possible slope change that can happen at a time. More specifically, when a not-yet-born trait becomes alive due to the previous trait reaching the typical size allowing for mutations. When this happens, the new born trait takes the slope the maximum between its inner growth rate or the current slope of the previous trait (as mentioned in point (i) of the heuristics preceding Theorem \ref{Theorem: general finite graph}). Any alive trait cannot update its slope because no backward mutation is permitted with this graph structure. Moreover, only a single trait becomes alive at each time, due to the scaling labels $\ell(i,i+1)$ being positive.

\begin{corollary}[Theorem \ref{Theorem: general finite graph} applied to a mono-directional graph]\label{Cor: theorem 2.2 mono directional graph} 
Assume the graph is infinite and mono-directional, i.e. $(V,E)=(\mathbb{N}_{0},\{(i,i+1), i \in \mathbb{N}_{0}\})$ and that $\ell^{*}:=\inf \{\ell(i,i+1), i \in \mathbb{N}_{0}\}>0$. Then the limiting functions $(x_i)_{i \in \mathbb{N}_0}$ of Theorem \ref{Theorem: general finite graph} have the following simplified form:
\begin{align}
    \forall t \in \mathbb{R}^{+}, x_{i}(t):=\lambda_{\max}(i) (t-\widetilde{t}(i))_{+},
\end{align}
where $\lambda_{\max}(i)=\max_{j \in \{0,\cdots,i\}}\lambda(j)$ and $\widetilde{t}(i):=\sum_{j=0}^{i-1}  \frac{\ell(j,j+1) \lambda(0)}{\lambda_{\max}(j)}$.
\end{corollary}

 \begin{remark}
     Using the previous corollary, the limits $(x_v)_{v \in V}$ defined in Theorem \ref{Theorem: general finite graph} can be rewritten by using the decomposition via walks. More specifically, let $v \in V$, then for any walk $\gamma \in P(v)$ define $x_{\gamma}$ as the limit obtained by applying the previous corollary to the mono-directional graph indexed by this walk $\gamma$. Then we have $x_{v}=\max_{\gamma \in P(v)} x_{\gamma}$. The maximum is well-defined because for all $t \in \mathbb{R}^{+}$ the set $\{\gamma \in P(v): x_{\gamma}(t)>0\}$ is finite. 
 \end{remark}

Theorem \ref{Theorem: general finite graph} is more general than Theorem \ref{Theorem: non increasin growth rate graph} in the sense that there is no assumption on the growth rate function, but it is a less refined result. In Remark \ref{Rq Thm 1} 1. we made explicit one contribution of Theorem \ref{Theorem: non increasin growth rate graph} about capturing a multiplicative factor of $\log(n)$ using the example of Figure \ref{Figure: 2 traits}. Next we are going to do a full comparison of Theorem \ref{Theorem: non increasin growth rate graph} and \ref{Theorem: general finite graph} on the example of Figure \ref{Figure: Heuristics2}. 

\textbf{Comparison between Theorems \ref{Theorem: non increasin growth rate graph} and \ref{Theorem: general finite graph}: }
The asymptotic function $x$ obtained through Theorem \ref{Theorem: general finite graph} for the plain purple trait is $x(t)=\1_{\{t \geq 4\}}\lambda(0)(t-4).$ In the caption of Figure \ref{Figure: Heuristics2}, it is already made explicit that only the plain green walk will contribute to the size order of the plain purple mutant subpopulation. If one denotes respectively by $1$, $2$ and $3$ the vertices on the plain green walk such that this walk is exactly $(0,1,2,3)$, where $3$ is the plain purple vertex, the asymptotic limits for vertex $3$, captured by Theorem \ref{Theorem: non increasin growth rate graph}, is for all $t \geq 4,$ 

\begin{align}
    &\frac{2\alpha(0) \mu(0,1)}{\lambda(0)}\cdot \frac{2 \alpha(1) \mu(1,2)}{\lambda(0)}\cdot \frac{2 \alpha(2) \mu(2,3)}{\lambda(0)-\lambda(3)}W \int_{3}^{t} \left(\int_{1}^{u} ds\right)du \cdot n^{t-4}\log^{2}(n)\\
    &\hspace{1.13cm}=\left(\frac{t^2}{2}-t-\frac{3}{2}\right)\frac{8\alpha(0) \alpha(1) \alpha(2) \mu(0,1)\mu(1,2)\mu(2,3)}{\lambda^{2}(0) \left(\lambda(0)-\lambda(3)\right)}W n^{t-4}\log^{2}(n).
\end{align}
In particular, Theorem \ref{Theorem: general finite graph} captures only the power of $n$ which is $t-4$ whereas Theorem \ref{Theorem: non increasin growth rate graph} captures the stochasticity $W$, a supplementary scaling factor $\log^{2}(n)$, a time polynomial $\frac{t^2}{2}-t-\frac{3}{2}$ and also a constant depending only on the parameters of the visited vertices $\frac{8\alpha(0) \alpha(1) \alpha(2)\mu(0,1)\mu(1,2)\mu(2,3)}{\lambda^{2}(0) \left(\lambda(0)-\lambda(3)\right)}$. To our knowledge, this is the first time that this level of refinement has been captured under the power law mutation rates limit.

Now we make explicit the form of the limiting functions $(x_{v})_{v \in V}$ in the special case where we assume the non-increasing growth rate condition. Under this condition, the limiting functions $(x_v)_{v \in V}$ take a very simple form. The only quantity one has to consider is the time $t(v)$ at which trait $v$ becomes alive, where $t(v)$ is defined in Definition \ref{Def: Admissible path}. Then after this time, trait $v$ grows at speed $\lambda(0)$ due to the non-increasing growth rate condition. This is made formal in the next corollary.
\begin{corollary}[Theorem \ref{Theorem: general finite graph} applied with the non-increasing growth rate condition of \eqref{Assumption:non increasing growth rate condition}]\label{Cor: theorem 2.2 non-increasing growth rate condition}  
    Assume the non-increasing growth rate condition of \eqref{Assumption:non increasing growth rate condition}. Then the limiting functions $(x_{v})_{v \in V}$ of Theorem \ref{Theorem: general finite graph} have the following simplified form: 
    \begin{align}
        \forall t \in \mathbb{R^{+}}, x_{v}(t)=\lambda(0) \left(t-t(v)\right)_{+},
    \end{align}
    where $\forall x \in \mathbb{R}, x_{+}:=x\1_{\{x \in \mathbb{R}^{+}\}}$.
\end{corollary}
 
\section{First-order asymptotics of the sizes of the mutant subpopulations for an infinite mono-directional graph}
\label{Section: proof mono directional graph}
In this section, we consider the model described in Section \ref{Introduction} within a specific infinite mono-directional graph setting:
\begin{align}
\label{infinite modo-directional graph}
    (V,E)=\left(\mathbb{N}_{0}, \{(i,i+1), i \in \mathbb{N}_{0} \} \right).
\end{align}
Studying this special case will enable us to address cycles, particularly those generated by backward mutations, in the more general finite graph scenario. We assume the non-increasing growth rate condition given in \eqref{Assumption:non increasing growth rate condition}. For simplicity of notation, we introduce the following new notations for all $i \in \mathbb{N}_{0}$
\begin{align}
    \mu_{i}^{(n)}:=\mu^{(n)}(i,i+1) \text{ and } \ell(i):=\ell(i,i+1).
\end{align}
In other words, the mutation regime is
\begin{align}
\label{Mutation regime}
    \forall i \in \mathbb{N}_{0},n^{\ell(i)} \mu_{i}^{(n)} \underset{n \to \infty}{\longrightarrow}\mu_{i}.
\end{align}
Assume that $\ell^{*}:=\inf \{\ell(i): i \in \mathbb{N}_0\}>0.$ For all $i \in \mathbb{N}_{0}$, denote by $\alpha_{i}$, $\beta_{i}$ and $\lambda_{i}$ the division, death and growth rates associated to trait $i$ instead of $\alpha(i), \beta(i)$ and $\lambda(i)$. With this setting, three different scenarios can occur during a division event of a cell of trait $i \in \mathbb{N}_{0}$:
\begin{itemize}
    \item both daughter cells keep the trait $i$ of their mother cell, with probability $\big{(}1-\mu_{i}^{(n)}\big{)}^{2}$,
    \item exactly one daughter cell mutates to the next trait $i+1$ when the second daughter cell keeps the trait $i$ of its mother cell, with probability $2\mu_{i}^{(n)}\big{(}1-\mu_{i}^{(n)}\big{)}$,
    \item both daughter cells mutate to the next trait $i+1$, with probability $\big{(}\mu_{i}^{(n)}\big{)}^{2}$.
\end{itemize}
This model is graphically represented in Figure~\ref{Figure: model}.
\begin{figure}
\centering
\includegraphics[width=0.95 \linewidth]{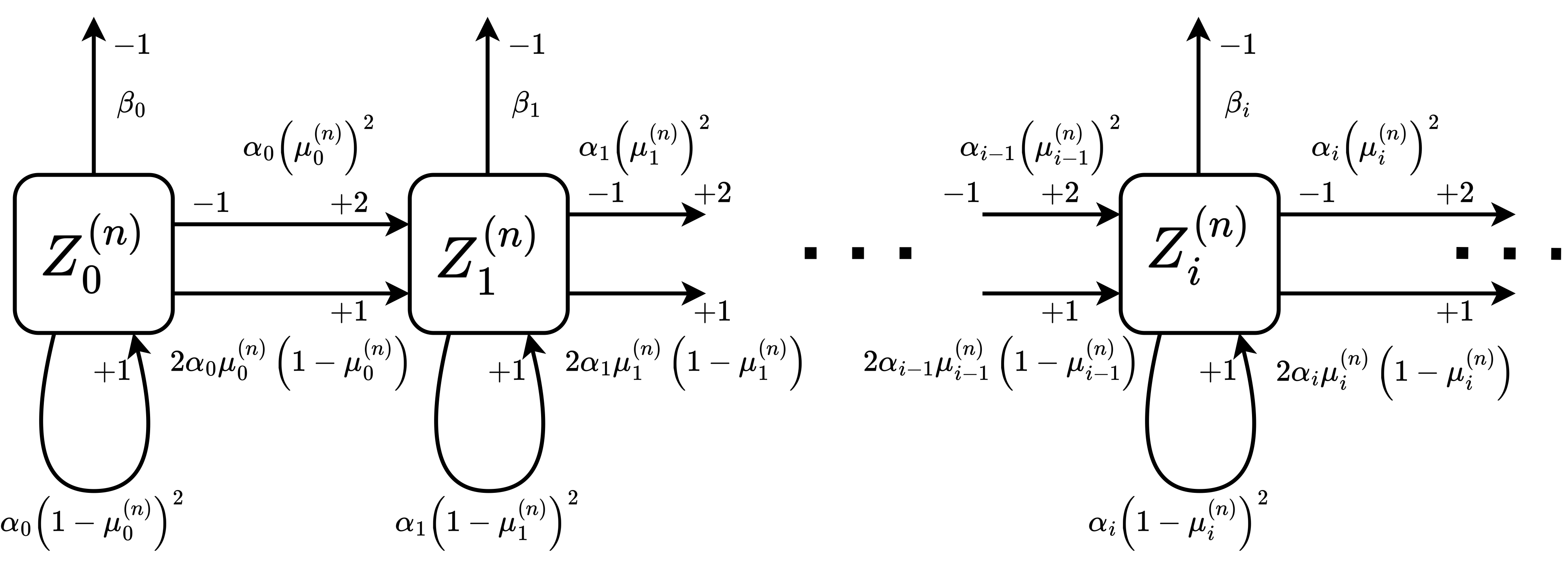}
\caption{Dynamical representation of the infinite mono-directional graph}
\label{Figure: model}
\end{figure}
In particular, any lineage of a cell of trait $i$ follows a birth-death branching process with birth rate $\alpha_{i}\big{(}1-\mu_{i}^{(n)} \big{)}^{2}$ and death rate $\beta_{i}+\alpha_{i}^{(n)}\big{(}\mu_{i}^{(n)}\big{)}^{2}$. Thus, we introduce the birth, death and growth rates of any lineage of a cell with trait $i$ as follows
\begin{align}
    \label{birth, death, growth rates}
    &\alpha_{i}^{(n)}:= \alpha_{i}\big{(}1-\mu_{i}^{(n)} \big{)}^{2}, \beta_{i}^{(n)}:=\beta_{i}+\alpha_{i}^{(n)}\big{(}\mu_{i}^{(n)}\big{)}^{2} \text{ and } \lambda_{i}^{(n)}:=\alpha_{i}^{(n)}-\beta_{i}^{(n)}=\lambda_{i}-2\alpha_{i}\mu_{i}^{(n)}.
\end{align}
Compared to the general finite graph, for this mono-directional graph, there is only one path from trait $0$ to any trait $i \in \mathbb{N}$, implying that 
\begin{align}
\label{t(i) and theta(i)}
    &t(i)=\sum_{i=0}^{i-1} \ell(i) \text{ and }\theta(i)=\vert\{j\in \{1, \cdots,i\}: \lambda_{j}=\lambda_{0} \}\vert.
\end{align}
Let $w_i:=w_{(0,1,\cdots,i)}$ denote the weight on the path $(0,\cdots,i)$. The sequence $\big{(}\big{(}Z^{(n)}_{i}\big{)}_{i \in \mathbb{N}_{0}}\big{)}_{n \in \mathbb{N}}$ is mathematically constructed using independent Poisson Point Measures (PPMs). Let $Q_{0}^{b}(ds,d\theta)$, $Q_{0}^{d}(ds,d\theta)$, $\left(Q_{i}(ds,d\theta) \right)_{i\in \mathbb{N}}$, $\left(N_{i}(ds,d\theta) \right)_{i \in \mathbb{N}_{0}}$, and $\left(Q_{i}^{m}(ds,d\theta) \right)_{i \in \mathbb{N}_{0}}$ be independent PPMs with intensity $dsd\theta$. The subpopulation of wild-type cells is
\begin{align}
\label{Equation: wild type pop}
    Z_{0}^{(n)}(t)&:=1+\int_{0}^{t}\int_{\mathbb{R}^{+}} \1_{\left\{\theta \leq \alpha_0^{(n)}Z_{0}^{(n)}(s^{-})\right\}}Q_{0}^{b}(ds,d\theta)\\
    &-\int_{0}^{t}\int_{\mathbb{R}^{+}}\1_{\left\{\theta \leq \beta_0Z_{0}^{(n)}(s^{-}) \right\}}Q_{0}^{d}(ds,d\theta)-H_{0}^{(n)}(t)
\end{align}
and for all $i\in \mathbb{N}$ the mutant subpopulation of trait $i$ is
\begin{align}
\label{Equation: def mutant pop}
    Z_{i}^{(n)}(t)&:=\int_{0}^{t}\int_{\mathbb{R}^{+}} \Big{(} \1_{\left\{\theta \leq \alpha_i^{(n)}Z_{i}^{(n)}(s^{-})\right\}}-\1_{\left\{\alpha_i^{(n)}Z_{i}^{(n)}(s^{-}) \leq \theta \leq \left(\alpha_i^{(n)}+\beta_i\right)Z_{i}^{(n)}(s^{-}) \right\}}\Big{)}Q_{i}(ds,d\theta) \\ 
    &+K_{i-1}^{(n)}(t)+2H_{i-1}^{(n)}(t)-H_{i}^{(n)}(t), \\
\end{align}
where for all $i \in \mathbb{N}_{0}$
\begin{align}
\label{Equation: def mutation processes}
    K_{i}^{(n)}(t):=\int_{0}^{t}\int_{\mathbb{R}^{+}}\1_{\left\{\theta \leq 2\alpha_{i}\mu_{i}^{(n)}\left(1-\mu_{i}^{(n)}\right)Z_{i}^{(n)}(s^{-})\right\}}N_{i}(ds,d\theta)
\end{align}
and 
\begin{align}
    H_{i}^{(n)}(t):=\int_{0}^{t}\int_{\mathbb{R}^{+}} \1_{\left\{ \theta \leq \alpha_{i} \left( \mu_{i}^{(n)}\right)^{2}Z_{i}^{(n)}(s^{-})\right\}} Q_{i}^{m}(ds,d\theta).
\end{align}
The processes $\big{(}K_{i}^{(n)}(t)\big{)}_{t\in\mathbb{R}^{+}}$ and $\big{(}H_{i}^{(n)}(t)\big{)}_{t\in\mathbb{R}^{+}}$ count the number of mutations up to time $t$ in the subpopulation of trait $i$ that result in exactly one and exactly two mutated daughter cells of trait $i+1$. 

Let $(Z_{0}(t))_{t \in \mathbb{R}^{+}}$ be the birth-death branching process with birth and death rates $\alpha_{0}$ and $\beta_{0}$ constructed in the following way 
\begin{align}
\label{def Z0}
    \hspace{1cm}Z_{0}(t):=1+\int_{0}^{t}\int_{\mathbb{R}^{+}} \1_{\left\{\theta \leq \alpha_{0}Z_{0}(s^{-})\right\}}Q_{0}^{b}(ds,d\theta)-\int_{0}^{t}\int_{\mathbb{R}^{+}} \1_{\left\{ \theta \leq \beta_{0}Z_{0}(s^{-})\right\}} Q_{0}^{d}(ds,d\theta).
\end{align}
Notice that with this construction, the following monotone coupling immediately holds:
\begin{align}
\label{eq:monotonecoupling}
    \forall t \geq 0, Z_{0}^{(n)}(t) \leq Z_{0}(t) \  a.s.
\end{align}
Denote by 
\begin{align}
\label{W def}
    W:=\lim_{t \to \infty} e^{-\lambda_{0} t}Z_{0}(t)
\end{align}
the almost sure limit of the positive martingale $\left( e^{-\lambda_{0} t}Z_{0}(t)\right)_{t \in \mathbb{R}^{+}}$, whose law is
\begin{align}
\label{Eq:distribution W}
    W\overset{law}{=}Ber\Big{(}\frac{\lambda_{0}}{\alpha_{0}}\Big{)}\otimes Exp \Big{(} \frac{\lambda_{0}}{\alpha_{0}}\Big{)},
\end{align}
see \cite[Section 1.1]{durrett2010evolution} or  \cite[Theorem 1]{durrett2015branching}.

\subsection{The wild-type subpopulation dynamics} \label{Sec:proof lemma}
Using the same PPMs $Q_{0}^{b}$ and $Q_{0}^{d}$ in the construction of $\big{(}Z_{0}^{(n)}\big{)}_{n \in \mathbb{N}}$ and $Z_0$ (see Equations \eqref{Equation: wild type pop} and \eqref{def Z0}) allows us to control the size dynamics of the previous sequence over time by comparing it with the size of $Z_0$. More precisely, we show that the natural martingale associated with $Z_{0}^{(n)}$ can be compared to the natural one of $Z_0$. This comparison follows from the fact that $\big{(}\alpha_{0}^{(n)}, \beta^{(n)}_{0}\big{)} \to(\alpha_0,\beta_0)$ as $n \to \infty.$ The control is obtained along the entire trajectory and in probability. The rate of convergence is quantified to be at most of order $\mathcal{O}\big{(}\mu^{(n)}_{0}\big{)}$.
\begin{lemma}
\label{Lem:control primary pop}
There exist $C(\alpha_{0},\lambda_{0})>0$ and $N \in \mathbb{N}$ such that for all $\varepsilon>0$ and $n \geq N$, 
\begin{align}
    \mathbb{P}\Big{(} \sup_{t \in \mathbb{R}^{+}} \Big{\vert} e^{-\lambda_{0}t}Z_{0}(t)-e^{-\lambda_{0}^{(n)}t}Z_{0}^{(n)}(t)\Big{\vert} \geq \varepsilon\Big{)} \leq \frac{ C(\alpha_{0},\lambda_{0})}{\varepsilon^{2}} \mu_{0}^{(n)} \underset{n \to \infty}{\longrightarrow}0.
\end{align}
\end{lemma}

\begin{proof}[Proof of Lemma \ref{Lem:control primary pop}]
Let the filtration $(\mathcal{F}_{t})_{t \geq 0}$ defined for all $t \geq 0$ as $$\mathcal{F}_{t}:=\sigma(Q_{0}^{b}((0,s] \times A)), Q_{0}^{d}((0,s]\times A), s \leq t, A \in \mathcal{B}(\mathbb{R}^{+})).$$ Notice that $\big{(} e^{-\lambda_{0}t}Z_{0}(t)-e^{-\lambda_{0}^{(n)}t}Z_{0}^{(n)}(t)\big{)}_{t \in \mathbb{R}^{+}}$ is a martingale, with respect to $(\mathcal{F}_{t})_{t \in \mathbb{R}^{+}}$, as the difference between the two martingales $\left( e^{-\lambda_{0}t}Z_{0}(t)\right)_{t \in \mathbb{R}^{+}}$ and $\big{(} e^{-\lambda_{0}^{(n)}t}Z_{0}^{(n)}(t)\big{)}_{t \in \mathbb{R}^{+}}$. Let $(f(m))_{m \in \mathbb{N}}$ be a non decreasing sequence satisfying $f(m) \underset{m \to \infty}{\to}\infty$. Using Doob's Inequality in $L^{2}$ (see \cite[Proposition 3.15]{le2016brownian}) we derive 
\begin{align} \label{eq:primary pop}
    &\mathbb{P}\Big{(} \sup_{t \in [0,f(m)]} \Big{\vert} e^{-\lambda_{0} t}Z_{0}(t)-e^{-\lambda_{0}^{(n)}t}Z_{0}^{(n)}(t)\Big{\vert} \geq \varepsilon\Big{)}\leq \frac{4}{\varepsilon^{2}}\E\Big{[} e^{-2\lambda_{0}f(m)}Z_{0}(f(m))^{2}\\
    &\hspace{1cm}+e^{-2\lambda_{0}^{(n)}f(m)}Z_{0}^{(n)}(f(m))^{2}-2e^{-(\lambda_{0}+\lambda_{0}^{(n)})f(m)}Z_{0}(f(m))Z_{0}^{(n)}(f(m))\Big{]}.
\end{align}
Ito's formula and Equation \eqref{eq:monotonecoupling} give $$\E \big{[}Z_{0}(t)Z_{0}^{(n)}(t)\big{]}=1+\int_{0}^{t}\big{(} \lambda_{0}+\lambda_{0}^{(n)}\big{)}\E \big{[}Z_{0}(s)Z_{0}^{(n)}(s)\big{]}ds+\int_{0}^{t}\big{(} \alpha_{0}^{(n)}+\beta_{0}\big{)}\E\big{[}Z_{0}^{(n)}(s)\big{]}ds.$$ By solving this equation, we obtain for all $t\geq0$
\begin{align}
\label{second cross moment birth death process}
    &\E \big{[}Z_{0}(t)Z_{0}^{(n)}(t)\big{]}=\frac{\alpha_{0}+\alpha_{0}^{(n)}}{\lambda_{0}}e^{\left(\lambda_{0}+\lambda_{0}^{(n)}\right)t}-\frac{\alpha_{0}^{(n)}+\beta_{0}}{\lambda_{0}}e^{\lambda_{0}^{(n)}t}.
\end{align}
Similarly we have 
\begin{align} \label{second moment birth death process}
        &\mathbb{E}\big{[}\left(Z_{0}(t)\right)^{2}\big{]}=\frac{2\alpha_{0}}{\lambda_{0}}e^{2\lambda_{0}t}-\frac{\alpha_{0}+\beta_{0}}{\lambda_{0}}e^{\lambda_{0}t}\leq \frac{2\alpha_{0}}{\lambda_{0}}e^{2\lambda_{0}t},\\
    &\mathbb{E}\big{[} \big{(}Z_{0}^{(n)}(t)\big{)}^{2}\big{]}=\frac{2\alpha_{0}^{(n)}}{\lambda_{0}^{(n)}}e^{2\lambda_{0}^{(n)}t}-\frac{\alpha_{0}^{(n)}+\beta_{0}^{(n)}}{\lambda_{0}^{(n)}}e^{\lambda_{0}^{(n)}t}\leq \dfrac{2\alpha_{0}^{(n)}}{\lambda_{0}^{(n)}}e^{2\lambda_{0}^{(n)}t}.
\end{align}
Consequently, combining \eqref{eq:primary pop}, \eqref{second cross moment birth death process} and \eqref{second moment birth death process} yields
\begin{align}
    &\mathbb{P}\Big{(} \sup_{t \in [0,f(m)]} \Big{\vert} e^{-\lambda_{0} t}Z_{0}(t)-e^{-\lambda_{0}^{(n)}t}Z_{0}^{(n)}(t)\Big{\vert} \geq \varepsilon\Big{)} \\
    &\hspace{3cm}\leq \frac{4}{\varepsilon^{2}}\Big{(} \frac{2\alpha_{0}}{\lambda_{0}}+\frac{2\alpha_{0}^{(n)}}{\lambda_{0}^{(n)}}-2\frac{\alpha_{0}+\alpha_{0}^{(n)}}{\lambda_{0}}+2\frac{\alpha_{0}^{(n)}+\beta_{0}}{\lambda_{0}}e^{-\lambda_{0}f(m)}\Big{)}.
\end{align}
The event $\big{\{}\sup_{t \in [0,f(m)]} \big{\vert} e^{-\lambda_{0} t}Z_{0}(t)-e^{-\lambda_{0}^{(n)}t}Z_{0}^{(n)}(t)\big{\vert} \geq \varepsilon\big{\}}$ is increasing with respect to the parameter $m$. Therefore, taking the limit as $m \to \infty$ and applying the monotonicity of the measure, it follows that
\begin{align}
    \mathbb{P}\Big{(} \sup_{t \in \mathbb{R}^{+}} \Big{\vert} e^{-\lambda_{0} t}Z_{0}(t)-e^{-\lambda_{0}^{(n)}t}Z_{0}^{(n)}(t)\Big{\vert} \geq \varepsilon\Big{)} \leq \frac{4}{\varepsilon^{2}}\Big{(} \frac{2\alpha_{0}}{\lambda_{0}}+\frac{2\alpha_{0}^{(n)}}{\lambda_{0}^{(n)}}-2\frac{\alpha_0+\alpha_0^{(n)}}{\lambda_0}\Big{)}.
\end{align}
Recalling that $\lambda_{0}^{(n)}=\lambda_{0}-2\alpha_{0}\mu_{0}^{(n)}$ it easily follows that $\frac{2\alpha_{0}^{(n)}}{\lambda_{0}^{(n)}}=\frac{2\alpha_{0}}{\lambda_{0}}+\frac{4\beta_{0}\alpha_{0}}{\lambda_{0}^{2}}\mu_{0}^{(n)}+\mathcal{O}\big{(}\big{(}\mu_{0}^{(n)}\big{)}^{2}\big{)}$ as well as $2\frac{\alpha_{0}+\alpha_{0}^{(n)}}{\lambda_{0}}=\frac{4\alpha_{0}}{\lambda_{0}}-\frac{4\alpha_{0}}{\lambda_{0}}\mu_{0}^{(n)}+\mathcal{O}\big{(}\big{(}\mu_{0}^{(n)}\big{)}^{2}\big{)}.$ Finally we have
\begin{align}
    \mathbb{P}\Big{(} \sup_{t \in \mathbb{R}^{+}}  \Big{\vert} e^{-\lambda_{0} t}Z_{0}(t)-e^{-\lambda_{0}^{(n)}t}Z_{0}^{(n)}(t)\Big{\vert} \geq \varepsilon\Big{)} &\leq \frac{4}{\varepsilon^{2}}\Big{(}\frac{4\beta_{0}\alpha_{0}}{\lambda_{0}^{2}}+\frac{4\alpha_{0}}{\lambda_{0}}\Big{)}\mu_{0}^{(n)}+\mathcal{O}\big{(}\big{(}\mu_{0}^{(n)}\big{)}^{2}\big{)} \\
    &=\frac{16\alpha_{0}^{2}}{\varepsilon^{2}\lambda_{0}^{2}}\mu_{0}^{(n)}+\mathcal{O}\big{(}\big{(}\mu_{0}^{(n)}\big{)}^{2}\big{)},
\end{align}
which concludes the proof.
\end{proof}
The next lemma provides an asymptotic comparison between the stopping times $\eta_{t}^{(n)}$, at which the wild-type subpopulation reaches the size $n^{t}$, and the deterministic times $\mathfrak{t}_{t}^{(n)}$. This asymptotic comparison is given in probability and is conditioned on $\{W>0\}$. It clarifies why these deterministic times are the natural candidates for studying the asymptotic behavior of the mutant subpopulations at the corresponding stopping times. The result is obtained uniformly over time intervals whose lengths tend to infinity, but not too quickly. 
\begin{lemma}
\label{convergence temps arret}
For all $\varepsilon>0$, $(T_1,T_2) \in \mathbb{R}^{+}$ and $\varphi_n$ such that  $\log(n)=o(\varphi_n)$ and $\varphi_n=o\left(n^{\ell(0)}\right)$, we have
\begin{equation}
    \mathbb{P} \Big{(}\sup_{t \in \left[T_{1},T_{2}\frac{\varphi_n}{\log(n)}\right]} \Big{\vert}\eta_{t}^{(n)}-\Big{(}\mathfrak{t}^{(n)}_{t}-\frac{\log(W)}{\lambda_0} \Big{)}\Big{\vert} \geq \varepsilon \Big{\vert} W>0\Big{)}\underset{n \to \infty}{\longrightarrow} 0.
\end{equation}
\end{lemma}

\begin{proof}[Proof of Lemma \ref{convergence temps arret}]
Let $\varepsilon>0$ and for all $n \in \mathbb{N}$ introduce  the event 
\begin{align}
    A^{(n)}:=\Big\{ \sup_{t \in \left[T_{1},T_{2}\frac{\varphi_n}{\log(n)}\right]}\Big\vert\eta_{t}^{(n)}-\Big(\mathfrak{t}^{(n)}_{t}-\frac{\log(W)}{\lambda_0} \Big)\Big\vert \geq \varepsilon\Big\}.
\end{align}

\textbf{Step 1:} We begin by showing that for all $0<\delta_1<\delta_2$ 
\begin{align}
\label{deterministic approx random time with W right}
    \mP \left(A^{(n)} \cap \{\delta_1<W<\delta_2\} \right)\underset{n \to \infty}{\longrightarrow}0.
\end{align}
Let $\nu>0$ and $\Tilde{\varepsilon}<\frac{\delta_1}{2}$. Firstly, since $e^{-\lambda_{0} t}Z_{0}(t) \underset{t \to \infty}{\to}W$ almost surely, it immediately follows that $Y(t):=\sup_{s \in [t,\infty]} \big{|}e^{-\lambda_{0} s}Z_{0}(s) -W\big{|} \underset{t \to \infty}{\longrightarrow}0$ almost surely and as a consequence in probability. Thus, introducing the event $B_{t}:=\{Y(t) \leq \Tilde{\varepsilon} \}$ for all $t>0$, there exists $t_{1}>0$ such that for all $t \geq t_1$,  $\mP \left( B_{t}\right) \geq 1-\frac{\nu}{3}.$ Secondly, using Lemma \ref{Lem:control primary pop}, there exists $n_{1} \in \mathbb{N}$ such that for all $n \geq n_{1}$, $\mP \left( C^{(n)}\right) \geq 1-\frac{\nu}{3}$ where $C^{(n)}:=\big\{ \sup_{t \in \mathbb{R}^{+}} \big{\vert} e^{-\lambda_{0} t}Z_{0}(t)-e^{-\lambda_{0}^{(n)}t}Z_{0}^{(n)}(t)\big{\vert} \leq \widetilde{\varepsilon}\big\}.$ Combining these two facts, we obtain the following inequality for all $n\geq n_{1}$
\begin{equation} \label{lab: E0}
    \mP \left( A^{(n)}\cap \{\delta_1<W<\delta_2\}\right) \leq \mP \left( A^{(n)}\cap \{\delta_1<W<\delta_2\}\cap B_{t_1}\cap C^{(n)} \right) +\frac{2\nu}{3}.
\end{equation}
It remains to show that $\mP \left( A^{(n)}\cap \{\delta_1<W<\delta_2\}\cap B_{t_1}\cap C^{(n)} \right)\leq \frac{\nu}{3}$ for sufficiently large $n$. \\
Under the event $B_{t_1}$ we have 
\begin{equation}
    \forall s \geq t_1, \left(W -\widetilde{\varepsilon}\right)e^{\lambda_{0}s} \leq Z_{0}(s) \leq \left(W +\widetilde{\varepsilon}\right)e^{\lambda_{0}s}.
\end{equation}
Given that $\lambda_{0}^{(n)}\leq \lambda_{0}$, we obtain that under the event $C^{(n)}$, for all $n \geq n_1$
\begin{equation}
    \forall s \in \mathbb{R}^{+}, \Big(e^{-\lambda_{0}s}Z_{0}(s)-\widetilde{\varepsilon}\Big)e^{\lambda_{0}^{(n)}s} \leq Z_{0}^{(n)}(s) \leq \Big(e^{-\lambda_{0}s}Z_{0}(s)+\widetilde{\varepsilon}\Big)e^{\lambda_{0}^{(n)}s} \leq Z_{0}(s)+\widetilde{\varepsilon}e^{\lambda_{0}s}.
\end{equation}
Combining the two previous inequalities, it follows that under $\{\delta_1<W<\delta_2\} \cap B_{t_1}\cap C^{(n)}$, we have  
\begin{equation}
    \forall s \geq t_1, \forall n \geq n_1,  \left( W-2\widetilde{\varepsilon}\right) e^{\lambda_{0}^{(n)}s}\leq Z_{0}^{(n)}(s) \leq \left(W+2\widetilde{\varepsilon}\right)e^{\lambda_{0}s} \leq \left(\delta_2+2\widetilde{\varepsilon}\right)e^{\lambda_{0}s}.
\end{equation}
Notice that, by definition of $\widetilde{\varepsilon}$, we have $W-2\widetilde{\varepsilon}>0$ under the event $\{\delta_1 <W\}$. Now, we introduce the following quantities, which almost surely increase with time:
\begin{align}
        &\underline{T}^{(n)}_{\delta_2,t}:=\inf \{ s>0: (\delta_2+2 \widetilde{\varepsilon})e^{\lambda_{0} s} \geq n^t \}=\frac{1}{\lambda_0} \left(t \log(n)-\log(\delta_2+2\widetilde{\varepsilon})\right), \\
        &\underline{T}^{(n)}_{t}:=\inf \{ s>0: (W+2 \widetilde{\varepsilon})e^{\lambda_{0} s} \geq n^t \}=\frac{1}{\lambda_0} \left(t \log(n)-\log(W+2\widetilde{\varepsilon})\right), \\ 
        &\overline{T}^{(n)}_{t}:=\inf \{ s>0: (W-2 \widetilde{\varepsilon})e^{\lambda_{0}^{(n)} s} \geq n^t \}=\frac{1}{\lambda_0^{(n)}} \left(t \log(n)-\log(W-2\widetilde{\varepsilon})\right).
\end{align}
We have that there exists $n_{2} \in \mathbb{N}$ such that for all $n \geq n_{2}$, $t_1 \leq \underline{T}^{(n)}_{\delta_2,T_1}.$ Moreover, under the event $\{\delta_1<W<\delta_2\}\cap B_{t_1}\cap C^{(n)}$, we have for all $n \geq \max(n_1,n_2)$ and for all $t \in \big[T_1,T_{2}\frac{\varphi_n}{\log(n)}\big]$, 
\begin{align}
\underline{T}^{(n)}_{\delta_2,T_1} \leq \underline{T}^{(n)}_{\delta_2,t} \leq \underline{T}^{(n)}_{t} \leq \eta^{(n)}_{t}\leq \overline{T}^{(n)}_{t}.
\end{align}
Using that $\lambda_0/\lambda_{0}^{(n)}=1/\big(1-2\alpha_{0}\mu_{0}^{(n)}/\lambda_{0}\big)$, and from the previous equation, we derive that for all $ t \in \big[T_1,T_{2}\frac{\varphi_n}{\log(n)}\big]$ and for all $n \geq \max(n_1,n_2)$,
\begin{align}
    \frac{t \log(n)}{\lambda_{0}}&-\frac{\log(W)}{\lambda_{0}}-\frac{\log\left(1+2\tilde{\varepsilon} / W\right)}{\lambda_{0}} \leq \eta_{t}^{(n)} \\
    &\leq \Big( \frac{t\log(n)}{\lambda_{0}}-\frac{\log(W)}{\lambda_{0}}-\frac{\log(1-2 \tilde{\varepsilon}/W)}{\lambda_{0}} \Big)\Big(1-2\alpha_{0}\mu_{0}^{(n)}/\lambda_{0}\Big)^{-1}.
\end{align}
From this we obtain
\begin{align}
    -&\frac{\log(1+2\tilde{\varepsilon}/W)}{\lambda_{0}}\leq \eta_{t}^{(n)}-\Big(\frac{t \log(n)}{\lambda_{0}}-\frac{\log(W)}{\lambda_{0}}\Big)  \\ 
    &\hspace{0.5cm}\leq \Big(1-2\alpha_{0} \mu_{0}^{(n)} / \lambda_{0}\Big)^{-1} \Big(\Big(\frac{t \log(n)}{\lambda_{0}}-\frac{\log(W)}{\lambda_{0}} \Big)\frac{2\alpha_{0} \mu_{0}^{(n)}}{\lambda_{0}}-\frac{\log(1-2\tilde{\varepsilon}/W)}{\lambda_{0}}\Big).
\end{align}
In particular, this implies that for all $n\geq \max (n_1,n_2)$,
\begin{align}
    &\sup_{t\in \left[T_1,T_2\frac{\varphi_n}{\log(n)}\right]}\Big{\vert} \eta_{t}^{(n)}-\Big(\frac{t \log(n)}{\lambda_{0}}-\frac{\log(W)}{\lambda_{0}}\Big) \Big{\vert} \leq \max \Big{\{}\frac{\log(1+2\widetilde{\varepsilon}/W)}{\lambda_{0}}\\
    &\hspace{2cm};\Big(1-2\alpha_{0} \mu_{0}^{(n)} / \lambda_{0}\Big)^{-1} \Big(\Big(\frac{T_2 \varphi_n}{\lambda_{0}}-\frac{\log(W)}{\lambda_{0}} \Big)\frac{2\alpha_{0} \mu_{0}^{(n)}}{\lambda_{0}}-\frac{\log(1-2\tilde{\varepsilon}/W)}{\lambda_{0}}\Big)\Big{\}}.
\end{align}
Denote by $D^{(n)}$ the right-hand side of the last inequality. Then it directly follows that 
\begin{equation} \label{lab: E1}
    \mP \left( A^{(n)}\cap \{\delta_1<W<\delta_2\}\cap B_{t_1}\cap C^{(n)}\right) \leq \mP \left(\{D^{(n)}\geq \varepsilon\} \cap \{ \delta_1 <W<\delta_2\} \right).
\end{equation}
Because $\varphi_n$ was defined such that $\varphi_n \mu_0^{(n)} \underset{n \to \infty}{\to}0$, it is possible to find an adequate $\widetilde{\varepsilon}>0$ and $n_3 \in \mathbb{N}$ such that for all $n\geq n_3$, $\mP \left(\{D^{(n)}\geq \varepsilon\}\cap\{ \delta_1<W<\delta_2\} \right) \leq \frac{\nu}{3}$. Together with \eqref{lab: E0} and \eqref{lab: E1}, we deduce \eqref{deterministic approx random time with W right}. 

\textbf{Step 2:} We are going to prove that $\mP \left(A^{(n)} \cap  \{W>0\}\right) \underset{n \to \infty}{\longrightarrow}0.$ We have 
\begin{align}
    \mP \left(A^{(n)} \cap \{W>0\} \right) \leq \mP\left(A^{(n)} \cap \{\delta_1<W<\delta_2\} \right) +\mP\left(0<W<\delta_1\right)+\mP \left(W>\delta_2\right).
\end{align}
Using Equation \eqref{deterministic approx random time with W right}, we obtain 
\begin{equation}
    \limsup_{n \to \infty}\mP \left(A^{(n)} \cap \{W>0\} \right)  \leq \mP\left(0<W<\delta_1\right)+\mP \left(\delta_2<W\right).
\end{equation}
Taking the limit as $(\delta_1, \delta_2) \underset{n \to \infty}{\longrightarrow}(0,\infty)$, and noting that $W$ is finite almost surely (see \eqref{Eq:distribution W}), we conclude the proof.
\end{proof}

\begin{remark}
\label{Rq: order of stopping time wtp reaches size power of n}
From Lemma \ref{convergence temps arret}, the useful results
\begin{align}
\label{uniform conv stopping time}
    \mathbb{P} \Big(\sup_{t \in \left[T_{1},T_{2}\frac{\varphi_n}{\log(n)}\right]} \Big{\vert} \frac{\eta_{t}^{(n)}}{\log(n)}\lambda_{0}-t \Big{\vert} \geq \varepsilon \Big{\vert} W>0 \Big) \underset{n \to \infty}{\longrightarrow} 0
\end{align}
and 
\begin{align}
\label{Rq: approx W by exponential term with tau}
   \mathbb{P} \Big( \sup_{t \in \left[T_{1},T_{2}\frac{\varphi_n}{\log(n)}\right]} \Big{\vert} e^{-\lambda_{0}\left(\eta_{t}^{(n)}-\mathfrak{t}^{(n)}_{t}\right)}-W \Big{\vert} \geq \varepsilon \Big{\vert} W>0\Big)\underset{n \to \infty}{\longrightarrow}0
\end{align}
follow.
\end{remark}

\subsection{The mutant subpopulations dynamics on the deterministic time scale (Theorem \ref{Theorem: non increasin growth rate graph} (i))}
\label{Proof:Number of Mutants, neutral}
In this subsection, Equations \eqref{Equation: conv neutral case} and \eqref{Equation: convergence deleterious case} are proven for the mono-directional graph. The proof will be carried out in two steps. First, we will show the result for a fixed $s \in \mathbb{R}$, uniformly in the parameter $t$. Then, in the second step, we will establish the result uniformly in the parameter $s$ by adapting a method developed in \cite[Lemma 3]{foo2013dynamics}.

\subsubsection{Uniform control on the time parameter t}
\label{First step of the proof}
In this subsection, we will prove the following proposition, which is a less refined result than \eqref{Equation: conv neutral case} and \eqref{Equation: convergence deleterious case}, as it is not uniform in the parameter $s$.
\begin{proposition}
\label{Proposition: mutant pop deterministic time scale not uniform in s}
Let $i \in \mathbb{N}$, $(\psi_n(i),h_{n}(i)) \underset{n \to \infty}{\rightarrow} \infty$ such that there exist $\varphi_n(i) \underset{n \to \infty}{\to} \infty$ such that $h_n(i)=\frac{\log(n)}{\log(\log(n))\theta(i-1)+\varphi_n(i)}$ and $\sqrt{\log(n)}=o(\psi_n(i))$. For all $(t,s) \in \mathbb{R}^{+}\times \mathbb{R}$ define 
\begin{align}
    d_{i}^{(n)}(t,s):=&\1_{\left\{t \in [0,t(i)-h_{n}^{-1}(i))\right\}}+\1_{\left\{t \in [t(i)-h_{n}^{-1}(i),t(i))\right\}}\psi_{n}\log^{\theta(i)-1}(n)\\
    &\hspace{3cm}+\1_{\left\{t \in [ t(i),\infty)\right\}}n^{t-t(i)}\log^{\theta(i)}(n)e^{\lambda(0) s}.
\end{align}
Let $T>0$, $0<T_1<T_2$, and $s \in \mathbb{R}$. We have
\begin{itemize}
    \item If $\lambda_i=\lambda_0$ then $t \mapsto Z_{i}^{(n)}\big(\mathfrak{t}^{(n)}_{t}+s\big)/d_{i}^{(n)}(t,s)$ converges in probability in $L^{\infty}([0,T])$ to $W w_{i}(t)$.
    \item If $\lambda_i<\lambda_0$ then $t \mapsto Z_i^{(n)}\big(\mathfrak{t}^{(n)}_{t(i)+t}+s\big)/n^t \log^{\theta(i)}(n)e^{\lambda_{0}s}$ converges in probability in $L^{\infty}([T_1,T_2])$ to $Ww_{i}(t(i)+t)$.
\end{itemize}
\end{proposition}
The proof is carried out by induction on $i \in \mathbb{N}$. For $i \geq 2$, we assume that Proposition \ref{Proposition: mutant pop deterministic time scale not uniform in s} holds for $i-1$. In the base case, $i=1$, Proposition \ref{Proposition: mutant pop deterministic time scale not uniform in s} is proved without any assumptions. As long as the proof is similar for the initialization and the inductive step, the specific step under consideration will not be indicated. To make the proof as clear as possible, it is divided into several lemmas. All results are derived using a martingale approach. In the next lemma, we introduce the martingales considered for all mutant subpopulations and compute their quadratic variations. 
\begin{lemma}
\label{Lem martingale type 1 pop}
For all $i \in \mathbb{N}$ define
\begin{align}
\label{Martingale definition}
    M_{i}^{(n)}(t):=Z_{i}^{(n)}(t)e^{-\lambda_{i}^{(n)} t}-\int_{0}^{t}2\alpha_{i-1}\mu_{i-1}^{(n)}e^{-\lambda_{i}^{(n)} s}Z_{i-1}^{(n)}(s) ds.
\end{align}
$\big(M_{i}^{(n)}(t)\big)_{t\geq 0}$ is a martingale, with quadratic variation
\begin{align}
\label{quadratic variation}
    &\big\langle M_{i}^{(n)}\big\rangle_t=\int_0^t 2\alpha_{i-1}\mu_{i-1}^{(n)}e^{-2\lambda_{i}^{(n)} s}Z_{i-1}^{(n)}(s) ds\\
    &\hspace{4cm}+\big(\alpha_{i}^{(n)}+\beta_{i}^{(n)}\big)\int_0^t e^{-2 \lambda_{i}^{(n)} s}Z_{i}^{(n)}(s) ds.
\end{align}
\end{lemma}
Since the proof of this lemma is fairly standard, it can be found in the \hyperref[appn]{Appendix}. We can now proceed to prove Proposition \ref{Proposition: mutant pop deterministic time scale not uniform in s}. This proof is structures as follows:
\begin{itemize}
    \item[1.] \textbf{Neutral case ($\lambda_{i}=\lambda_0$):} The proof begins by addressing the neutral case. This part is divided into three major steps, each corresponding to a different time regime for the normalizing term $d_{i}^{(n)}(t,s)$:
    \begin{itemize}
        \item \textbf{First time regime ($t \in [0,t(i)-h_{n}^{-1}(i)]$):} Lemma \ref{Lem: no mutant cell} establishes the asymptotic result for this interval.
        \item \textbf{Second time regime ($t \in [t(i)-h_n^{-1}(i)]$):} Lemma \ref{Lem: control before first mut} covers the convergence within this interval. The proof begins with a first step where the result is established under a more restrictive condition on $\psi_n(i)$. This step is further divided using: Lemma \ref{Lem: negligeable finite variation}, which handles the finite variation process associated with the mutant subpopulation, Lemma \ref{Lem:upper bound expectation subpop}, which controls the expected value of the size of the mutant subpopulation and Lemma \ref{Lem: control quadratic variation}, which controls the expected value of the quadratic variation of the martingale associated with the mutant subpopulation. The second step of the proof proceeds by relaxing the aforementioned restrictive condition on $\psi(n)$ from step 1.
        \item \textbf{Third time regime ($t \in [t(i),T]$):} Lemma \ref{convergence finite variation process K} controls the finite variation process associated with the mutant subpopulation in this regime.
    \end{itemize}
    \item[2.] \textbf{Deleterious case ($\lambda_i<\lambda_0$):} After completing the neutral case, the proof proceeds to the deleterious case, using some of the previously proven lemmas.
\end{itemize} 
\begin{proof} [Proof of Proposition \ref{Proposition: mutant pop deterministic time scale not uniform in s}] 
Let $i \in \mathbb{N}$. For $i \geq 2$ assume that Proposition \ref{Proposition: mutant pop deterministic time scale not uniform in s} holds for $i-1$. We begin by proving the result when $i$ is a neutral trait; specifically, we aim to establish Proposition \ref{Proposition: mutant pop deterministic time scale not uniform in s} (i). All the lemmas mentioned in the proof are not restricted to this neutral assumption and also apply to deleterious mutant traits. 

\textbf{1. Neutral case:} Assume that $\lambda_{i}=\lambda_0$. Let $(\psi_n(i),h_n(i))$ as in Proposition \ref{Proposition: mutant pop deterministic time scale not uniform in s} and $\varepsilon>0$. Notice that
\begin{align}
    &\mP \Big{(} \sup_{t\in \left[0,T\right]} \Big{\lvert}  \frac{Z_{i}^{(n)}\left(\mathfrak{t}^{(n)}_{t}+s\right)}{d_{i}^{(n)}(t,s)}-W w_{i}(t) \Big{\lvert} \geq 3\varepsilon \Big{)} \\ 
    \label{Eq:4}&\hspace{2cm}\leq \mP \Big{(} \sup_{t\in \left[0,t(i)-h_n^{-1}(i)\right)}   Z_{i}^{(n)}\left(\mathfrak{t}^{(n)}_{t}+s\right) \geq \varepsilon \Big{)}\\
    \label{Eq:5}&\hspace{2cm}+\mP \Big{(} \sup_{t \in \left[t(i)-h_n^{-1}(i),t(i)\right)} \frac{Z_{i}^{(n)}\Big(\mathfrak{t}^{(n)}_{t}+s\Big)}{\psi_{n}(i)\log^{\theta(i-1)}(n)} \geq \varepsilon \Big{)} \\ 
    \label{Eq:6}&\hspace{2cm}+\mP \Big{(} \sup_{t \in \left[t(i),T\right]} \Big{\lvert}  \frac{Z_{i}^{(n)}\Big(\mathfrak{t}^{(n)}_{t}+s\Big)}{n^{t-t(i)}\log^{\theta(i)}(n)e^{\lambda_{0} s}}-Ww_{i}(t) \Big{\lvert} \geq \varepsilon \Big{)},
\end{align}
where we used that $\omega_i(t)=0$ for all $t\leq t(i)$. We will show that \eqref{Eq:4}, \eqref{Eq:5} and \eqref{Eq:6} converge to $0$ as $n$ goes to infinity. 

\textbf{(i) First time regime $(t \in [0,t(i)-h_{n}^{-1}(i)])$, convergence to 0 of \eqref{Eq:4}:} The characterization of $t(i)$ as the first time mutant cells of trait $i$ appear on the time scale $t \mapsto \mathfrak{t}^{(n)}_{t}$ is made explicit in the next lemma. More precisely, we exactly show that up until time $t(i)-h_{n}^{-1}(i)$, asymptotically no mutant cells of trait $i$ are generated. In particular, the convergence to $0$ of \eqref{Eq:4} is deduced from the next lemma.  
\begin{lemma}
\label{Lem: no mutant cell}
Let $i \in \mathbb{N}$, and $h_{n}(i)=\frac{\log(n)}{\log(\log(n))\theta(i-1)+\varphi_n(i)}$, where $\varphi_n(i) \underset{n \to \infty}{\to}\infty$ such that $h_n(i) \underset{n \to \infty}{\to} \infty$, and $s \in \mathbb{R}$. For $i \geq 2$, we prove that if Proposition \ref{Proposition: mutant pop deterministic time scale not uniform in s} holds for $i-1$ then
\begin{equation}
\label{Eq:13}
    \mP \Big{(} \sup_{t\in \left[0, t(i)-h_n^{-1}(i)\right]} Z_{i}^{(n)}\Big(\mathfrak{t}^{(n)}_{t}+s\Big)=0 \Big{)} \underset{n \to \infty}{\longrightarrow}1.
\end{equation}
For $i=1$, we prove \eqref{Eq:13} without any conditions.
\end{lemma}

\begin{proof}[Proof of Lemma \ref{Lem: no mutant cell}]
Notice first that 
\begin{align}
\label{event no mutation}
    \Big{\{}\sup_{t \in \left[0,t(i)-h_n^{-1}(i)\right]}  Z_{i}^{(n)}\Big(
\mathfrak{t}^{(n)}_{t}+s\Big)=0\Big{\}}=A^{(n)}\cap B^{(n)},
\end{align}
where $A^{(n)}:=\big{\{} K_{i-1}^{(n)}\big(\mathfrak{t}^{(n)}_{t(i)-h_{n}^{-1}(i)}+s\big)=0\big{\}}$ and $B^{(n)}:=\big{\{} H_{i-1}^{(n)}\big(\mathfrak{t}^{(n)}_{t(i)-h_{n}^{-1}(i)}+s\big)=0\big{\}}$. Indeed, the event on the left-hand side of \eqref{event no mutation} is satisfied if and only if no mutant cell of the subpopulation $Z_{i}^{(n)}$ is generated from the subpopulation $Z_{i-1}^{(n)}$ up until time $\mathfrak{t}^{(n)}_{t(i)-h_{n}^{-1}(i)}+s$. This corresponds almost surely to $A^{(n)} \cap B^{(n)}$. In what follows, we will provide the proof that $\mP\left(A^{(n)}\right)\underset{n \to \infty}{\longrightarrow}1$. The proof that $\mP\left(B^{(n)}\right)\underset{n \to \infty}{\longrightarrow}1$ can be established using a similar method, so it will not be detailed here. This will conclude the proof of Lemma \ref{Lem: no mutant cell}. Therefore, we now address the proof that $\mP\left(A^{(n)}\right)\underset{n \to \infty}{\longrightarrow}1$, which will vary slightly depending on whether $i=1$ or $i\geq 2$. 

\textbf{Case $i=1$:} For all $\widetilde{t} \in \mathbb{R}^{+}$ and $\varepsilon \in \mathbb{R}^{+}$, let $C_{\varepsilon,\widetilde{t}}:=\big{\{} \sup_{s \in [\widetilde{t},\infty]} \big{\vert}e^{-\lambda_0 s}Z_{0}(s)-W \big{\vert}\leq \varepsilon\big{\}}.$ Using the almost sure inequality \eqref{eq:monotonecoupling}, under the event $C_{\varepsilon,\widetilde{t}}$, we have 
\begin{align}
\label{Eq:7}
    &K_{0}^{(n)}\Big( \mathfrak{t}^{(n)}_{t(1)-h_{n}^{-1}(1)}+s\Big) \leq \int_{0}^{\widetilde{t}} \int_{\mathbb{R}^{+}}\1_{\left\{\theta \leq 2 \alpha_{0}\mu_{0}^{(n)}\sup_{v \in [0,\widetilde{t}]}Z_{0}(v)\right\}}N_{0}(du,d\theta) \\
    &\hspace{3cm}+\int_{\widetilde{t}}^{\mathfrak{t}^{(n)}_{t(1)-h_{n}^{-1}(1)}+s} \int_{\mathbb{R}^{+}}\1_{\left\{\theta \leq 2 \alpha_{0}\mu_{0}^{(n)}e^{\lambda_0 u}(\varepsilon+W)\right\}}N_{0}(du,d\theta).
\end{align}
Let us set the following notations
\begin{align}
    &D_{\widetilde{t}}^{(n)}:=\Big{\{} \int_{0}^{\widetilde{t}} \int_{\mathbb{R}^{+}}\1_{\left\{ \theta \leq 2\alpha_{0}\mu_{0}^{(n)}\sup_{v\in [0,\widetilde{t}]}Z_{0}(v)\right\}}N_{0}(du,d\theta)=0\Big{\}}, \\
    &E_{\varepsilon,\widetilde{t}}^{(n)}:=\Big{\{} \int_{\widetilde{t}}^{\mathfrak{t}^{(n)}_{t(1)-h_{n}^{-1}(1)}+s}\int_{\mathbb{R}^{+}}\1_{\left\{\theta \leq 2\alpha_{0}\mu_{0}^{(n)}e^{\lambda_0 u}(\varepsilon+W) \right\}}N_{0}(du,d\theta)=0 \Big{\}}.
\end{align}
Using Equation \eqref{Eq:7} we have that
\begin{align}
    \mP\Big(A^{(n)}\Big)\geq \mP\Big(A^{(n)} \cap C_{\varepsilon,\widetilde{t}}\Big)\geq \mP\Big(C_{\varepsilon,\widetilde{t}}\cap D_{\widetilde{t}}^{(n)}\cap E_{\varepsilon,\widetilde{t}}^{(n)}\Big).
\end{align}
It remains to show that the right-hand side converges to $1$. By the definition of $W$ as the almost sure time limit of the positive martingale $e^{-\lambda_0 t}Z_{0}(t)$, it follows that $\mP\big(C_{\varepsilon,\widetilde{t}}\big) \underset{\widetilde{t} \to \infty}{\longrightarrow}1.$ We also have that $\sup_{v \in [0,\widetilde{t}]}Z_{0}(v)$ is finite almost surely. Combined with the fact that $Z_{0}$ and $N_{0}$ are independent, we deduce that  $\mP\big(D_{\widetilde{t}}^{(n)}\big) \underset{n \to \infty}{\longrightarrow}1,$ because $\mu_0^{(n)} \underset{n \to \infty}{\longrightarrow} 0$. Recall the distribution of $W$, given in \eqref{Eq:distribution W}. Since $W$ and $N_{0}$ are independent, we have
\begin{align}
\label{Eq: no mutant cell computation}
    \hspace{1cm}&\mP\Big(E_{\varepsilon, \widetilde{t}}^{(n)}\Big)=\frac{\beta_{0}}{\alpha_{0}}\mP\Big( \int_{\widetilde{t}}^{\mathfrak{t}^{(n)}_{t(1)-h_{n}^{-1}(1)}+s}\int_{\mathbb{R}^{+}}\1_{\left\{\theta \leq 2\alpha_{0}\mu_{0}^{(n)}e^{\lambda_0 u}\varepsilon \right\}}N_{0}(du,d\theta)=0\Big)+\frac{\lambda_{0}}{\alpha_{0}}\\
    &\cdot\int_{0}^{\infty} \frac{\lambda_{0}}{\alpha_{0}}e^{-\frac{\lambda_{0}}{\alpha_{0}}w}\mP \Big( \int_{\widetilde{t}}^{\mathfrak{t}^{(n)}_{t(1)-h_{n}^{-1}(1)}+s}\int_{\mathbb{R}^{+}}\1_{\left\{\theta \leq 2\alpha_{0}\mu_{0}^{(n)}e^{\lambda_{0} u}(\varepsilon+w) \right\}}N_{0}(du,d\theta)=0\Big)dw\\
    &=\frac{\beta_{0}}{\alpha_{0}}\exp\Big(-\int_{\widetilde{t}}^{\mathfrak{t}^{(n)}_{t(1)-h_{n}^{-1}(1)}+s}2\alpha_{0}\mu_{0}^{(n)}e^{\lambda_{0} u}\varepsilon du\Big)\\
    &\hspace{1cm}+\frac{\lambda_{0}}{\alpha_{0}}\int_{0}^{\infty} \frac{\lambda_{0}}{\alpha_{0}}e^{-\frac{\lambda_{0}}{\alpha_{0}}w} \exp\Big(-\int_{\widetilde{t}}^{\mathfrak{t}^{(n)}_{t(1)-h_{n}^{-1}(1)}+s}2\alpha_{0}\mu_{0}^{(n)}e^{\lambda_{0} u}(\varepsilon+w) du\Big)dw \\
    & \geq \frac{\beta_{0}}{\alpha_{0}}\exp\Big(-\frac{2\alpha_{0}\Big(n^{t(1)}\mu_{0}^{(n)}\Big)}{\lambda_{0}}\varepsilon e^{\lambda_{0} s}e^{-h_{n}^{-1}(1)\log(n)}\Big)\\
    &\hspace{1cm}+\frac{\lambda_{0}}{\alpha_{0}}\int_{0}^{\infty} \frac{\lambda_{0}}{\alpha_{0}}e^{-\frac{\lambda_{0}}{\alpha_{0}}w}\exp\Big(-\frac{2\alpha_{0}\Big(n^{t(1)}\mu_{0}^{(n)}\Big)}{\lambda_{0}}(\varepsilon+w)e^{\lambda_{0} s}e^{-h_{n}^{-1}(1)\log(n)} \Big)dw \\
    &\underset{n \to \infty}{\longrightarrow}1,
\end{align}
where we first use that for all $w\geq 0$, $$\frac{2\alpha_{0}\big(n^{t(1)}\mu_{0}^{(n)}\big)}{\lambda_{0}}(\varepsilon+w) e^{\lambda_{0} s}e^{-h_{n}^{-1}(1)\log(n)} \underset{n \to \infty}{\longrightarrow}0.$$ This follows from the choice of $h_n(1)$, which ensures that $h_{n}^{-1}(1)\log(n)\underset{n \to \infty}{\longrightarrow}\infty$. Then, we apply the dominated convergence theorem to obtain
\begin{align}
   \int_{0}^{\infty} \frac{\lambda_{0}}{\alpha_{0}}e^{-\frac{\lambda_{0}}{\alpha_{0}}w}&e^{\Big(-\frac{2\alpha_{0}\big(n^{t(1)}\mu_{0}^{(n)}\big)}{\lambda_{0}}(\varepsilon+w)e^{\lambda_{0} s}e^{-h_{n}^{-1}(1)\log(n)} \Big)}dw\underset{n \to \infty}{\to}\int_{0}^{\infty} \frac{\lambda_{0}}{\alpha_{0}}e^{-\frac{\lambda_{0}}{\alpha_{0}}w}dw=1.
\end{align}
Finally, we have shown that $\mP\left(A^{(n)}\right)\underset{n \to \infty}{\longrightarrow}1,$ which concludes the proof for the case $i=1$.

\textbf{Case $i\geq 2$:} 
Let $\widetilde{t}(i):=\frac{t(i)+t(i-1)}{2}$ and $\Psi_{n} \underset{n \to \infty}{\longrightarrow}\infty$. For $\varepsilon>0$, define
\begin{align}
    C_{\varepsilon}^{(n)}:= \Big\{\sup_{t \in [0,t(i)]}\Big{\vert} \frac{Z_{i-1}^{(n)}\Big(\mathfrak{t}_{t}^{(n)}\Big)}{d^{(n)}(t)}-W \1_{\{t \geq \tilde{t}(i)\}}w_{i-1}(t) \Big{\vert} \leq \varepsilon\Big\},
\end{align}
where $$d^{(n)}(t)=\1_{\left\{t \in [0,\widetilde{t}(i))\right\}}n^{\widetilde{t}(i)-t(i-1)}\log^{\theta(i-1)}(n)\Psi_{n}+\1_{\left\{t \in \left[\widetilde{t}(i),t(i)\right]\right\}}n^{t-t(i-1)}\log^{\theta(i-1)}(n).$$ Under the event $C_{\varepsilon}^{(n)}$, we have 
\begin{align}
\label{Eq:16}
    \hspace{0.5cm}&K_{i-1}^{(n)} \Big(\mathfrak{t}^{(n)}_{t(i)-h_{n}^{-1}(i)}+s\Big) \leq \int_{0}^{\mathfrak{t}^{(n)}_{\widetilde{t}(i)}} \int_{\mathbb{R}^{+}}\1_{\left\{\theta \leq 2 \alpha_{i-1}\mu_{i-1}^{(n)}\varepsilon n^{\tilde{t}(i)-t(i-1)}\log^{\theta(i-1)}(n)\Psi_{n}\right\}}N_{i-1}(du,d\theta)\\
    &+\int_{\mathfrak{t}^{(n)}_{\widetilde{t}(i)}}^{\mathfrak{t}^{(n)}_{t(i)-h_{n}^{-1}(i)}+s} \int_{\mathbb{R}^{+}}\1_{\left\{\theta \leq 2 \alpha_{i-1}\mu_{i-1}^{(n)}\left(\varepsilon+W w_{i-1}(t(i))\right)e^{\lambda_0 u}n^{-t(i-1)}\log^{\theta(i-1)}(n)\right\}}N_{i-1}(du,d\theta).
\end{align}
Let us introduce the events 
\begin{align}
    & D_{\varepsilon}^{(n)} :=\Big\{\int_{0}^{\mathfrak{t}_{\widetilde{t}(i)}^{(n)}} \int_{\mathbb{R}^{+}}\1_{\left\{\theta \leq 2 \alpha_{i-1}\mu_{i-1}^{(n)}\varepsilon n^{\tilde{t}(i)-t(i-1)}\log^{\theta(i-1)}(n)\Psi_{n}\right\}}N_{i-1}(du,d\theta)=0\Big\}, \\
    & E_{\varepsilon}^{(n)} := \Big\{ \int_{\mathfrak{t}_{\widetilde{t}(i)}^{(n)}}^{\mathfrak{t}_{t(i)-h_{n}^{-1}(i)}^{(n)}+s} \int_{\mathbb{R}^{+}}\1_{\left\{\theta \leq 2 \alpha_{i-1}\mu_{i-1}^{(n)}\left(\varepsilon+W w_{i-1}(t(i))\right)e^{\lambda_0 u}n^{-t(i-1)}\log^{\theta(i-1)}(n)\right\}}\\
    &\hspace{9cm}\cdot N_{i-1}(du,d\theta)=0\Big\}.
\end{align}
From \eqref{Eq:16} we obtain $\mP \big(A^{(n)}\big) \geq \mP \big(A^{(n)}\cap C_{\varepsilon}^{(n)}\big) \geq \mP \big(C_{\varepsilon}^{(n)} \cap D_{\varepsilon}^{(n)} \cap E_{\varepsilon}^{(n)}\big).$ It remains to show that the right-hand side converges to $1$. By assuming Proposition \ref{Proposition: mutant pop deterministic time scale not uniform in s} holds for trait $i-1$, it follows that $\mathbb{P} \big( C_{\varepsilon}^{(n)}\big) \underset{n \to \infty}{\longrightarrow}1$. Secondly, we have 
\begin{align}
    \mP \Big( D_{\varepsilon}^{(n)}\Big)=\exp \Big(-\widetilde{t}(i)\frac{\log^{\theta(i-1)+1}(n)}{\lambda_0}2\alpha_{i-1}\mu_{i-1}^{(n)}\varepsilon \sqrt{n^{\ell(i-1)}}\Psi_{n}\Big) \underset{n \to \infty}{\longrightarrow}1,
\end{align}
because $\widetilde{t}(i)-t(i-1)=\ell(i-1)/2$, and $\Psi_{n}$ can be chosen to satisfy both $\Psi_{n} \underset{n \to \infty}{\longrightarrow}\infty$ and \\
$\log^{\theta(i-1)+1}(n)\Psi_{n} \sqrt{n^{\ell(i-1)}}\mu_{i-1}^{(n)}\underset{n \to \infty}{\longrightarrow}0$. Using a similar approach as in the computation of \eqref{Eq: no mutant cell computation}, we get 
\begin{align}
    &\mP\Big(E_{\varepsilon}^{(n)}\Big)\geq \frac{\beta_{0}}{\alpha_{0}}\exp\Big[-\frac{2\alpha_{i-1}\Big(n^{\ell(i-1)}\mu_{i-1}^{(n)}\Big)}{\lambda_0}\varepsilon \log^{\theta(i-1)}(n)e^{\lambda_0 s}e^{-h_{n}^{-1}(i)\log(n)}\Big]\\
    &+\frac{\lambda_0}{\alpha_{0}}\int_{0}^{\infty} \frac{\lambda_0}{\alpha_{0}} e^{-\frac{\lambda_0}{\alpha_{0}}w}\\
    &\cdot\exp\Big(-\frac{2\alpha_{i-1}\Big(n^{\ell(i-1)}\mu_{i-1}^{(n)}\Big)}{\lambda_0}(\varepsilon+ww_{i-1}(t(i)))\log^{\theta(i-1)}(n) e^{\lambda_0 s}e^{-h_{n}^{-1}(i)\log(n)} \Big)dw\\
    &\underset{n \to \infty}{\longrightarrow}1,
\end{align}
where, $\forall w\geq 0$, $$\frac{2\alpha_{i-1}\left(n^{\ell(i-1)}\mu_{i-1}^{(n)}\right)}{\lambda_0}(\varepsilon+w w_{i-1}(t(i)))\log^{\theta(i-1)}(n) e^{\lambda_0 s}e^{-h_{n}^{-1}(i)\log(n)} \underset{n \to \infty}{\longrightarrow}0,$$ because $$\log^{\theta(i-1)}(n)e^{-h_{n}^{-1}(i)\log(n)}=\exp \left(\theta(i-1)\log(\log(n))-\log(n)h_n^{-1}(i)\right)\underset{n \to \infty}{\longrightarrow}0$$ by hypothesis on $h_n(i)$. Then, we apply the dominated convergence theorem to get
\begin{align}
    &\int_{0}^{\infty} \frac{\lambda_0}{\alpha_{0}}e^{-\frac{\lambda_0}{\alpha_{0}}w}\exp\Big(-\frac{2\alpha_{i-1}\Big(n^{\ell(i-1)}\mu_{i-1}^{(n)}\Big)}{\lambda_0}\\
    &\hspace{5cm}\cdot(\varepsilon+ww_{i-1}(t(i)))\log^{\theta(i-1)}(n) e^{\lambda_0 s}e^{-h_{n}^{-1}(i)\log(n)} \Big)dw \\
    &\hspace{2cm}\underset{n \to \infty}{\longrightarrow}\int_{0}^{\infty} \frac{\lambda_0}{\alpha_{0}}e^{-\frac{\lambda_0}{\alpha_{0}}w}dw=1.
\end{align}
Finally, we have shown that $\mP\left(A^{(n)}\right)\underset{n \to \infty}{\rightarrow}1,$ which concludes the proof.
\end{proof}

\textbf{(ii) Second time regime $t \in [t(i)-h_n^{-1}(i),t(i)]$, convergence to 0 of \eqref{Eq:5}:} In the next lemma we show that in the time interval $[t(i)-h_n^{-1}(i), t(i)]$, the size of the mutant subpopulation of trait $i$ does not achieve any power of $n$. We control its growth by the factor $\psi_n(i) \log^{\theta(i-1)}(n)$, with a well-chosen function $\psi_n(i)$. Heuristically, the total number of mutant cells of trait $i$ generated from mutational events up to time $t(i)$ is of order $\mathcal{O}\big(\log^{\theta(i-1)}(n)\big)$. Moreover, with the remaining time for the lineages of these mutant cells to grow, we are able to control the size of the mutant subpopulation of trait $i$ by at most $\sqrt{\log(n)}\log^{\theta(i-1)}(n)$. Consequently, by dividing by any function $\psi_n(i)$ that satisfies $\sqrt{\log(n)}=o(\psi_n(i))$, the asymptotic limit is $0$.
\begin{lemma}
\label{Lem: control before first mut}
    Let $i \in \mathbb{N}$, $h_{n}(i)=\frac{\log(n)}{\log(\log(n))\theta(i-1)+\varphi_n(i)}$, where $\varphi_n(i) \underset{n \to \infty}{\to} \infty$ such that $h_n(i) \underset{n \to \infty}{\to}\infty$, $\psi_n(i) \to \infty$ such that $\sqrt{\log(n)}=o(\psi_n(i))$, $s \in \mathbb{R}$ and $\varepsilon>0$. For $i \geq 2$, we prove that if Proposition \ref{Proposition: mutant pop deterministic time scale not uniform in s} holds for $i-1$ then
    \begin{align} \label{E9}
        \mP \Big( \sup_{t \in \left[t(i)-h_n^{-1}(i),t(i)\right]} \frac{Z_{i}^{(n)}\Big(\mathfrak{t}^{(n)}_{t}+s\Big)}{\psi_n(i) \log^{\theta(i-1)}(n)} \geq \varepsilon\Big)\underset{n \to \infty}{\longrightarrow}0.
    \end{align}
For $i=1$, we prove \eqref{E9} without any conditions.
\end{lemma}

\begin{proof}
We begin by proving the same result under the more restrictive condition $\log(n)e^{\varphi_{n}(i)}=o(\psi_{n}^{2}(i)).$

\textbf{Step 1:} Let $\psi_{n}(i)$ satisfying the previous equation. For all $t \in \left[t(i)-h_n^{-1}(i),t(i)\right]$, we have 
 \begin{align}
 \label{Eq: calcul intercale}
     &\frac{Z_{i}^{(n)}\big( \mathfrak{t}^{(n)}_{t}+s\big)}{\psi_n(i) \log^{\theta(i-1)}(n)}\\
     &=\frac{Z_{i}^{(n)}\big( \mathfrak{t}^{(n)}_{t}+s\big)e^{-\lambda_i^{(n)} \big(\mathfrak{t}^{(n)}_{t}+s\big)}-Z_{i}^{(n)}\big(\mathfrak{t}^{(n)}_{t(i)-h_n^{-1}(i)}+s\big)e^{-\lambda_{i}^{(n)} \big(\mathfrak{t}^{(n)}_{t(i)-h_n^{-1}(i)}+s\big)}}{\psi_n(i) \log^{\theta(i-1)}(n) e^{-\lambda_i^{(n)} \big( \mathfrak{t}^{(n)}_{t}+s\big)}}\\
     &\hspace{2cm}+\frac{Z_{i}^{(n)}\big(\mathfrak{t}^{(n)}_{t(i)-h_n^{-1}(i)}+s\big)e^{-\lambda_{i}^{(n)} \big(\mathfrak{t}^{(n)}_{t(i)-h_n^{-1}(i)}+s\big)}}{\psi_n(i)  \log^{\theta(i-1)}(n)e^{-\lambda_i^{(n)}\big( \mathfrak{t}^{(n)}_{t}+s\big) }}\\
     &=\frac{M_{i}^{(n)}\big( \mathfrak{t}^{(n)}_{t}+s\big)-M_{i}^{(n)}\big(\mathfrak{t}^{(n)}_{t(i)-h_n^{-1}(i)}+s\big)}{\psi_n(i) \log^{\theta(i-1)}(n) e^{-\lambda_i^{(n)} \big( \mathfrak{t}^{(n)}_{t}+s\big)}}+\frac{\int_{\mathfrak{t}^{(n)}_{t(i)-h_n^{-1}(i)}+s}^{\mathfrak{t}^{(n)}_{t}+s}2\alpha_{i-1}\mu_{i-1}^{(n)}e^{-\lambda_{i}^{(n)} u}Z_{i-1}^{(n)}(u) du}{\psi_n(i) \log^{\theta(i-1)}(n) e^{-\lambda_i^{(n)} \big( \mathfrak{t}^{(n)}_{t}+s\big)}}\\ 
     &\hspace{2cm}+\frac{Z_{i}^{(n)}\big(\mathfrak{t}^{(n)}_{t(i)-h_n^{-1}(i)}+s\big)}{\psi_n(i) \log^{\theta(i-1)}(n) e^{-\lambda_i^{(n)} \mathfrak{t}^{(n)}_{t-t(i)+h_n^{-1}(i)}}}.
 \end{align}
Then, we have 
 \begin{align}
    & \mP \Big( \sup_{t \in \left[t(i)-h_n^{-1}(i),t(i)\right]} \frac{Z_{i}^{(n)}\big(\mathfrak{t}^{(n)}_{t}+s\big)}{\psi_n(i) \log^{\theta(i-1)}(n)} \geq 3\varepsilon\Big) \\
    &\label{Eq: eq1}\hspace{1cm}\leq \mP \Big( \sup_{t \in \left[t(i)-h_n^{-1}(i),t(i)\right]} \Big{\vert} \frac{M_{i}^{(n)}\big( \mathfrak{t}^{(n)}_{t}+s\big)-M_{i}^{(n)}\big(\mathfrak{t}^{(n)}_{t(i)-h_n^{-1}(i)}+s\big)}{\psi_n(i) \log^{\theta(i-1)}(n) e^{-\lambda_i^{(n)} \left( \mathfrak{t}^{(n)}_{t}+s\right)}}\Big{\vert} \geq \varepsilon\Big) \\
     &\label{Eq: eq2}\hspace{1cm}+\mP \Big( \sup_{t \in \left[t(i)-h_n^{-1}(i),t(i)\right]}\frac{\int_{\mathfrak{t}^{(n)}_{t(i)-h_n^{-1}(i)}+s}^{\mathfrak{t}^{(n)}_{t}+s}2\alpha_{i-1}\mu_{i-1}^{(n)}e^{-\lambda_{i}^{(n)} u}Z_{i-1}^{(n)}(u) du}{\psi_n(i) \log^{\theta(i-1)}(n) e^{-\lambda_i^{(n)} \left( \mathfrak{t}^{(n)}_{t}+s\right)}}\geq \varepsilon\Big) \\
     &\label{Eq: eq3}\hspace{1cm}+\mP \Big( \sup_{t \in \left[t(i)-h_n^{-1}(i),t(i)\right]}\frac{Z_{i}^{(n)}\big(\mathfrak{t}^{(n)}_{t(i)-h_n^{-1}(i)}+s\big)}{\psi_n(i) \log^{\theta(i-1)}(n) e^{-\lambda_i^{(n)} \mathfrak{t}^{(n)}_{t-t(i)+h_n^{-1}(i)}}}\geq \varepsilon\Big).
 \end{align}
We have $\eqref{Eq: eq3}\leq \mP \Big( Z_{i}^{(n)}\big(\mathfrak{t}^{(n)}_{t(i)-h_n^{-1}(i)}+s\big) \geq 1\Big),$ because a necessary condition for fulfilling the condition of interest is that there is at least one mutant cell of trait $i$ at time $\mathfrak{t}^{(n)}_{t(i)-h_n^{-1}(i)}+s$. Then, applying Lemma \ref{Lem: no mutant cell} shows that \eqref{Eq: eq3} converges to $0$. The convergence to 0 of the term \eqref{Eq: eq2} follows by applying the subsequent lemma. Note, in particular, that $(\psi_n(i), h_n(i))$ satisfies the condition of this lemma.

\begin{lemma}
\label{Lem: negligeable finite variation}
Let $i \in \mathbb{N}$, $h_n(i)=\frac{\log(n)}{\log(\log(n))\theta(i-1)+\varphi_n(i)}$, where $\varphi_n(i) \underset{n \to \infty}{\to} \infty$ such that $h_n(i) \underset{n \to \infty}{\to} \infty$, $\psi_n(i) \underset{n \to \infty}{\to} \infty$ such that $\log(n)=o(\psi_n(i) h_n(i))$, $s \in \mathbb{R}$ and $\varepsilon>0$. For $i \geq 2$, we prove that if Proposition \ref{Proposition: mutant pop deterministic time scale not uniform in s} holds for $i-1$ then
    \begin{align}
    \label{E7}
        \mP \Big( \sup_{t \in [t(i)-h_n^{-1}(i),t(i)]} \frac{\int_{\mathfrak{t}^{(n)}_{t(i)-h_n^{-1}(i)}+s}^{\mathfrak{t}^{(n)}_{t}+s}2\alpha_{i-1}\mu_{i-1}^{(n)}e^{-\lambda_{i}^{(n)} u}Z_{i-1}^{(n)}(u) du}{\psi_n(i)  \log^{\theta(i-1)}(n) e^{-\lambda_i^{(n)} \left( \mathfrak{t}^{(n)}_{t}+s\right)}}\geq \varepsilon\Big) \underset{n \to \infty}{\longrightarrow}0.
    \end{align}
For $i=1$, we prove \eqref{E7} without any conditions.
\end{lemma}

\begin{proof}[Proof of Lemma \ref{Lem: negligeable finite variation}]
Let 
\begin{align}
    a_{t}^{(n)}:=\frac{\int_{\mathfrak{t}^{(n)}_{t(i)-h_n^{-1}(i)}+s}^{\mathfrak{t}^{(n)}_{t}+s} 2\alpha_{i-1}\mu_{i-1}^{(n)} e^{-\lambda_{i}^{(n)} u}Z_{i-1}^{(n)}(u)du}{\psi_n(i)\log^{\theta(i-1)}(n) e^{-\lambda_i^{(n)} \left( \mathfrak{t}^{(n)}_{t}+s\right)}}.
\end{align}
Our aim is to prove that for all $\varepsilon>0$,
\begin{align}\label{E8}
    \mP \Big( \sup_{t \in \left[t(i)-h_n^{-1}(i),t(i)\right]} a_{t}^{(n)} \leq \varepsilon\Big) \underset{n \to \infty}{\to}1.
\end{align}

\textbf{Case $i=1$:} 
We have 
\begin{align}
\label{Eq: eq 51}
    a_{t}^{(n)}&=\frac{e^{\lambda_1^{(n)}\left(\mathfrak{t}_{t}^{(n)}+s\right)}}{\psi_n(1)} \int_{\mathfrak{t}^{(n)}_{t(1)-h_n^{-1}(1)}+s}^{\mathfrak{t}^{(n)}_{t}+s} 2\alpha_{0}\mu_{0}^{(n)} \Big[W+\left(e^{-\lambda_0 u}Z_{0}(u)-W\right)\\
    &\hspace{5cm}+\left(e^{-\lambda_{0}^{(n)}u}Z_{0}^{(n)}(u)-e^{-\lambda_0 u}Z_{0}(u)\right)\Big]e^{\left(\lambda_{0}^{(n)}-\lambda_{1}^{(n)}\right)u}du.
\end{align}
Let us define
\begin{align}
    &B_{\widetilde{\varepsilon}}^{(n)}:= \Big\{ \sup_{u \in \mathbb{R}^{+}} \Big{\vert} e^{-\lambda_0 u}Z_{0}(u) -e^{-\lambda_{0}^{(n)}u}Z_{0}^{(n)}(u)\Big{\vert} \leq \widetilde{\varepsilon}\Big\}, \\
    &C_{x,\widetilde{\varepsilon}}:= \Big\{ \sup_{u \in [x,\infty]} \vert e^{-\lambda_0 u}Z_{0}(u)-W\vert \leq \widetilde
    \varepsilon\Big\}.
\end{align}
According to Lemma \ref{Lem:control primary pop} and the definition of $W$ (see \eqref{W def}) we have both that $\mP \big( B_{\widetilde{\varepsilon}}^{(n)}\big) \underset{n \to \infty}{\to}1$ and $\mP \big( C_{\sqrt{\log(n)},\widetilde{\varepsilon}}\big) \underset{n \to \infty}{\to}1$. Then, for sufficiently large $n$, under the event $B_{\widetilde{\varepsilon}}^{(n)} \cap C_{\sqrt{\log(n)},\widetilde{\varepsilon}}$, we have 
\begin{align}
    a_{t}^{(n)}&\leq 2 \alpha_0\Big(n^{t(1)}\mu_{0}^{(n)}\Big)(W+2\widetilde{\varepsilon}) I_n,
\end{align}
where $I_n:=\frac{e^{\lambda_{1}^{(n)}\left(\mathfrak{t}^{(n)}_{t}+s\right)}}{\psi_n(1) n^{t(1)}}\int_{\mathfrak{t}^{(n)}_{t(i)-h_n^{-1}(1)}+s}^{\mathfrak{t}^{(n)}_{t}+s} e^{\left(\lambda_0-\lambda_{1}^{(n)}\right)u}du.$ In the case where $\lambda_1<\lambda_0$, we have that 
\begin{align}
\label{Eq: eq 55}
    I_n\leq \frac{e^{\lambda_{1}^{(n)}\left(\mathfrak{t}^{(n)}_{t}+s\right)}}{\psi_n(1) n^{t(1)}} \frac{e^{\left(\lambda_0-\lambda_1^{(n)}\right)\left(\mathfrak{t}^{(n)}_{t}+s\right)}}{\lambda_0-\lambda_1}= \frac{e^{-\lambda_0 \mathfrak{t}^{(n)}_{t(1)-t}}e^{\lambda_0 s}}{\psi_n(1) (\lambda_0-\lambda_1)} \leq \frac{e^{\lambda_0 s}}{\psi_n(1) (\lambda_0-\lambda_1)}.
\end{align}
In the case where $\lambda_1=\lambda_0$, recalling that $\lambda_1^{(n)}=\lambda_0-2\alpha_1 \mu_1^{(n)}$, we obtain 
\begin{align}
\label{Eq: eq 56}
   I_n& \leq \frac{e^{\lambda_0 s} e^{-2\alpha_1\mu_1^{(n)}\left(\mathfrak{t}^{(n)}_{t}+s\right)}}{\psi_n(1)}\frac{e^{2\alpha_1\mu_1^{(n)}\left(\mathfrak{t}^{(n)}_{t}+s\right)}-e^{2\alpha_1\mu_1^{(n)}\left(\mathfrak{t}^{(n)}_{t(1)-h_n^{-1}(1)}+s\right)}}{2\alpha_1 \mu_1^{(n)}} \\
   &=\frac{e^{\lambda_0 s}}{\psi_n(1)} \frac{1-e^{-2\alpha_1 \mu_1^{(n)} \mathfrak{t}^{(n)}_{t-t(1)+h_n^{-1}(1)}}}{2\alpha_1 \mu_1^{(n)}} \\
    &\leq \frac{e^{\lambda_0 s}\log(n)}{\psi_n(1) h_n(1)\lambda_0} ,
\end{align}
where for the last inequality, we use the fact that $\mathfrak{t}^{(n)}_{t-t(1)+h_n^{-1}(1)} \leq \log(n)/(h_n(1)\lambda_0)$ and apply the following equation with $a=2\alpha_1 \mu_1^{(n)}>0$ and $x=\mathfrak{t}^{(n)}_{t-t(1)+h_n^{-1}(1)}$ 
\begin{align}
\label{Eq: eq 54}
    \forall x \geq 0, \forall a>0, \frac{1-e^{-ax}}{a} \leq x.
\end{align}
In any case, since $W$ is a finite random variable (see \eqref{Eq:distribution W}), we find \eqref{E8}. 

\textbf{Case $i \geq 2$:}
Assume Proposition \ref{Proposition: mutant pop deterministic time scale not uniform in s} holds for $i-1$. We have $\mP \left(B_{\widetilde{\varepsilon}}^{(n)}\right) \underset{n \to \infty}{\to}1$ with 
\begin{align}
    &B_{\widetilde{\varepsilon}}^{(n)}:= \Big\{ \sup_{v \in \left[t(i)-h_n^{-1}(i),t(i)\right]} \Big{\vert} \frac{Z_{i-1}^{(n)}\big(\mathfrak{t}^{(n)}_{v}+s\big)}{n^{v-t(i-1)}e^{\lambda_0 s}\log^{\theta(i-1)}(n)}-W w_{i-1}(v)\Big{\vert} \leq \widetilde{\varepsilon}\Big\}.
\end{align}
Using the change of variable $u=\mathfrak{t}_{v}^{(n)}+s$ and the fact that $t(i-1)=t(i)-\ell(i-1)$, notice that 
\begin{align}
    &a_{t}^{(n)}=\frac{e^{\lambda_i^{(n)}\left(\mathfrak{t}^{(n)}_{t}+s\right)}}{\psi_n(i) n^{t(i)}}\int_{t(i)-h_n^{-1}(i)}^{t}2\alpha_{i-1}\big(n^{\ell(i-1)}\mu_{i-1}^{(n)}\big)\\
    &\hspace{4cm}\cdot \frac{Z_{i-1}^{(n)}\big(\mathfrak{t}^{(n)}_{v}+s\big)}{n^{v-t(i-1)}e^{\lambda_0 s}\log^{\theta(i-1)}(n)} e^{\left(\lambda_0-\lambda_{i}^{(n)}\right)\left(\mathfrak{t}^{(n)}_{v}+s\right)}\frac{\log(n)}{\lambda_0}dv.
\end{align}
Since $w_{i-1}$ is a non-decreasing function, it follows that, under the event $B_{\widetilde{\varepsilon}}^{(n)}$,
\begin{align}
    a_{t}^{(n)} &\leq 2 \alpha_{i-1} \big( n^{\ell(i-1)}\mu_{i-1}^{(n)}\big)  \left(W w_{i-1}(t(i)) +\widetilde{\varepsilon}\right)\frac{e^{\lambda_{i}^{(n)}\left(\mathfrak{t}^{(n)}_{t}+s\right)}}{\psi_n(i) n^{t(i)}}\int_{\mathfrak{t}^{(n)}_{t(i)-h_n^{-1}(i)}+s}^{\mathfrak{t}^{(n)}_{t}+s} e^{\left(\lambda_0-\lambda_{i}^{(n)}\right)u}du.
\end{align}
By similar computations as in \eqref{Eq: eq 55} and \eqref{Eq: eq 56}, \eqref{E8} follows.
\end{proof}

Now, we will prove that \eqref{Eq: eq1} converges to 0. We begin by introducing two lemmas, whose proofs are provided in the \hyperref[appn]{Appendix}, which allow us to control both the expected size of any mutant subpopulation and the quadratic variation associated to the martingale $M^{(n)}_{i}$. First, a natural upper bound on the expected growth of each mutant subpopulation can be easily obtained, as stated in the next lemma. 
\begin{lemma}
\label{Lem:upper bound expectation subpop}
For all $i \in \mathbb{N}_0$ and $u \geq 0$,
\begin{align}
\label{Equation101}
    \E\Big[Z_{i}^{(n)}(u)\Big]&\leq C_{i}\mu_{\otimes,i}^{(n)}  u^{\theta(i)} e^{\lambda_{0}u},
\end{align}
where $\mu_{\otimes,i}^{(n)}:=\overset{i}{\underset{j=1}{\prod}}\mu_{j-1}^{(n)}$ and $C_{i}:=\prod_{j=1}^{i}2\alpha_{j-1}\big(\1_{\{\lambda_{j}=\lambda_0\}}+\1_{\{\lambda_j<\lambda_0\}}\frac{1}{\lambda_0-\lambda_j}\big).$
\end{lemma}
Notice that there are three key components. The first is the mutational cost to produce such mutant cells, represented by the term $\mu_{\otimes,i}^{(n)}$. The second component is the contribution over time of all neutral mutations along the path leading to the mutant subpopulation in question. The third component is the exponential growth at rate $\lambda_0$ exhibited by the wild-type subpopulation. Additionally, using the expression for the quadratic variation of the martingale associated to a mutant subpopulation, given in Equation \eqref{quadratic variation}, and the previous Lemma \ref{Lem:upper bound expectation subpop}, a natural upper bound on its mean is derived and summarized in the next lemma. 

\begin{lemma}
\label{Lem: control quadratic variation}
    Let $0<t_1^{(n)}<t_2$ and $s \in \mathbb{R}$. There exist $N \in \mathbb{N}$ and $C(i)>0$ such that, for all $n \geq N$, we have 
    \begin{align}
        &\E \Big[\big\langle M_{i}^{(n)}\big\rangle_{\mathfrak{t}^{(n)}_{t_2}+s}-\big\langle M_{i}^{(n)}\big\rangle_{\mathfrak{t}^{(n)}_{t_1^{(n)}}+s}\Big] \leq C(i)  \mu_{\otimes,i}^{(n)}\Big[ \1_{\{\lambda_i=\lambda_0\}}\frac{ e^{-\lambda_0 s}\big(\mathfrak{t}^{(n)}_{t_1^{(n)}}+s\big)^{\theta(i)} }{ n^{t_1^{(n)}}}  \big( \mathfrak{t}^{(n)}_{t_2}+s\big)^{\theta(i)}\\ &\cdot\Big(\1_{\{\lambda_0>2\lambda_i\}}e^{\left(\lambda_0-2\lambda_i\right)\big( \mathfrak{t}^{(n)}_{t_2}+s\big)}+\1_{\{\lambda_0=2\lambda_i\}}\Big(\mathfrak{t}^{(n)}_{t_2}+s\Big)+\1_{\{\lambda_i<\lambda_0<2\lambda_i\}}e^{-\left(2\lambda_i-\lambda_0\right)\big(\mathfrak{t}^{(n)}_{t_1^{(n)}}+s\big)}\Big) \Big].
    \end{align}
\end{lemma}

Now we can prove that \eqref{Eq: eq1} converges to 0. Using the fact that $\log^{\theta(i-1)}(n)=e^{\lambda_0 \mathfrak{t}^{(n)}_{h_{n}^{-1}(i)}}e^{-\varphi_{n}(i)}$, we can rewrite for all $t \in \left[t(i)-h_n^{-1}(i),t(i)\right] $
\begin{align}
    &\frac{ \left\vert M_{i}^{(n)}\big( \mathfrak{t}^{(n)}_{t}+s\big)-M_{i}^{(n)}\big(\mathfrak{t}^{(n)}_{t(i)-h_n^{-1}(i)}+s\big) \right\vert}{\psi_n(i) \log^{\theta(i-1)}(n) e^{-\lambda_i^{(n)} \left( \mathfrak{t}^{(n)}_{t}+s\right)}} \\
    &\hspace{2cm}=\frac{ \left\vert M_{i}^{(n)}\big( \mathfrak{t}^{(n)}_{t}+s\big)-M_{i}^{(n)}\big(\mathfrak{t}^{(n)}_{t(i)-h_n^{-1}(i)}+s\big)\right\vert}{ \psi_{n}(i)e^{-\varphi_{n}(i)} e^{\left(\lambda_0-\lambda_i^{(n)}\right)\mathfrak{t}^{(n)}_{t-t(i)+h_{n}^{-1}(i)}}e^{-\lambda_{i}^{(n)}\big(\mathfrak{t}^{(n)}_{t(i)-h_{n}^{-1}(i)}+s\big)} e^{\lambda_0 \mathfrak{t}^{(n)}_{t(i)-t}}} \\
    &\hspace{2cm} \leq \frac{ \left\vert M_{i}^{(n)}\big( \mathfrak{t}^{(n)}_{t}+s\big)-M_{i}^{(n)}\big(\mathfrak{t}^{(n)}_{t(i)-h_n^{-1}(i)}+s\big)\right\vert}{ \psi_{n}(i)e^{-\varphi_{n}(i)} e^{\left(\lambda_0-\lambda_i^{(n)}\right)\mathfrak{t}^{(n)}_{t-t(i)+h_{n}^{-1}(i)}}e^{-\lambda_{i}^{(n)}\big(\mathfrak{t}^{(n)}_{t(i)-h_{n}^{-1}(i)}+s\big)} }.
\end{align}
In the case where $\lambda_i=\lambda_0$, we simplify the denominator using that $e^{\left(\lambda_0-\lambda_i^{(n)}\right)\mathfrak{t}^{(n)}_{t-t(i)+h_{n}^{-1}(i)}} \geq 1$. Then, we apply Doob's inequality to the martingale $\big(M_{i}^{(n)}\big(\mathfrak{t}_{t}^{(n)}+s\big)-M_{i}^{(n)}\big(\mathfrak{t}_{t(i)-h_n^{-1}(i)}^{(n)}+s\big)\big)_{t \geq t(i)}$, and use the property that if $M$ is a square integrable martingale, then $\E[(M(t)-M(s))^2]=\E[M^2(t)-M^2(s)]= \E[\langle M\rangle_t-\langle M \rangle_s]$. It follows that
\begin{align}
    \eqref{Eq: eq1} &\leq \mP \Big(\sup_{t \in \left[t(i)-h_n^{-1}(i),t(i)\right]} \Big{\vert} \frac{M_{i}^{(n)}\big( \mathfrak{t}^{(n)}_{t}+s\big)-M_{i}^{(n)}\big(\mathfrak{t}^{(n)}_{t(i)-h_n^{-1}(i)}+s\big)}{\psi_n(i) e^{-\varphi_{n}(i)} e^{-\lambda_i^{(n)}\big(\mathfrak{t}^{(n)}_{t(i)-h_{n}^{-1}(i)}+s\big)}}\Big{\vert} \geq \varepsilon \Big) \\
    & \leq \frac{4 e^{2\lambda_i \mathfrak{t}^{(n)}_{t(i)-h_{n}^{-1}(i)}}e^{2\lambda_i^{(n)} s}}{\varepsilon^2 \psi^{2}_n(i) e^{-2\varphi_{n}(i)}}  \E \Big[\big\langle M_{i}^{(n)}\big\rangle_{\mathfrak{t}^{(n)}_{t(i)}+s}- \big\langle M_{i}^{(n)}\big\rangle_{\mathfrak{t}^{(n)}_{t(i)-h_{n}^{-1}(i)}+s}\Big].
\end{align}
Applying Lemma \ref{Lem: control quadratic variation} at times $t_{1}^{(n)}=t(i)-h_n^{-1}(i)$ and $t_2=t(i)$, there exists a constant $C=C(s,i,\varepsilon)$ (which may change from line to line) such that  
\begin{align}
    \eqref{Eq: eq1} \leq \frac{C e^{2\varphi_{n}(i)}}{ \psi_n^{2}(i)} \big( n^{t(i)}\mu_{\otimes, i}^{(n)}\big) \big(\mathfrak{t}^{(n)}_{t(i)-h_{n}^{-1}(i)}+s\big)^{\theta(i)}n^{-h_{n}^{-1}(i)}.\end{align}
Note that 
\begin{align} \label{Equ: mut proba scaling}
    n^{t(i)}\mu_{\otimes,i}^{(n)}=\prod_{j=1}^{i} n^{\ell(j-1)}\mu_{j-1}^{(n)}\underset{n \to \infty}{\longrightarrow}\prod_{j=1}^{i}\mu_{j-1}<\infty.
\end{align}
Then, for $n$ large enough, and recalling that $\theta(i)=\theta(i-1)+1$, we have 
\begin{align}
    \eqref{Eq: eq1} \leq C\frac{\log(n) e^{\varphi_{n}(i)}}{\psi_{n}^{2}(i)} \underset{n \to \infty}{\longrightarrow}0,
\end{align}
according to the scaling of $\psi_n(i)$. In the case where $\lambda_{i}< \lambda_0$, using the Maximal inequality (see \cite[Chapter VI, page 72]{delacherie}) applied to the supermartingale $$\Big[\frac{M_{i}^{(n)}\big(\mathfrak{t}^{(n)}_{t}+s\big)-M_{i}^{(n)}\big(\mathfrak{t}^{(n)}_{t(i)-h_n^{-1}(i)}+s\big)}{\psi_n(i)e^{-\varphi_{n}(i)}  e^{\left(\lambda_0-\lambda_i^{(n)}\right)\mathfrak{t}^{(n)}_{t-t(i)+h_{n}^{-1}(i)}}e^{-\lambda_i^{(n)}\big(\mathfrak{t}^{(n)}_{t(i)-h_{n}^{-1}(i)}+s\big)} }\Big]_{t\geq t(i)-h_n^{-1}(i)},$$ it follows that

\begin{align}\label{Equ: max inequality}
    \eqref{Eq: eq1} \leq \frac{3 e^{\lambda_i^{(n)}\big(\mathfrak{t}^{(n)}_{t(i)-h_{n}^{-1}(i)}+s\big)}}{\varepsilon \psi_n(i) e^{-\varphi_{n}(i)}}\sup_{t \in \left[t(i)-h_n^{-1}(i),t(i)\right]}f^{(n)}(t),
\end{align}
where $f^{(n)}(t):=e^{-\left(\lambda_0-\lambda_i^{(n)}\right)\mathfrak{t}^{(n)}_{t-t(i)+h_{n}^{-1}(i)}}\E\big[\big\langle M_{i}^{(n)}\big\rangle_{\mathfrak{t}^{(n)}_{t}+s}-\big\langle M_{i}^{(n)}\big\rangle_{\mathfrak{t}^{(n)}_{t(i)-h_n^{-1}(i)}+s}\big]^{\frac{1}{2}}.$ According to Lemma \ref{Lem: control quadratic variation} applied with $t_{1}^{(n)}=t(i)-h_n^{-1}(i)$ and $t_2=t(i)$, we have that 

\begin{align}\label{Equ: fn(t)}
    &\sup_{t \in \left[t(i)-h_n^{-1}(i),t(i)\right]}f^{(n)}(t) \leq C  \big(\mu_{\otimes,i}^{(n)}  \big)^{\frac{1}{2}} \big(\mathfrak{t}^{(n)}_{t(i)}+s\big)^{\frac{\theta(i)}{2}} e^{\left(\lambda_0-\lambda_i^{(n)}\right) \mathfrak{t}^{(n)}_{t(i)-h_{n}^{-1}(i)}} \\
    &\hspace{1cm}\cdot\Big{(}\1_{\{\lambda_0>2\lambda_i\}} e^{-\frac{\lambda_0}{2}\mathfrak{t}^{(n)}_{t(i)}}e^{\frac{\lambda_0-2\lambda_i}{2}s}+\1_{\{\lambda_0=2\lambda_i\}}\big(\mathfrak{t}^{(n)}_{t(i)}+s\big)^{\frac{1}{2}}e^{-\left(\lambda_0-\lambda_i^{(n)}\right) \mathfrak{t}^{(n)}_{t(i)-h_{n}^{-1}(i)}}\\
    &\hspace{1cm}+\1_{\{\lambda_i<\lambda_0<2\lambda_i\}}e^{-\frac{\lambda_0}{2}\mathfrak{t}^{(n)}_{t(i)-h_{n}^{-1}(i)}}e^{-\frac{2\lambda_i-\lambda_0}{2}s}\Big{)}.
\end{align}
Combining \eqref{Equ: max inequality}, \eqref{Equ: fn(t)}, \eqref{Equ: mut proba scaling}, and using the facts that $e^{\lambda_0 \mathfrak{t}^{(n)}_{t(i)}}=n^{t(i)}$, $e^{\lambda_0 \mathfrak{t}^{(n)}_{h_{n}^{-1}(i)}}= \log^{\theta(i-1)}(n)e^{\varphi_{n}(i)}$ and $\theta(i-1)=\theta(i)$, it follows that 
\begin{align}
    \eqref{Eq: eq1} &\leq \frac{C e^{\varphi_{n}(i)}}{\varepsilon \psi_{n}(i)} e^{\lambda_i^{(n)}s} \big(n^{t(i)}\mu_{\otimes,i}^{(n)}\big)^{\frac{1}{2}} \big( \mathfrak{t}^{(n)}_{t(i)}+s\big)^{\frac{\theta_{i}}{2}}  e^{\frac{\lambda_0}{2}\mathfrak{t}^{(n)}_{t(i)}}e^{-\lambda_0 \mathfrak{t}^{(n)}_{h_{n}^{-1}(i)}} \\
    &\cdot \Big{(}\1_{\{\lambda_0>2\lambda_i\}} e^{-\frac{\lambda_0}{2}\mathfrak{t}^{(n)}_{t(i)}}e^{\frac{\lambda_0-2\lambda_i}{2}s}+\1_{\{\lambda_0=2\lambda_i\}}\big(\mathfrak{t}^{(n)}_{t(i)}+s\big)^{\frac{1}{2}}e^{-\left(\lambda_0-\lambda_i\right) \mathfrak{t}^{(n)}_{t(i)-h_{n}^{-1}(i)}}\\
    &\hspace{4cm}+\1_{\{\lambda_i<\lambda_0<2\lambda_i\}}e^{-\frac{\lambda_0}{2}\mathfrak{t}^{(n)}_{t(i)-h_{n}^{-1}(i)}}e^{-\frac{2\lambda_i-\lambda_0}{2}s}\Big{)} \\
    &\leq \frac{C e^{\varphi_{n}(i)}}{\varepsilon \psi_{n}(i)}\log^{\frac{\theta(i)}{2}}(n) \Big{(} \1_{\{\lambda_0>2\lambda_i\}} \frac{e^{-\varphi_{n}(i)} }{\log^{\theta(i-1)}(n)} +\1_{\{\lambda_0=2\lambda_i\}} \log^{\frac{1}{2}}(n) e^{-\frac{\lambda_0}{2}\mathfrak{t}^{(n)}_{h_{n}^{-1}(i)}} \\
    &\hspace{4cm}+\1_{\{\lambda_i<\lambda_0<2\lambda_i\}} e^{-\frac{\lambda_0}{2}\mathfrak{t}^{(n)}_{h_{n}^{-1}(i)}} \Big{)} \\
    &\leq \frac{C}{\varepsilon \psi_{n}(i)} \Big( \1_{\{\lambda_0>2\lambda_i\}} \frac{1}{\log\frac{\theta(i)}{2}(n)}+\1_{\{\lambda_0=2\lambda_i\}} \sqrt{\log(n)} e^{\frac{\varphi_{n}(i)}{2}}+\1_{\{\lambda_1<\lambda_0<2\lambda_i\}}e^{\frac{\varphi_{n}(i)}{2}} \Big) \\
    &\underset{n \to \infty}{\longrightarrow}0,
\end{align}
according to the scaling of $\psi_{n}(i)$. 

\textbf{Step 2:} Let $\psi_{n}(i)$ satisfy $\sqrt{\log(n)}=o(\psi_{n}(i))$, but such that $\log(n) e^{\varphi_{n}(i)} \neq o(\psi_{n}^{2}(i))$. Let $\widetilde{\varphi}_n(i)$ be such that $\log(n)e^{\widetilde{\varphi}_n(i)}=o(\psi_{n}^2(i))$, and define $\widetilde{h}_{n}(i):=\frac{\log(n)}{\log(\log(n))\theta(i-1)+\widetilde{\varphi}_n(i)}$. Notice, in particular, that $\widetilde{h}_{n}(i) \geq h_{n}(i)$. We have 
\begin{align}
    &\mP \Big( \sup_{t \in \left[t(i)-h_{n}^{-1}(i),t(i)\right]} \frac{Z_{i}^{(n)}\big(\mathfrak{t}^{(n)}_{t}+s\big)}{\psi_{n}(i)\log^{\theta(i-1)}(n)} \geq \varepsilon\Big) \\
    &\hspace{2cm}\leq \mP \Big( \sup_{t \in \left[t(i)-h_n^{-1}(i),t(i)-\widetilde{h}_{n}^{-1}(i)\right]} Z_{i}^{(n)}\big(\mathfrak{t}^{(n)}_{t}+s\big) \geq \varepsilon\Big)\\
    &\hspace{2cm}+\mP \Big( \sup_{t \in \left[t(i)-\widetilde{h}_{n}^{-1}(i),t(i)\right]} \frac{Z_{i}^{(n)}\big(\mathfrak{t}^{(n)}_{t}+s\big)}{\psi_{n}(i)\log^{\theta(i-1)}(n)} \geq \varepsilon\Big),
\end{align}
where the first term on the right-hand side converges to 0 according to Lemma \ref{Lem: no mutant cell}, and the second term converges from Step 1 of this proof. 
\end{proof}
\textbf{(iii) Third time regime $(t \in [t(i),T])$, convergence to 0 of \eqref{Eq:6}:} Applying similar computations as in \eqref{Eq: calcul intercale}, notice that for all $t \geq t(i)$
\begin{align}
    &\frac{Z_{i}^{(n)}\big( \mathfrak{t}^{(n)}_{t}+s\big)}{n^{t-t(i)}\log^{\theta(i)}(n)e^{\lambda_0 s}}
    =\frac{M_{i}^{(n)}\big( \mathfrak{t}^{(n)}_{t}+s\big)-M_{i}^{(n)}\big(\mathfrak{t}^{(n)}_{t(i)}+s\big)}{n^{-t(i)}\log^{\theta(i)}(n)e^{\left(\lambda_0-\lambda_{i}^{(n)}\right)\left(\mathfrak{t}^{(n)}_{t}+s\right)}}\\
    &\hspace{2cm}+\frac{\int_{\mathfrak{t}^{(n)}_{t(i)}+s}^{\mathfrak{t}^{(n)}_{t}+s}2\alpha_{i-1}\mu_{i-1}^{(n)}Z_{i-1}^{(n)}(u)e^{-\lambda_{i}^{(n)}u}du}{n^{-t(i)}\log^{\theta(i)}(n)e^{\left(\lambda_0-\lambda_{i}^{(n)}\right)\left(\mathfrak{t}^{(n)}_{t}+s\right)}}+\frac{Z_{i}^{(n)}\big(\mathfrak{t}^{(n)}_{t(i)}+s\big)}{\log^{\theta(i)}(n)e^{\left(\lambda_0-\lambda_{i}^{(n)}\right)\mathfrak{t}^{(n)}_{t-t(i)}}e^{\lambda_0 s}}.
\end{align}
Then this allows to write 
\begin{align}
\label{Equation102}
&\mP \Big( \sup_{t \in \left[t(i),T\right]} \Big{\vert}\frac{Z_{i}^{(n)}\big( \mathfrak{t}^{(n)}_{t}+s\big)}{n^{t-t(i)}\log^{\theta(i)}(n)e^{\lambda_0 s}}-W w_{i}(t) \Big{\vert} \geq 3\varepsilon\Big)\\
\label{Eq:8}&\hspace{2cm}\leq \mP \Big( \sup_{t \in \left[t(i),T\right]} \Big{\vert}\frac{M_{i}^{(n)}\big( \mathfrak{t}^{(n)}_{t}+s\big)-M_{i}^{(n)}\big(\mathfrak{t}^{(n)}_{t(i)}+s\big)}{n^{-t(i)}\log^{\theta(i)}(n)e^{\left(\lambda_0-\lambda_{i}^{(n)}\right)\left(\mathfrak{t}^{(n)}_{t}+s\right)}} \Big{\vert} \geq \varepsilon\Big)\\
\label{Eq:9}&\hspace{2cm}+ \mP\Big(\sup_{t \in \left[t(i),T\right]}\Big{\vert}\frac{\int_{\mathfrak{t}^{(n)}_{t(i)}+s}^{\mathfrak{t}^{(n)}_{t}+s}2\alpha_{i-1}\mu_{i-1}^{(n)}Z_{i-1}^{(n)}(u)e^{-\lambda_{i}^{(n)}u}du}{n^{-t(i)}\log^{\theta(i)}(n)e^{\left(\lambda_0-\lambda_{i}^{(n)}\right)\left(\mathfrak{t}^{(n)}_{t}+s\right)}} -Ww_{i}(t)\Big{\vert}\geq \varepsilon \Big)\\
\label{Eq:10}&\hspace{2cm}+\mP\Big( \sup_{t \in \left[t(i),T\right]}\frac{Z_{i}^{(n)}\big(\mathfrak{t}^{(n)}_{t(i)}+s\big)}{\log^{\theta(i)}(n)e^{\left(\lambda_0-\lambda_{i}^{(n)}\right)\mathfrak{t}^{(n)}_{t-t(i)}}e^{\lambda_0 s}} \geq \varepsilon\Big).
\end{align}
We will show that \eqref{Eq:8}, \eqref{Eq:9} and \eqref{Eq:10} converge to $0$ as $n$ goes to infinity. For the term \eqref{Eq:8}, we start by using the fact that $\lambda_0 \geq \lambda_{i}^{(n)}$ to simplify the denominator. Then, we apply Doob's inequality to the martingale $\big(M_{i}^{(n)}\big(\mathfrak{t}_{t}^{(n)}+s\big)-M_{i}^{(n)}\big(\mathfrak{t}_{t(i)}^{(n)}+s)\big)\big)_{t \geq t(i)}$ to obtain 
\begin{align}
\label{Eq 15}
        \eqref{Eq:8}&\leq \mP \Big(\sup_{t \in \left[t(i),T\right]}\Big{\vert} \frac{M_{i}^{(n)}\big(\mathfrak{t}^{(n)}_{t}+s \big)-M_{i}^{(n)}\big(\mathfrak{t}^{(n)}_{t(i)}+s\big)}{n^{-t(i)}\log^{\theta(i)}(n)}\Big{\vert} \geq \varepsilon\Big) \\
        &\leq\frac{4n^{2t(i)}}{\varepsilon^{2}\log^{2\theta(i)}(n)} \E\Big[\big\langle M_{i}^{(n)}  \big\rangle_{\mathfrak{t}^{(n)}_{T}+s}-\big\langle M_{i}^{(n)}\big\rangle_{\mathfrak{t}^{(n)}_{t(i)}+s}\Big]. 
\end{align}
By applying Lemma \ref{Lem: control quadratic variation} at times $t_1^{(n)}=t(i)$ and $t_2=T$, we obtain 
\begin{align}
\label{Eq: eq 60}
    &\E\Big[\big\langle M_{i}^{(n)}  \big\rangle_{\mathfrak{t}^{(n)}_{T}+s}-\big\langle M_{i}^{(n)}\big\rangle_{\mathfrak{t}^{(n)}_{t(i)}+s}\Big]\leq C\frac{e^{-\lambda_0 s}\big(\mathfrak{t}^{(n)}_{t(i)}+s\big)^{\theta(i)} \mu_{\otimes,i}^{(n)}}{ n^{t(i)}}.
\end{align}
Then, combining \eqref{Eq 15} and \eqref{Eq: eq 60}, we get
\begin{align}
\label{Eq 16}
        &\mP \Big(\sup_{t \in \left[t(i),T\right]}\Big{\vert} \frac{M_{(i}^{(n)}\big(\mathfrak{t}^{(n)}_{t}+s\big)-M_{i}^{(n)}\big(\mathfrak{t}^{(n)}_{t(i)}+s\big)}{n^{-t(i)}\log^{\theta(i)}(n)e^{\left(\lambda_0-\lambda_{i}^{(n)}\right)\left(\mathfrak{t}^{(n)}_{t}+s\right)}}\Big{\vert} \geq \varepsilon \Big) \\
        &\hspace{2cm}\leq \frac{4Ce^{-\lambda_0 s}}{\varepsilon^{2}\log^{\theta(i)}(n)} \Big(\frac{\mathfrak{t}^{(n)}_{t(i)}+s}{\log(n)}\Big)^{\theta(i)} n^{t(i)}\mu_{\otimes,i}^{(n)}\\
        &\underset{n \to \infty}{\longrightarrow}0,
\end{align}
as $\theta(i) \geq 1$, since the vertex $i$ is assumed to be neutral. This concludes the proof of the convergence to $0$ of \eqref{Eq:8}. The term \eqref{Eq:9} also converges to $0$, as shown in the following lemma. 
\begin{lemma}
\label{convergence finite variation process K}
Let $i\in \mathbb{N}$, $T \geq t(i)$, $s \in \mathbb{R}$ and $\varepsilon>0$. For $i \geq 2$, we prove that if Proposition \ref{Proposition: mutant pop deterministic time scale not uniform in s} holds for $i-1$, then 
\begin{align}
\label{Eq:17}
    \mP \Big(\sup_{t \in \left[t(i),T\right]} \Big\vert\frac{\int_{\mathfrak{t}^{(n)}_{t(i)}+s}^{\mathfrak{t}^{(n)}_{t}+s} 2\alpha_{i-1}\mu_{i-1}^{(n)} e^{-\lambda_{i}^{(n)} u}Z_{i-1}^{(n)}(u)du}{n^{-t(i)}\log^{\theta(i)}(n) e^{\left(\lambda_0-\lambda_{i}^{(n)}\right) \left(\mathfrak{t}^{(n)}_{t}+s\right)}} -Ww_{i}(t)\Big\vert\geq \varepsilon\Big)\underset{n \to \infty}{\longrightarrow}0.
\end{align}
For $i=1$, we prove \eqref{Eq:17} without any conditions.
\end{lemma}

\begin{proof}[Proof of Lemma \ref{convergence finite variation process K}]
Let $c_n(t,s):=e^{\left(\lambda_0-\lambda_{i}^{(n)}\right)\left(\mathfrak{t}^{(n)}_{t}+s\right)}$ and $$a_{t}^{(n)}:=\frac{\int_{\mathfrak{t}^{(n)}_{t(i)}+s}^{\mathfrak{t}^{(n)}_{t}+s} 2\alpha_{i-1}\mu_{i-1}^{(n)} e^{-\lambda_{i}^{(n)} u}Z_{i-1}^{(n)}(u)du}{n^{-t(i)}\log^{\theta(i)}(n) c_n(t,s)}.$$
Our aim is to prove that, for all $\varepsilon>0$, $\mP \big( \sup_{t \in \left[t(i),T\right]} \Big{\vert}a_{t}^{(n)}-W w_{i}(t)\Big{\vert} \leq \varepsilon\big) \underset{n \to \infty}{\to}1.$

\textbf{Case $i=1$:} We have 
\begin{align}
\label{Eq: eq 510}
    a_{t}^{(n)}&=\frac{n^{t(1)}}{\log^{\theta(1)}(n) c_n(t,s)} \int_{\mathfrak{t}^{(n)}_{t(1)}+s}^{\mathfrak{t}^{(n)}_{t}+s} 2\alpha_{0}\mu_{0}^{(n)} \Big[W+\left(e^{-\lambda_0 u}Z_{0}(u)-W\right)\\
    &\hspace{5cm}+\left(e^{-\lambda_{0}^{(n)}u}Z_{0}^{(n)}(u)-e^{-\lambda_0 u}Z_{0}(u)\right)\Big]e^{\left(\lambda_{0}^{(n)}-\lambda_{1}^{(n)}\right)u}du.
\end{align}
For all $\varepsilon>0$, introduce the events
\begin{align}
    &B_{\widetilde{\varepsilon}}^{(n)}:= \Big\{ \sup_{u \in \mathbb{R}^{+}} \Big{\vert} e^{-\lambda_0 u}Z_{0}(u) -e^{-\lambda_{0}^{(n)}u}Z_{0}^{(n)}(u)\Big{\vert} \leq \widetilde{\varepsilon}\Big\}, \\
    &C_{x,\widetilde{\varepsilon}}:= \Big\{ \sup_{u \in [x,\infty]} \vert e^{-\lambda_0 u}Z_{0}(u)-W\vert \leq \widetilde
    \varepsilon\Big\}.
\end{align}
According to Lemma \ref{Lem:control primary pop} and the definition of $W$ (see \eqref{W def}), we have both that $\mP \big( B_{\widetilde{\varepsilon}}^{(n)}\big) \underset{n \to \infty}{\to}1$ and $\mP \big( C_{\sqrt{\log(n)},\widetilde{\varepsilon}}\big) \underset{n \to \infty}{\to}1$. Notice that when $\lambda_1<\lambda_0$, we have the following bound
\begin{align}
\label{Eq: eq 50}
    \frac{1}{c_n(t,s)}\int_{\mathfrak{t}^{(n)}_{t(1)}+s}^{\mathfrak{t}^{(n)}_{t}+s}e^{\left(\lambda_0-\lambda_1^{(n)}\right)u}du =\frac{1}{\lambda_0-\lambda_1^{(n)}}\frac{c_n(t,s)-c_n(t(1),s)}{c_n(t,s)} \leq \frac{1}{\lambda_0-\lambda_1},
\end{align}
and when $\lambda_1=\lambda_0$, we have
\begin{align}
\label{Eq: cn neutral case}
    \frac{1}{c_n(t,s)}\int_{\mathfrak{t}^{(n)}_{t(1)}+s}^{\mathfrak{t}^{(n)}_{t}+s}e^{\left(\lambda_0-\lambda_1^{(n)}\right)u}du =\frac{1-e^{-\left(\lambda_0-\lambda_1^{(n)}\right)\mathfrak{t}^{(n)}_{t-t(1)}}}{\lambda_0-\lambda_1^{(n)}} \leq \mathfrak{t}^{(n)}_{t-t(1)},
\end{align}
where for the last inequality, we use \eqref{Eq: eq 54} applied with $a=\lambda_0-\lambda_1^{(n)}=2\alpha_1 \mu_1^{(n)}>0$ and $x=\mathfrak{t}^{(n)}_{t-t(1)}$. It follows that, for sufficiently large $n$ (such that $\mathfrak{t}^{(n)}_{t(1)}+s \geq \sqrt{\log(n)}$), under the event $B_{\widetilde{\varepsilon}}^{(n)} \cap C_{\sqrt{\log(n)},\widetilde{\varepsilon}}$, we have that
\begin{align}
\label{Eq: eq 52}
    a_{t}^{(n)}&\leq \frac{n^{t(1)}}{\log^{\theta(1)}(n)c_n(t)} \int_{\mathfrak{t}^{(n)}_{t(1)}+s}^{\mathfrak{t}^{(n)}_{t}+s} 2\alpha_{0}\mu_{0}^{(n)}  \big( W+2\widetilde{\varepsilon}\big)e^{\left(\lambda_{0}-\lambda_{1}^{(n)}\right)u}du \\
    &\leq 2 \alpha_0\big(n^{t(1)}\mu_{0}^{(n)}\big)(W+2\widetilde{\varepsilon}) \Big(\1_{\{\lambda_1<\lambda_0\}}\frac{1}{\lambda_0-\lambda_1}+\1_{\{\lambda_1=\lambda_0\}}\frac{1}{\lambda_0} (t-t(1))\Big),
\end{align}
since $\theta(1)=\1_{\{\lambda_1=\lambda_0\}}.$ By definition, we have $$w_{1}(t)=2\alpha_{0}\mu_{0} \big(\1_{\{\lambda_1<\lambda_0\}}\frac{1}{\lambda_0-\lambda_1}+\1_{\{\lambda_1=\lambda_0\}}\frac{1}{\lambda_0}(t-t(1))\big).$$ This implies that 
\begin{align}
    a_{t}^{(n)}-Ww_{1}(t) \leq \frac{w_1(t)}{\mu_0}W \big(n^{t(1)}\mu_{0}^{(n)}-\mu_0\big)+C\widetilde{\varepsilon},
\end{align}
where $C>0$ is a sufficiently large constant.

Introduce the event $$D_{\widetilde{\varepsilon}}^{(n)}:= \big\{ \underset{t \in [t(1),T]}{\sup} \big{\vert}\frac{w_1(t)}{\mu_0}W \big(n^{t(1)}\mu_{0}^{(n)}-\mu_0\big)\big{\vert}\leq \widetilde{\varepsilon}\big\}.$$ This event satisfies $\mP \big(D_{\widetilde{\varepsilon}}^{(n)}\big)\underset{n \to \infty}{\to}1$ because $W$ is finite almost surely, $n^{t(1)}\mu_{0}^{(n)} \underset{n \to \infty}{\to} \mu_0$ and $w_{1}(t)$ is bounded from above on $[t(1),T]$. Under $B_{\widetilde{\varepsilon}}^{(n)}\cap C_{\sqrt{\log(n)},\widetilde{\varepsilon}}\cap D_{\widetilde{\varepsilon}}^{(n)}$, we have for all $t \in [t(1),T]$,
\begin{align}
    a_{t}^{(n)}-Ww_{1}(t) \leq (C+1)\widetilde{\varepsilon}.
\end{align}
Similarly, it follows that under $B_{\widetilde{\varepsilon}}^{(n)}\cap C_{\sqrt{\log(n)},\widetilde{\varepsilon}}\cap D_{\widetilde{\varepsilon}}^{(n)}$
\begin{align}
    \sup_{t \in [t(1),T]}\vert a_{t}^{(n)}-Ww_{1}(t) \vert \leq (C+1)\widetilde{\varepsilon}.
\end{align}
By choosing $\widetilde{\varepsilon}>0$ such that $(C+1)\widetilde{\varepsilon}\leq \varepsilon$, we deduce that under the event $B_{\widetilde{\varepsilon}}^{(n)}\cap C_{\sqrt{\log(n)},\widetilde{\varepsilon}}\cap D_{\widetilde{\varepsilon}}^{(n)}$,
\begin{align}
    \sup_{t \in [t(1),T]}\vert a_{t}^{(n)}-Ww_1(t) \vert \leq \varepsilon.
\end{align}
This concludes the proof for the case $i=1$, since $\mP \big(B_{\widetilde{\varepsilon}}^{(n)}\cap C_{\sqrt{\log(n)},\widetilde{\varepsilon}}\cap D_{\widetilde{\varepsilon}}^{(n)}\big) \underset{n \to \infty}{\longrightarrow}1$.

\textbf{Case $i \geq 2$:}
Assume Proposition \ref{Proposition: mutant pop deterministic time scale not uniform in s} holds for $i-1$. In particular, we have $\mP \big(B_{\widetilde{\varepsilon}}^{(n)}\big) \underset{n \to \infty}{\longrightarrow}1$, where
$$B_{\widetilde{\varepsilon}}^{(n)}:= \Big\{ \sup_{v \in [t(i),T]} \Big{\vert} \frac{Z_{i-1}^{(n)}\big(\mathfrak{t}^{(n)}_{v}+s\big)}{n^{v-t(i-1)}e^{\lambda_0 s}\log^{\theta(i-1)}(n)}-W w_{i-1}(v)\Big{\vert} \leq \widetilde{\varepsilon}\Big\}.$$
Using the change of variable $u=\mathfrak{t}^{(n)}_{v}+s$ and the fact that $t(i-1)=t(i)-\ell(i-1)$ we obtain
\begin{align}
    a_{t}^{(n)}=\int_{t(i)}^{t}2\alpha_{i-1}\big(n^{\ell(i-1)}\mu_{i-1}^{(n)}\big)\frac{Z_{i-1}^{(n)}\big(\mathfrak{t}^{(n)}_{v}+s\big)}{n^{v-t(i-1)}e^{\lambda_0 s}\log^{\theta(i)}(n)} \frac{c_n(v,s) }{c_n(t,s) }\frac{\log(n)}{\lambda_0}dv.
\end{align}
Notice that when $\lambda_i<\lambda_0$, we have $\theta(i-1)=\theta(i)$, and  when $\lambda_i=\lambda_0$, we have $\theta(i-1)=\theta(i)-1$. Additionally, we use that $v \mapsto c_n(v,s)$ and $w_{i-1}$ are non-decreasing functions, and we apply similar computations as in \eqref{Eq: eq 50} and \eqref{Eq: cn neutral case}, replacing the index $1$ with $i$ to find, under $B_{\widetilde{\varepsilon}}^{(n)}$, that
\begin{align}
    a_t^{(n)} &\leq 2 \alpha_{i-1} \big(n^{\ell(i-1)}\mu_{i-1}^{(n)}\big)\\
    &\hspace{1cm}\cdot\Big[\1_{\{\lambda_i<\lambda_0\}}\frac{Ww_{i-1}(t)+\widetilde{\varepsilon}}{\lambda_0-\lambda_i}+\1_{\{\lambda_i=\lambda_0\}}W\frac{1}{\lambda_0}\int_{t(i)}^{t}\left(w_{i-1}(v)+\widetilde{\varepsilon}\right)dv \Big].
\end{align} 
By definition (see \eqref{weight} and Remark \ref{Remark:recursive structure weight path}), we have
\begin{align}
    w_{i}(t)=2 \alpha_{i-1}\mu_{i-1}\Big(\1_{\{\lambda_i<\lambda_0\}}\frac{w_{i-1}(t)}{\lambda_0-\lambda_i}+\1_{\{\lambda_i=\lambda_0\}}\frac{1}{\lambda_0}\int_{t(i)}^{t}w_{i-1}(u)du\Big).
\end{align}
Thus, under the event $C_{\widetilde{\varepsilon}}^{(n)}:= \left\{  W \vert n^{\ell(i-1)}\mu_{i-1}^{(n)}-\mu_{i-1}\vert \leq \widetilde{\varepsilon}\right\}$, we find that for all $t \leq T$
\begin{align}
    a_{t}^{(n)} -Ww_{i}(t)&\leq 2\alpha_{i-1}\Big[\1_{\{\lambda_i<\lambda_0\}}\frac{1}{\lambda_0-\lambda_1}\big(w_{i-1}(T)+\big(n^{\ell(i-1)}\mu_{i-1}^{(n)}\big)\big)\\
    &\hspace{2cm}+\1_{\{\lambda_i=\lambda_0\}}\frac{1}{\lambda_0}\Big(\int_{t(i)}^{T}w_{i-1}(u)du+T \big(n^{\ell(i-1)}\mu_{i-1}^{(n)}\big)\Big) \Big]\widetilde{\varepsilon}\\
    &\leq C\widetilde{\varepsilon},
\end{align}
where $C$ is a positive constant that depends only on the parameters and on $T$, but is independent of $n$. Recalling that $n^{\ell(i-1)}\mu^{(n)}_{i-1}$ converges and that $W$ is finite almost surely (see \eqref{Eq:distribution W}) we obtain that the event $C_{\widetilde{\varepsilon}}^{(n)}$ satisfies $\mP \big(C_{\widetilde{\varepsilon}}^{(n)} \big) \underset{n \to \infty}{\longrightarrow}1$.

Then, by choosing $\widetilde{\varepsilon}>0$ such that $C\widetilde{\varepsilon}\leq \varepsilon$, we have shown that under $B_{\widetilde{\varepsilon}}^{(n)}\cap C_{\widetilde{\varepsilon}}^{(n)}$,
\begin{align}
    \sup_{t \in [t(i),T]} a_{t}^{(n)}-Ww_{i}(t) \leq \varepsilon.
\end{align}
With similar computations, it can also be shown that under $B_{\widetilde{\varepsilon}}^{(n)}\cap C_{\widetilde{\varepsilon}}^{(n)}$,
\begin{align}
    \sup_{t \in [t(i),T]} \vert a_{t}^{(n)}-Ww_{i}(t) \vert \leq \varepsilon.
\end{align}
We conclude the proof by noting that $\mP \big(C_{\widetilde{\varepsilon}}^{(n)} \big) \underset{n \to \infty}{\to}1$ and $\mP \big(B_{\widetilde{\varepsilon}}^{(n)}\big)\underset{n \to \infty}{\to}1$, as established by the induction assumption.
\end{proof}
 
Given that $\lambda_{0} \geq \lambda_{i}^{(n)}$, the term \eqref{Eq:10} satisfies 
\begin{align}
    \mP \Big( \sup_{t \in \left[t(i),T\right]}\frac{Z_{i}^{(n)} \big( \mathfrak{t}^{(n)}_{t(i)}+s\big)}{\log^{\theta(i)}(n)e^{\left(\lambda_0-\lambda_{i}^{(n)}\right)\mathfrak{t}^{(n)}_{t-t(i)}}e^{\lambda_0 s}} \geq \varepsilon\Big) &=\mP \Big(  \frac{Z_{i}^{(n)} \big( \mathfrak{t}^{(n)}_{t(i)}+s\big)}{\log^{\theta(i)}(n)e^{\lambda_0 s}}\geq \varepsilon \Big)\underset{n \to \infty}{\longrightarrow}0,
\end{align}
where the convergence is obtained by applying Lemma \ref{Lem: control before first mut} with $\psi_n(i)=\log(n) e^{\lambda_0 s}$. This is valid because assuming the vertex $i$ is neutral implies that $\theta(i)=\theta(i-1)+1$. 

This completes the proof of Proposition \ref{Proposition: mutant pop deterministic time scale not uniform in s} (i).We now turn to Proposition \ref{Proposition: mutant pop deterministic time scale not uniform in s} (ii).

\textbf{2. Deleterious case:} Assume that $\lambda_{i}<\lambda_{0}$. Let $0<T_1<T_2$. Using similar computations as in \eqref{Eq: calcul intercale}, for all $t \in \left[T_1,T_2\right]$, we have 
\begin{align}
    \frac{Z_i^{(n)}\big(\mathfrak{t}^{(n)}_{t(i)+t}+s\big)}{n^t \log^{\theta(i)}(n)e^{\lambda_{0}s}}
    &=\frac{M^{(n)}_{i}\big(\mathfrak{t}^{(n)}_{t(i)+t}+s\big)-M^{(n)}_{i}\big(\mathfrak{t}^{(n)}_{t(i)}+s\big)}{n^{-\frac{\lambda_i^{(n)}}{\lambda_0}t(i)}e^{\left(\lambda_0-\lambda_i^{(n)}\right)\left(\mathfrak{t}^{(n)}_{t}+s\right)}\log^{\theta(i)}(n)} \\
    &+\frac{\int_{\mathfrak{t}^{(n)}_{t(i)}+s}^{\mathfrak{t}^{(n)}_{t(i)+t}+s} 2\alpha_{i-1} \mu_{i-1}^{(n)} e^{-\lambda_i^{(n)} s}Z^{(n)}_{i-1}(s)ds}{n^{-\frac{\lambda_i^{(n)}}{\lambda_0}t(i)}e^{\left(\lambda_0-\lambda_i^{(n)}\right)\left(\mathfrak{t}^{(n)}_{t}+s\right)}\log^{\theta(i)}(n)}+\frac{Z_i^{(n)}\big(\mathfrak{t}^{(n)}_{t(i)}+s\big)}{n^{t \frac{\lambda_0-\lambda_i^{(n)}}{\lambda_0}}e^{\lambda_0 s}\log^{\theta(i)}(n)}.\\
\end{align}
Then, this allows to write
\begin{align}
    &\mP \Big{(}\sup_{t \in \left[T_1,T_2\right]} \Big{\vert} \frac{Z_i^{(n)}\big{(}\mathfrak{t}^{(n)}_{t(i)+t}+s\big{)}}{n^t \log^{\theta(i)}(n)e^{\lambda_{0}s}}-Ww_{i}(t(i)+t) \Big{\vert} \geq 3\varepsilon \Big{)}\\ 
    &\hspace{1cm}\label{Eq: martingale part deleterious} \leq \mP \Big{(}\sup_{t \in \left[T_1,T_2\right]} \Big{\vert} \frac{M^{(n)}_{i}\big{(}\mathfrak{t}^{(n)}_{t(i)+t}+s\big{)}-M^{(n)}_{i}\big{(}\mathfrak{t}^{(n)}_{t(i)}+s\big{)}}{n^{-\frac{\lambda_i^{(n)}}{\lambda_0}t(i)}e^{\left(\lambda_0-\lambda_i^{(n)}\right)\left( \mathfrak{t}^{(n)}_{t}+s\right)}\log^{\theta(i)}(n)} \Big{\vert} \geq \varepsilon \Big{)} \\ 
    &\hspace{1cm} \label{Eq: finite variation deleterious}+ \mP \Big{(}\sup_{t \in \left[T_1,T_2\right]} \Big{\vert} \frac{\int_{\mathfrak{t}^{(n)}_{t(i)}+s}^{\mathfrak{t}^{(n)}_{t(i)+t}+s} 2\alpha_{i-1} \mu_{i-1}^{(n)} e^{-\lambda_i^{(n)} s}Z^{(n)}_{i-1}(s)ds}{n^{-\frac{\lambda_i^{(n)}}{\lambda_0}t(i)}n^{t\frac{\lambda_0-\lambda_i^{(n)}}{\lambda_0}}e^{\left(\lambda_0-\lambda_i^{(n)}\right)s}\log^{\theta(i)}(n)} -Ww_{i}(t(i)+t)\Big{\vert} \geq \varepsilon \Big{)} \\ 
    & \hspace{1cm}\label{Eq: finite number mutant deleterious}+ \mP \Big{(}  \frac{Z_i^{(n)}\big{(}\mathfrak{t}^{(n)}_{t(i)}+s\big{)}}{n^{T_1 \frac{\lambda_0-\lambda_i^{(n)}}{\lambda_0}}e^{\lambda_0 s}\log^{\theta(i)}(n)} \geq \varepsilon \Big{)}.
\end{align}
For the convergence to $0$ of the term \eqref{Eq: martingale part deleterious}, we first use the fact that $\lambda_i^{(n)} \leq \lambda_{i}<\lambda_0$, to simplify the denominator. Then, we apply the Maximal inequality to the supermartingale $$\Big(\frac{M_{i}^{(n)}\big(\mathfrak{t}^{(n)}_{t(i)+t}+s\big)-M_{i}^{(n)}\big(\mathfrak{t}^{(n)}_{t(i)}+s\big)}{n^{-\frac{\lambda_i}{\lambda_0}t(i)}e^{\left(\lambda_0-\lambda_i\right)\left(\mathfrak{t}^{(n)}_{t}+s\right)}\log^{\theta(i)}(n)}\Big)_{t \geq 0}$$ to obtain 
\begin{align}
\label{Eq: eq 9}
    &\hspace{1cm}\mP \Big(\sup_{t \in \left[T_1,T_2\right]} \Big{\vert} \frac{M_{i}^{(n)}\big(\mathfrak{t}^{(n)}_{t(i)+t}+s\big)-M_{i}^{(n)}\big(\mathfrak{t}^{(n)}_{t(i)}+s\big)}{n^{-\frac{\lambda_i^{(n)}}{\lambda_0}t(i)}e^{\left(\lambda_0-\lambda_i^{(n)}\right)\left(\mathfrak{t}^{(n)}_{t}+s\right)}\log^{\theta(i)}(n)} \Big{\vert} \geq \varepsilon \Big) \\
    &\hspace{1.5cm}\leq \mP \Big(\sup_{t \in \left[T_1,T_2\right]} \Big{\vert} \frac{M_{i}^{(n)}\big(\mathfrak{t}^{(n)}_{t(i)+t}+s\big)-M_{i}^{(n)}\big(\mathfrak{t}^{(n)}_{t(i)}+s\big)}{n^{-\frac{\lambda_i}{\lambda_0}t(i)}e^{\left(\lambda_0-\lambda_i\right)\left(\mathfrak{t}^{(n)}_{t}+s\right)}\log^{\theta(i)}(n)} \Big{\vert} \geq \varepsilon \Big) \\
    &\hspace{1.5cm}\leq \frac{3 e^{-(\lambda_0-\lambda_i)s}}{\varepsilon \log^{\theta(i)}(n)}n^{\frac{\lambda_i}{\lambda_0}t(i)}\sup_{t \in \left[T_1,T_2\right]} n^{-t \frac{\lambda_0-\lambda_i}{\lambda_0}}\sqrt{\E\Big[\big\langle M_{i}^{(n)}\big\rangle_{\mathfrak{t}^{(n)}_{t(i)+t}+s}-\big\langle M_{i}^{(n)}\big\rangle_{\mathfrak{t}^{(n)}_{t(i)}+s}\Big]}.
\end{align}
Applying Lemma \ref{Lem: control quadratic variation} at the times $t_1^{(n)}=t(i)$ and $t_2=t(i)+t$, we obtain 
\begin{align}
\label{Eq: eq 62}
    &\sqrt{\E\Big[\big\langle M_{i}^{(n)}\big\rangle_{\mathfrak{t}^{(n)}_{t(i)+t}+s}-\big\langle M_{i}^{(n)}\big\rangle_{\mathfrak{t}^{(n)}_{t(i)}+s}\Big]} \\
    &\hspace{1cm} \leq C\big(\mathfrak{t}^{(n)}_{t(i)+t}+s\big)^{\frac{\theta(i)}{2}} \sqrt{\mu_{\otimes,i}^{(n)}}\cdot \Big{[}\1_{\{\lambda_0>2\lambda_i\}}n^{\frac{\lambda_0-2\lambda_i}{2\lambda_0}(t(i)+t)} e^{\frac{\lambda_0-2\lambda_i}{2}s}\\
    &\hspace{2cm}+\1_{\{\lambda_0=2\lambda_i\}}\sqrt{\mathfrak{t}^{(n)}_{t(i)+t}+s}+\1_{\{\lambda_i<\lambda_0<2\lambda_i\}}n^{-\frac{2\lambda_i-\lambda_0}{2\lambda_0}t(i)}e^{-\frac{2\lambda_i-\lambda_0}{2}s}\Big{]}.
\end{align}
For all $t \in [T_1,T_2]$, we perform the following auxiliary computations, which will be used to obtain the result
\begin{align}
\label{Eq: eq 63}
    &n^{\frac{\lambda_i}{\lambda_0}t(i)} \sqrt{\mu_{\otimes,i}^{(n)}}n^{-t\frac{\lambda_0-\lambda_i}{\lambda_0}}n^{\frac{\lambda_0-2\lambda_i}{2\lambda_0}(t(i)+t)}\leq\frac{\sqrt{n^{t(i)}\mu_{\otimes,i}^{(n)}}}{n^{\frac{T_1}{2}}}, \\
    & \1_{\{\lambda_0=2\lambda_i\}}n^{\frac{\lambda_i}{\lambda_0}t(i)} \sqrt{\mu_{\otimes,i}^{(n)}}n^{-t\frac{\lambda_0-\lambda_i}{\lambda_0}} \leq \frac{\sqrt{n^{t(i)}\mu_{\otimes,i}^{(n)}}}{n^{\frac{T_1}{2}}},\\
    &n^{\frac{\lambda_i}{\lambda_0}t(i)} \sqrt{\mu_{\otimes,i}^{(n)}}n^{-t\frac{\lambda_0-\lambda_i}{\lambda_0}}n^{-\frac{2\lambda_i-\lambda_0}{2\lambda_0}t(i)}\leq \frac{\sqrt{n^{t(i)}\mu_{\otimes,i}^{(n)}}}{n^{T_1\frac{\lambda_0-\lambda_i}{\lambda_0}}}.
\end{align}
Then, combining \eqref{Eq: eq 9}, \eqref{Eq: eq 62} and \eqref{Eq: eq 63}, we obtain  
\begin{align}
    &\mP \Big(\sup_{t \in [T_1,T_2]} \Big{\vert} \frac{M_{i}^{(n)}\big(\mathfrak{t}^{(n)}_{t(i)+t}+s\big)-M_{i}^{(n)}\big(\mathfrak{t}^{(n)}_{t(i)}+s\big)}{n^{-\frac{\lambda_i^{(n)}}{\lambda_0}t(i)}e^{\left(\lambda_0-\lambda_i^{(n)}\right)\left(\mathfrak{t}^{(n)}_{t}+s\right)}\log^{\theta(i)}(n)} \Big{\vert} \geq \varepsilon \Big) \\
    & \leq \frac{3C e^{-(\lambda_0-\lambda_i)s}}{\varepsilon \log^{\theta(i)/2}(n)} \Big(\frac{\mathfrak{t}^{(n)}_{t(i)+T_2}+s}{\log(n)}\Big)^{\frac{\theta(i)}{2}} \sup_{t \in [T_1,T_2]} n^{\frac{\lambda_i}{\lambda_0}t(i)} n^{-t \frac{\lambda_0-\lambda_i}{\lambda_0}} \sqrt{\mu_{\otimes,i}^{(n)}}\Big{[}\1_{\{\lambda_0>2\lambda_i\}} n^{\frac{\lambda_0-2\lambda_i}{2\lambda_0}(t(i)+t)}\\
    &\hspace{2cm}\cdot e^{\frac{\lambda_0-2\lambda_i}{2}s}+\1_{\{\lambda_0=2\lambda_i\}}\sqrt{\mathfrak{t}^{(n)}_{t(i)+t}+s}+\1_{\{\lambda_i<\lambda_0<2\lambda_i\}}n^{-\frac{2\lambda_i-\lambda_0}{2\lambda_0}}e^{-\frac{2\lambda_i-\lambda_0}{2}s}\Big{]}\\
    & = \frac{3C e^{-(\lambda_0-\lambda_i)s}}{\varepsilon \log^{\theta(i)/2}(n)} \Big(\frac{\mathfrak{t}^{(n)}_{t(i)+T_2}+s}{\log(n)}\Big)^{\frac{\theta(i)}{2}} \sqrt{n^{t(i)}\mu_{\otimes,i}^{(n)}}\\
    &\hspace{2cm} \cdot \Big{(} \1_{\{\lambda_0>2\lambda_i\}}\frac{e^{\frac{\lambda_0-2\lambda_i}{2}s}}{n^{\frac{T_1}{2}}}+\1_{\{\lambda_0=2\lambda_i\}} \frac{\sqrt{\mathfrak{t}^{(n)}_{t(i)+T_2}+s}}{n^{\frac{T_1}{2}}} +\1_{\{\lambda_i<\lambda_0<2\lambda_i\}}\frac{e^{-\frac{2\lambda_i-\lambda_0}{2}s}}{n^{T_1 \frac{\lambda_0-\lambda_i}{\lambda_0}}}\Big{)}\\
    &\underset{n \to \infty}{\longrightarrow}0.
\end{align}
The term \eqref{Eq: finite variation deleterious} converges to $0$ by Lemma \ref{convergence finite variation process K}. The convergence to $0$ for the term \eqref{Eq: finite number mutant deleterious} is obtained by applying Lemma \ref{Lem: control before first mut} with $\psi_n(i)=n^{T_1 \frac{\lambda_0-\lambda_i}{\lambda_0}}e^{\lambda_0 s}$. This completes the proof of Proposition \ref{Proposition: mutant pop deterministic time scale not uniform in s}. 
\end{proof}

\subsubsection{Uniform control on the parameter s}
\label{Proof deterministic time neutral mutation}
In this subsection, we will prove \eqref{Equation: conv neutral case} and \eqref{Equation: convergence deleterious case} for the mono-directional graph, as stated in Proposition \ref{Proposition: mutant pop deterministic time scale not uniform in s}, using an approach inspired by \cite[Lemma 3]{foo2013dynamics}.
Define $u_{s}^{(n)}:=t+\frac{s-M}{\log(n)}\lambda_0$ such that $\mathfrak{t}^{(n)}_{t}+s= \mathfrak{t}^{(n)}_{u_{s}^{(n)}}+M$. Notice that 
\begin{align}
\label{Eq: eq 6}
    0 \leq t-u_{s}^{(n)} \leq \frac{2M}{\log(n)}\lambda_0.
\end{align}

\textbf{1. Deleterious case:} We begin by showing \eqref{Equation: convergence deleterious case}. We will use that 
\begin{align}
    n^{t}\log^{\theta(i)}(n)e^{\lambda_0 s}=n^{u_{s}^{(n)}}\log^{\theta(i)}(n)e^{\lambda_0 M}.
\end{align}
This gives that 
\begin{align}
    \Big{\lvert}  \frac{Z_{i}^{(n)}\big(\mathfrak{t}^{(n)}_{t(i)+t}+s\big)}{n^{t}\log^{\theta(i)}(n)e^{\lambda_0 s}}-W w_i(t(i)+t)\Big{\lvert} &\leq \Big{\vert} \frac{Z_{i}^{(n)}\big(\mathfrak{t}^{(n)}_{t(i)+u^{(n)}_{s}}+M\big)}{n^{u_{s}^{(n)}}\log^{\theta(i)}(n)e^{\lambda_0 M}}-W w_i\big(t(i)+u_s^{(n)}\big)\Big{\vert}\\
    &\hspace{1cm}+ W\Big{\vert}w_{i}(t(i)+t)-w_i\big(t(i)+u_s^{(n)}\big)\Big{\vert}.
\end{align}
Since $w_{i}(t(i)+\cdot)$ is a polynomial function, there exists a constant $C_i>0$ such that, for all $t \leq T_2$ and $s \in [-M,M]$, we have
\begin{align}
\label{Eq: eq 15}
    \Big{\vert} w_i(t(i)+t)-w_i\big(t(i)+u_{s}^{(n)}\big) \Big{\vert} \leq \frac{C_i}{\log(n)},
\end{align}
due to \eqref{Eq: eq 6}. Let $0<\widetilde{T}_1 <T_1$. For $n$ sufficiently large such that $u_{s}^{(n)} \geq \widetilde{T}_1$ for all $(t,s) \in \left[T_1,T_2\right]\times[-M,M]$, we have  
\begin{align}
     \Big{\lvert}  \frac{Z_{i}^{(n)}\big(\mathfrak{t}^{(n)}_{t(i)+t}+s\big)}{n^{t}\log^{\theta(i)}(n)e^{\lambda_0 s}}-W w_i(t(i)+t)\Big{\lvert} & \leq \sup_{x \in [\widetilde{T}_1,T_2]}  \Big{\lvert}  \frac{Z_{i}^{(n)}\big(\mathfrak{t}^{(n)}_{t(i)+x}+M\big)}{n^{x}\log^{\theta(i)}(n)e^{\lambda_0 M}}-W w_i(t(i)+x)\Big{\lvert}\\ 
     &\hspace{2cm}+W\frac{C_i}{\log(n)} .
\end{align}
Thus, for sufficiently large $n$, we have
\begin{align}
    &\mP \Big{(}\sup_{s\in [-M,M]}\sup_{t\in \left[T_1,T_2\right]} \Big{\lvert}  \frac{Z_{i}^{(n)}\big(\mathfrak{t}^{(n)}_{t(i)+t}+s\big)}{n^{t}\log^{\theta(i)}(n)e^{\lambda_0 s}}-W w_i(t(i)+t)\Big{\lvert} \geq 2\varepsilon \Big{)} \\
    &\hspace{1cm}\leq \mP \Big{(}\sup_{x\in \left[\widetilde{T}_1,T_2\right]} \Big{\lvert}  \frac{Z_{i}^{(n)}\big(\mathfrak{t}^{(n)}_{t(i)+x}+M\big)}{n^{x}\log^{\theta(i)}(n)e^{\lambda_0 M}}-W w_i(t(i)+x)\Big{\lvert} \geq \varepsilon \Big{)}+ \mP \Big( W \geq \frac{\varepsilon \log(n)}{C_i}\Big),
\end{align}
from which \eqref{Equation: convergence deleterious case} is obtained. Indeed, the first term on the right-hand side converges to $0$ according to Proposition \ref{Proposition: mutant pop deterministic time scale not uniform in s} (ii) and the second term converges to $0$ since $W$ is finite almost surely (see \eqref{Eq:distribution W}). 

\textbf{2. Neutral case:} Now, we show \eqref{Equation: conv neutral case}. We have  
\begin{align}
\label{Eq: eq 17}
    &\Big{\lvert}  \frac{Z_{i}^{(n)}\big(\mathfrak{t}^{(n)}_{t}+s\big)}{d_{i}^{(n)}(t,s)}-W w_i(t) \Big{\lvert}\leq \1_{\left\{t\in \left[0,t(i)-h_{n}^{-1}(i)\right)\right\}}Z_{i}^{(n)}\big( \mathfrak{t}^{(n)}_{u_{s}^{(n)}}+M\big)+\1_{\left\{t \in \left[t(i)-h_{n}^{-1}(i),t(i)\right)\right\}}\\
    &\cdot\frac{Z_{i}^{(n)}\big( \mathfrak{t}^{(n)}_{u_{s}^{(n)}}+M\big)}{\psi_{n}(i)\log^{\theta(i-1)}(n)}+\1_{\left\{t \geq t(i)\right\}}\1_{\left\{u_{s}^{(n)}<t(i)\right\}}\Big{\vert}\frac{Z_{i}^{(n)}\big( \mathfrak{t}^{(n)}_{u_{s}^{(n)}}+M\big)}{\log^{\theta(i)}(n)e^{\lambda_0 M}}e^{\left(t(i)-u_{s}^{(n)}\right)\log(n)}-W w_i(t)\Big{\vert} \\
    &+\1_{\left\{u_{s}^{(n)} \geq t(i)\right\}} \Big{\vert}\frac{Z_{i}^{(n)}\big(\mathfrak{t}^{(n)}_{u_{s}^{(n)}}+M\big)}{n^{u_{s}^{(n)}-t(i)}\log^{\theta(i)}(n)e^{\lambda_0 M}}-W w_i\big(u_s^{(n)}\big)\Big{\vert}\\
    &+\1_{\left\{u_{s}^{(n)} \geq t(i)\right\}} W \Big{\vert} w_i(t)-w_i\big(u_{s}^{(n)}\big) \Big{\vert}.
\end{align}
As in \eqref{Eq: eq 15}, there exists a constant $C_i$ such that, for all $(t,s) \in \left[0,T\right]\times [-M,M]$, we have
\begin{align}
    \Big{\vert} w_i(t)-w_i\big(u_{s}^{(n)}\big) \Big{\vert} \leq \frac{C_i}{\log(n)}.
\end{align}
In the case where $t \geq t(i)$ and $u_{s}^{(n)}<t(i)$, we have that $t(i)-u_{s}^{(n)} \leq \frac{2M}{\log(n)}\lambda_0$, which, in particular, implies that $e^{\left(t(i)-u_{s}^{(n)}\right)\log(n)} \leq e^{2M \lambda_0}$. Moreover, since $w_i\big(u_s^{(n)}\big)=0$ (because $w_i(s)=0$ for all $s \in [0,t(i)]$), it follows from the previous inequality that $w_i(t) \leq \frac{C_i}{\log(n)}$. Combining these arguments, we obtain
\begin{align}
\label{Eq: eq 16}
    \Big{\vert}\frac{Z_{i}^{(n)}\big(\mathfrak{t}^{(n)}_{u_{s}^{(n)}}+M\big)}{\log^{\theta(i)}(n)e^{\lambda_0 M}}e^{(t(i)-u_{s}^{(n)})\log(n)}-W w_i(t)\Big{\vert} &\leq \frac{Z_{i}^{(n)}\big(\mathfrak{t}^{(n)}_{u_{s}^{(n)}}+M\big)}{\log^{\theta(i)}(n)}e^{\lambda_0 M}\\
    &+ W\frac{C_i}{\log(n)}.
\end{align}
Finally, using \eqref{Eq: eq 17} and \eqref{Eq: eq 16}, we obtain for all $(t,s) \in [0,T]\times [-M,M]$ 
\begin{align}
    \Big{\lvert}  \frac{Z_{i}^{(n)}\big(\mathfrak{t}^{(n)}_{t}+s\big)}{d_{i}^{(n)}(t,s)}-W w_i(t)\Big{\lvert}&\leq \sup_{x \in \left[0,t(i)-h_{n}^{-1}(i)\right)} Z_{i}^{(n)}\big( \mathfrak{t}^{(n)}_{x}+M\big)\\
    &+\sup_{x \in [0,t(i)]}\frac{Z_{i}^{(n)}\big( \mathfrak{t}^{(n)}_{x}+M\big)}{\psi_{n}(i) \log^{\theta(i-1)}(n)}\\
    &+\sup_{x\in [0,t(i)]} \frac{Z_{i}^{(n)}\big( \mathfrak{t}^{(n)}_{x}+M\big)}{\log^{\theta(i)}(n)} e^{\lambda_0 M}+W\frac{2C_i}{\log(n)}\\
    &+\sup_{x \in [t(i),T]} \Big{\vert}\frac{Z_{i}^{(n)}\big( \mathfrak{t}^{(n)}_{x}+M\big)}{n^{x-t(i)}\log^{\theta(i)}(n)e^{\lambda_0 M}}-W w_i(x) \Big{\vert}.
\end{align}
Then we have
\begin{align}
    &\mathbb{P} \Big( \sup_{s \in [-M,M]} \sup_{t \in \left[0,T\right]} \Big{\vert} \frac{Z_{i}^{(n)}\big(\mathfrak{t}^{(n)}_{t}+s\big)}{d_{i}^{(n)}(t,s)}-Ww_i(t) \Big{\vert} \geq 5\varepsilon\Big) \\ 
    \label{Equation2}& \hspace{1cm} \leq \mP \Big( \sup_{x \in \left[0,t(i)-\gamma_{n}^{-1}(i)\right)} Z_{i}^{(n)}\big(\mathfrak{t}^{(n)}_{x}+M\big) \geq \varepsilon\Big)+\mP \Big( \sup_{x \in [0,t(i)]} \frac{Z_{i}^{(n)}\big(\mathfrak{t}^{(n)}_{x}+M\big)}{\psi_{n}(i)\log^{\theta(i-1)}(n)} \geq \varepsilon\Big)\\
   \label{Equation3}&\hspace{1cm}+\mathbb{P} \Big( \sup_{x \in [0,t(i)]} \Big{\vert} \frac{Z_{i}^{(n)}\big( \mathfrak{t}^{(n)}_{x}+M\big)}{e^{-\lambda_0 M}\log^{\theta(i)}(n)}\Big{\vert} \geq \varepsilon \Big)+\mathbb{P}\Big(W\geq \frac{\varepsilon \log(n)}{2C_i}\Big)\\
    \label{Equation1}&\hspace{1cm}+ \mathbb{P}\Big(\sup_{x \in \left[t(i),T\right]} \Big{\vert}\frac{Z_{i}^{(n)}\big(\mathfrak{t}^{(n)}_{x}+M\big)}{n^{x-t(i)}\log^{\theta(i)}(n)e^{\lambda_0 M}}-Ww_i(x)\Big{\vert} \geq \varepsilon\Big)\\
    &\underset{n \to \infty}{\longrightarrow}0,
\end{align}
where the different convergences to $0$ are obtained as follows: 
\begin{itemize}
    \item Lemma \ref{Lem: no mutant cell} gives the convergence of the first term in \eqref{Equation2},
    \item Lemma \ref{Lem: control before first mut} gives the convergence of the second term in \eqref{Equation2} and the first term in \eqref{Equation3}; for the latter, we apply Lemma \ref{Lem: control before first mut} with $\psi_n(i)=e^{-\lambda_0 M}\log(n)$, which is valid because $\theta(i)=\theta(i-1)+1$,
    \item for the second term in \eqref{Equation3}, we use the fact that $W$ is finite almost surely, see \eqref{Eq:distribution W},
    \item Step 3 of the neutral case in the proof of Proposition \ref{Proposition: mutant pop deterministic time scale not uniform in s} directly establishes the convergence of \eqref{Equation1}.
\end{itemize}
Finally, we have proven Equations \eqref{Equation: conv neutral case} and \eqref{Equation: convergence deleterious case} in the specific case of the infinite mono-directional graph. 

\subsection{First-order asymptotics of the size of the mutant subpopulations on the random time scale (Theorem \ref{Theorem: non increasin growth rate graph} (ii))} \label{Subsection: monodirectional graph random time}

In this subsection, we will first show that the random time at which the total population reaches the size $n^t$ behaves asymptotically as the random time at which the wild-type subpopulation reaches the size $n^t$. This result is obtained uniformly on the time parameter $t$, conditioned on $\{W>0\}$, and in probability. Intuitively, for any mutant trait $i \in \mathbb{N}$, the corresponding subpopulation grows exponentially at rate $\lambda_0$ after time $t(i)$, see Proposition \ref{Proposition: mutant pop deterministic time scale not uniform in s}. Due to these time delays (on the $\log(n)$-accelerated time scale), the total mutant subpopulation remains consistently negligible compared to the wild-type subpopulation. Consequently the difference between $\eta_t^{(n)}$ and $\sigma_{t}^{(n)}$ converges to $0$.
\begin{proposition}
\label{Prop: conv diff random time}
Assume Equation \eqref{Mutation regime} holds. Then, for all $\varepsilon>0$ and $0<T_{1}<T_{2}$, we have 
\begin{equation}
    \mP \Big(\sup_{t \in [T_1,T_2]} \big(\eta_{t}^{(n)}-\sigma_{t}^{(n)} \big) \leq \varepsilon \Big{\vert} W>0\Big) \underset{n \to \infty}{\longrightarrow}1.
\end{equation}
\end{proposition}

\begin{proof}
The proof will be carried out in two steps. We begin by establishing the result under a stronger condition. 

\textbf{Step 1:} In this step, we will show that for all $0<\delta_1<\delta_2$ and $\varepsilon>0$ we have
\begin{align}
\label{equation step i convergence zero random time scales}
    \mP \Big(\sup_{t \in [T_1,T_2]} \big( \eta^{(n)}_{t}-\sigma^{(n)}_{t}\big) \geq \varepsilon \Big{\vert} \delta_1<W<\delta_2 \Big) \underset{n \to \infty}{\longrightarrow} 0.
\end{align}
Let $0<\delta_1<\delta_2$. Then there exists $M \in \mathbb{R}^{+}$ such that
\begin{align}
\label{eq: control log(W)}
    \mP \Big(\Big\vert \frac{\log(W)}{\lambda_0}\Big\vert  \leq M \Big{\vert} \delta_1<W<\delta_2\Big)=1.
\end{align}
For all $\varepsilon>0$ introduce the event $A_\varepsilon^{(n)}:=\big\{\sup_{t \in [T_1,T_2]} \big(\eta_{t}^{(n)}-\sigma_{t}^{(n)}\big) \geq \varepsilon\big\}$. Assume that there exists $\varepsilon>0$ such that the sequence $\big(\mP \big(
A_{\varepsilon}^{(n)}\big{\vert}\delta_1< W<\delta_2\big)\big)_{n \in \mathbb{N}} $ does not converge to 0. This means that there exists $\eta>0$ for which there is an infinite subset $N \in \mathbb{N}$ such that for all $n \in N$, $\mP \big(A_{\varepsilon}^{(n)} \big{\vert} \delta_1<W<\delta_2\big) \geq \eta.$ For all $\widetilde{\varepsilon}>0$ introduce the event
\begin{align}
    B^{(n)}_{\widetilde{\varepsilon}}:= \Big\{\sup_{t \in [T_1,T_2]} \Big\vert\eta_t^{(n)}-\Big( \mathfrak{t}^{(n)}_{t}-\frac{\log(W)}{\lambda_0}\Big) \Big\vert \leq \widetilde{\varepsilon}\Big\},
\end{align}
which satisfies $\mP \big( B^{(n)}_{\widetilde{\varepsilon}} \big{\vert} \delta_1<W<\delta_2\big)\underset{n \to \infty}{\to}1,$ according to Lemma \ref{convergence temps arret}. From this fact, and since $\sigma^{(n)}_{t} \leq \eta^{(n)}_{t}$ for all $t>0$ almost surely, it follows that under $ B^{(n)}_{\widetilde{\varepsilon}}$, we have $ \sigma_{t}^{(n)}<\infty$ for all $t \in [T_1,T_2].$ Moreover, it also follows that under $ B^{(n)}_{\widetilde{\varepsilon}}$, we have $Z_{0}^{(n)}\big(\eta^{(n)}_{t}\big)=n^{t}$ for all $t \in [T_1,T_2]$. In particular, under $A^{(n)}_{\varepsilon}$, there exists $t_n \in [T_1,T_2]$ such that $\eta^{(n)}_{t_n}-\sigma^{(n)}_{t_n} \geq \varepsilon$, which implies that $Z_{0}^{(n)}\big(\sigma^{(n)}_{t_n}\big) \leq n^{t_n}e^{-\lambda_0\frac{\varepsilon}{2}}$. Otherwise, by applying the strong Markov property, it would lead to a contradiction with $A^{(n)}_{\varepsilon}$. Combining these reasonings, it follows that under $A_{\varepsilon}^{(n)} \cap B^{(n)}_{\widetilde{\varepsilon}}$,  we have that 
\begin{align}
\label{Eq: mutant pop at least certain size}
    \sum_{i \geq 1} Z_{i}^{(n)}\big(\sigma^{(n)}_{t_n}\big)= Z_{tot}^{(n)}\big(\sigma_{t_n}^{(n)}\big)-Z_0^{(n)}\big(\sigma_{t_n}^{(n)}\big) \geq n^{t_n}\big(1-e^{-\lambda_0 \frac{\varepsilon}{2}}\big)=\Omega\left(n^{t_n}\right),
\end{align}
where we use the standard Landau notation for $\Omega$. However, the result regarding the mutant subpopulations indicates that, due to the power law mutation rates regime, the mutant subpopulations have a negligible size compared to the wild-type subpopulation. More precisely, under the event $A_{\varepsilon}^{(n)} \cap B^{(n)}_{\widetilde{\varepsilon}}$, using \eqref{eq: control log(W)} and Proposition \ref{Proposition: mutant pop deterministic time scale not uniform in s}, we have 
\begin{align}
\label{Eq: control mutant pop proof convergence stopping times}
    \sum_{i\geq 1}Z_{i}^{(n)}\big(\sigma_{t_n}^{(n)}\big) &\leq \sup_{u \in [T_1,t_n]}\sum_{i\geq 1}Z_{i}^{(n)}\big(\eta_u^{(n)}\big)  \\
    &\leq \sup_{u \in [T_1,t_n]}\sup_{s \in [-(M+\widetilde{\varepsilon}),M+\widetilde{\varepsilon}]}\sum_{i\geq 1}Z_{i}^{(n)}\big(\mathfrak{t}^{(n)}_{u}+s\big) \\
    &= o(n^{t_n}).
\end{align}
There is a contradiction between \eqref{Eq: mutant pop at least certain size} and \eqref{Eq: control mutant pop proof convergence stopping times}, so we have proven \eqref{equation step i convergence zero random time scales} for all $\varepsilon>0$ and $0<\delta_1<\delta_2$. 

\textbf{Step 2:} Using a similar method as in Step 2 of the proof of Lemma \ref{convergence temps arret}, one can show that for all $\varepsilon>0$
\begin{align}
    \mP \Big( A^{(n)}_{\varepsilon}\Big{\vert} W>0\Big) \underset{n \to \infty}{\longrightarrow}0,
\end{align}
which concludes the proof.
\end{proof}

In the remainder of this subsection, we will prove the following proposition.
\begin{proposition}
\label{Prop: random time convergence 1 site}
Assume Equation \eqref{Mutation regime} holds. Let $0<T_{1}<T_{2}$, $M>0$ and $\varepsilon>0$. Consider $\big(\rho_{t}^{(n)}\big)_{t \in \mathbb{R}^{+}}$ as defined in \eqref{rho}. Then, we have
\begin{itemize}
    \item If $\lambda_i=\lambda_0$ 
\begin{align}
    &\mP \Big(\sup_{s \in [-M,M]}\sup_{t \in [T_1,T_2]} \Big{\vert} \frac{Z_{i}^{(n)}\big(\rho_{t}^{(n)}+s\big)}{d_i^{(n)}(t,s)}-\1_{ \{W>0\}}w_{i}(t)\Big{\vert} \geq \varepsilon\Big)\underset{n \to \infty}{\longrightarrow} 0.
\end{align}
\item If $\lambda_i<\lambda_0$
\begin{align}
    & \mP\Big(\sup_{s \in [-M,M]}\sup_{t \in [T_1,T_2]}  \Big{\vert}\frac{Z_{i}^{(n)}\big( \rho_{t(i)+t}^{(n)}+s\big)}{n^{t}\log^{\theta(i)}(n)e^{\lambda_0 s}} - \1_{\{W>0\}}w_{i}(t(i)+t) \Big{\vert} \geq \varepsilon\Big) \underset{n \to \infty}{\longrightarrow} 0.
\end{align}
\end{itemize}
\end{proposition}
These results correspond to \eqref{Equation: conv neutral case random scale} and \eqref{Equation: convergence deleterious case in random scale} for the mono-directional graph. The proof will be carried out under the assumption that $\lambda_i=\lambda_0$. The case where $\lambda_i<\lambda_0$ can be addressed using similar reasoning and is left to the reader.
\begin{proof}[Proof of Proposition \ref{Prop: random time convergence 1 site}]
Estimate the quantity of interest from above as
\begin{align}
    &\mP \Big(\sup_{s\in [-M,M]}\sup_{t \in [T_{1},T_{2}]} \Big{\vert} \frac{Z_{i}^{(n)}\big(\rho^{(n)}_{t}+s\big)}{d_{i}^{(n)}(t,s)}-w_{i}(t)\1_{\{W>0\}}  \Big{\vert} \geq \varepsilon\Big)\\
    \label{Equa1}&\hspace{1cm}\leq \mP \Big( \{W>0\} \cap \Big\{\sup_{s \in [-M,M]}\sup_{t \in [T_{1},T_{2}]} \Big{\vert} \frac{Z_{i}^{(n)}\big(\rho^{(n)}_{t}+s\big)}{d_{i}^{(n)}(t,s) }-w_{i}(t)\Big{\vert} \geq \varepsilon \Big\} \Big)\\
    \label{Equa2}&\hspace{1cm}+ \mP \Big(\{W=0\} \cap \Big\{K^{(n)}_{0}\big(\rho_{T_2}^{(n)}+M\big)\geq 1\Big\} \cup  \Big\{H^{(n)}_{0}\big(\rho_{T_2}^{(n)}+M\big)\geq 1\Big\} \Big),
\end{align}
where, for the term in \eqref{Equa2}, we use the fact that a necessary condition for the mutant subpopulation of trait $i$ to be strictly positive is that at least one mutational event from the wild-type subpopulation must have occurred before.  

\textbf{Step 1: }
The convergence to 0 of \eqref{Equa2} follows from proving that 
\begin{align}
        \mP \Big( \Big\{\sup_{t \in \mathbb{R}^{+}} K^{(n)}_0(t)=0 \Big\} \cap  \Big\{\sup_{t \in \mathbb{R}^{+}} H^{(n)}_0(t)=0 \Big\} \Big{\vert} W=0 \Big) \underset{n \to \infty}{\longrightarrow} 1. 
    \end{align}
Let us first show that $\mP \big( \sup_{t \in \mathbb{R}^{+}}K^{(n)}_{0}(t) \geq 1 \big{\vert} W=0\big) \underset{n \to \infty}{\to}0.$ Notice that, almost surely, for all $t \in \mathbb{R}^{+}$
$$K^{(n)}_0(t) \leq \widetilde{K}^{(n)}(t):=\int_{0}^{t} \int_{\mathbb{R}^{+}} \1_{\left\{\theta \leq 2 \alpha_0 \mu_0^{(n)} Z_0(s^{-})\right\}}N_0(ds,d\theta),$$
because, almost surely, for all $t \in \mathbb{R}^{+}$, we have $Z_{0}^{(n)}(t) \leq Z_{0}(t)$. 
Then it follows that 
\begin{align}
    \mP \Big( \sup_{t \in \mathbb{R}^{+}} K^{(n)}_0(t) \geq 1 \Big{\vert} W=0\Big) &\leq \E \Big[\sup_{t \in \mathbb{R}^{+} } \widetilde{K}^{(n)}(t) \wedge 1 \Big{\vert} W=0\Big]\underset{n \to \infty}{\longrightarrow}0
\end{align}
by dominated convergence. Indeed, for all $\omega \in \{W=0\}$, there exists $T(\omega) \in \mathbb{R}^{+}$ such that for all $t \geq T(\omega), Z_{0}(t)=0$. Combined with $\mu_0^{(n)} \underset{n \to \infty}{\to} 0$, it follows that there exists $N(\omega) \in \mathbb{N}$ such that for all $n \geq N(\omega)$, we have $\sup_{t \in \mathbb{R}^{+}}\widetilde{K}^{(n)}(t)=0$. We conclude the proof of Step 1 by showing that $\mP \big( \sup_{t \in \mathbb{R}^{+}}H^{(n)}_{0}(t) \geq 1 \big{\vert} W=0\big) \underset{n \to \infty}{\longrightarrow}0$ using similar reasoning.  

\textbf{Step 2:} We will show that \eqref{Equa1} converges to $0$ in three steps.

\textbf{Step 2) (i)}: We begin by showing that for all $\varepsilon>0$ and $\eta>0$ we have 
\begin{equation}
\label{Eq: eq 18}
    \mP \Big( \sup_{s \in [-M,M]}\sup_{t \in [T_{1},T_{2}]} \Big{\vert} \frac{Z_{i}^{(n)}\big(\rho^{(n)}_{t}+s\big)}{d_i^{(n)}(t,s)}e^{-\lambda_0 \left[\rho^{(n)}_{t}-\mathfrak{t}^{(n)}_{t}\right]}- Ww_{i}(t)\Big{\vert} \geq \varepsilon \Big{\vert} W>\eta\Big) \underset{n \to \infty}{\longrightarrow}0.
\end{equation}
We have 
\begin{align}
      &\mP \Big( \sup_{t \in [T_{1},T_{2}]} \Big{\vert} \rho^{(n)}_{t}-\mathfrak{t}^{(n)}_{t} \Big{\vert}  \geq M \Bigg{\vert} W> \eta\Big)\\
      &\hspace{1cm}\leq  \mP \Big( \sup_{t \in [T_{1},T_{2}]} \Big{\vert} \eta^{(n)}_{t}- \Big(\mathfrak{t}^{(n)}_{t}-\frac{\log(W)}{\lambda_0} \Big)\Big{\vert} \geq \frac{M}{3} \Big{\vert} W> \eta \Big) \\
      &\hspace{1cm}+\mP \Big( \sup_{t \in [T_1,T_2]} \Big\vert \rho_{t}^{(n)}-\eta_{t}^{(n)} \Big\vert \geq \frac{M}{3}\Big)+ \mP \Big(\frac{\vert \log(W)\vert}{\lambda_0} \geq \frac{M}{3} \Big{\vert} W> \eta\Big).
\end{align}
Let $\delta>0$. Using Lemma \ref{convergence temps arret}, Proposition \ref{Prop: conv diff random time} and the distribution of $W$ given in \eqref{Eq:distribution W}, there exist $M>0$ and $N_1 \in \mathbb{N}$ such that for all $n \geq N_1$,
\begin{equation}
\label{Eq 54}
    \mP \Big( \sup_{t \in [T_{1},T_{2}]} \Big{\vert} \rho^{(n)}_{t}-\mathfrak{t}^{(n)}_{t} \Big{\vert}  \geq M \Big{\vert} W> \eta\Big) \leq \frac{\delta}{2}.
\end{equation}
Now, we can apply Theorem \ref{Theorem: non increasin growth rate graph} (i) Eq. \eqref{Equation: conv neutral case} to get that there exists $N_2 \in \mathbb{N}$ such that for all $n \geq N_2$,
\begin{equation}
\label{Eq 55}
     \mP \Big{(}  \sup_{s \in [-M,M]}\sup_{s_1 \in [-M,M]}\sup_{t \in [T_{1},T_{2}]}  \Big{\lvert}  \frac{Z_{i}^{(n)}\big(\mathfrak{t}^{(n)}_{t}+s+s_1\big)}{d_i^{(n)}(t,s+s_1)}-Ww_{i}(t) \Big{\lvert} \geq  \varepsilon \Big{)} \leq \frac{\delta}{2}.
\end{equation}
Consequently, using Equations \eqref{Eq 54} and \eqref{Eq 55}, we have shown that for all $\delta>0$, there exists $N:= \max (N_1,N_2) \in \mathbb{N}$ such that for all $n \geq N$, 
\begin{equation}
    \mathbb{P} \Big( \sup_{s\in [-M,M]}\sup_{t \in [T_{1},T_{2}]} \Big{\vert} \frac{Z_{i}^{(n)}\big(\rho^{(n)}_{t}+s\big)}{d_{i}^{(n)}(t,s)}e^{-\lambda_0 \left[\rho^{(n)}_{t}-\mathfrak{t}^{(n)}_{t}\right]}- Ww_{i}(t)\Big{\vert} \geq \varepsilon \Big{\vert} W>\eta\Big) \leq \delta,
\end{equation}
which concludes Step 2) (i). 

\textbf{Step 2) (ii):} Now, we are going to prove that  
\begin{align}
\label{Eq: eq 20}
    \mP \Big( \sup_{s\in [-M,M]}\sup_{t \in [T_{1},T_{2}]} \Big{\vert} \frac{Z_{i}^{(n)}\big(\rho^{(n)}_{t}+s\big)}{d_{i}^{(n)}(t,s)}-w_{i}(t)\Big{\vert} \geq \varepsilon \Big{\vert} W >\eta\Big) \underset{n \to \infty}{\longrightarrow}0.
\end{align}
Let $\delta>0$ and $0<\widetilde{\varepsilon}<\eta $. According to Remark \ref{Rq: order of stopping time wtp reaches size power of n} Equation \eqref{Rq: approx W by exponential term with tau} and Proposition \ref{Prop: conv diff random time}, we have 
\begin{align}
    \mP \left(A_{\widetilde{\varepsilon}}^{(n)} \big\vert W>\eta\right) \geq 1-\frac{\delta}{2}, \text{ where } A_{\widetilde{\varepsilon}}^{(n)}:=\Big{\{} \sup_{t \in [T_{1},T_{2}]} \Big{\vert} e^{-\lambda_0\left(\rho_{t}^{(n)}-\mathfrak{t}^{(n)}_{t}\right)}-W \Big{\vert} \leq \widetilde{\varepsilon}\Big{\}}.
\end{align}
Combined with Step 2) (i), there exists $N \in \mathbb{N}$ such that for all $n \geq N$, we have $\mP \big( A_{\widetilde{\varepsilon}}^{(n)}\cap B_{\widetilde{\varepsilon}}^{(n)}|W >\eta\big)\geq 1-\delta$, where
\begin{align}
    &B_{\widetilde{\varepsilon}}^{(n)}:= \Big{\{}\sup_{s\in [-M,M]}\sup_{t \in [T_{1},T_{2}]} \Big{\vert} \frac{Z_{i}^{(n)}\big(\rho^{(n)}_{t}+s\big)}{d_{i}^{(n)}(t,s)}e^{-\lambda_0 \left[\rho^{(n)}_{t}-\mathfrak{t}^{(n)}_{t}\right]}-Ww_{i}(t)\Big{\vert} \leq \widetilde{\varepsilon}\Big{\}}.
\end{align}
In particular, conditioned on $\{W >\eta \}$, under the event $A_{\widetilde{\varepsilon}}^{(n)}\cap B_{\widetilde{\varepsilon}}^{(n)}$, we have that for all $t \in [T_1,T_2]$ and for all $s\in [-M,M]$,
\begin{align}
    \frac{Z_{i}^{(n)}\big(\rho_{t}^{(n)}+s\big)}{d_{i}^{(n)}(t,s)}-w_{i}(t)&\leq \left(\widetilde{\varepsilon}+w_{i}(t)W\right)e^{\lambda_0 \left(\rho_{t}^{(n)}-\mathfrak{t}^{(n)}_{t}\right)}-w_{i}(t) \\
    &\leq \frac{\widetilde{\varepsilon}}{W-\widetilde{\varepsilon}}+w_{i}(t)\Big( \frac{W}{W-\widetilde{\varepsilon}}-1\Big) \\ 
    & \leq \left(1+ w_{i}(T_2)\right)\frac{\widetilde{\varepsilon}}{\eta-\widetilde{\varepsilon}}\\
    &\underset{\widetilde{\varepsilon} \to 0}{\longrightarrow}0,
\end{align}
so that we can choose $\widetilde{\varepsilon}$ arbitrarily small such that this upper bound is smaller than $\varepsilon$. By applying a similar approach for the lower bound, we find that, conditioned on $\{W >\eta\}$, under the event $A_{\widetilde{\varepsilon}}^{(n)}\cap B_{\widetilde{\varepsilon}}^{(n)}$, we have that for all $t \in [T_1,T_2]$ and for all $s\in [-M,M]$,
\begin{align}
    \frac{Z_{i}^{(n)}\big(\rho_{t}^{(n)}+s\big)}{d_{i}^{(n)}(t,s)}-w_{i}(t)&\geq -\left(1+w_{i}(T_{2}) \right)\frac{\tilde{\varepsilon}}{\eta-\tilde{\varepsilon}}\underset{\tilde{\varepsilon}\to 0}{\longrightarrow}0.
\end{align}
Consequently, by choosing an appropriate $\widetilde{\varepsilon}>0$, we have shown that there exists $N \in \mathbb{N}$ such that for all $n \geq N$,
\begin{align}
    \mP \Big(\sup_{s\in [-M,M]}\sup_{t \in [T_{1},T_{2}]} \Big{\vert} \frac{Z_{i}^{(n)}\big(\rho^{(n)}_{t}+s\big)}{d_{i}^{(n)}(t,s)}-w_{i}(t) \Big{\vert} \leq \varepsilon \Big{\vert} W>\eta \Big) \geq 1-\delta.
\end{align}

\textbf{Step 2) (iii):} Introduce the notation $C_{\varepsilon}^{(n)}:=\big{\{}\underset{s\in [-M,M]}{\sup}\underset{t \in [T_{1},T_{2}]}{\sup} \big{\vert} \frac{Z_{i}^{(n)}\big(\rho^{(n)}_{t}+s\big)}{d_{i}^{(n)}(t,s)}-w_{i}(t)\big{\vert} \geq \varepsilon \big{\}}$. To complete the proof of Step 2, we will show that $\mP \big(C_{\varepsilon}^{(n)} \cap \{W>0\}\big) \underset{n \to \infty}{\to}0.$ We have 
\begin{align}
    \mP \left(C_{\varepsilon}^{(n)} \cap \{W>0\} \right) \leq \mP\left(C_{\varepsilon}^{(n)} \cap \{W>\eta\} \right) +\mP\left(0<W<\eta\right).
\end{align}
Using Step 2) (ii), we obtain 
\begin{equation}
    \limsup_{n \to \infty}\mP \left(C_{\varepsilon}^{(n)}  \cap \{W>0\} \right)  \leq \mP\left(0<W<\eta\right).
\end{equation}
Taking the limit as $\eta\underset{n \to \infty}{\to}0$ completes the proof.
\end{proof}

\section{First-order asymptotics of the size of the mutant subpopulations for a general finite trait space (Theorem \ref{Theorem: non increasin growth rate graph})}
\label{Section: generalisation of the proof}
As in Section \ref{Section: proof mono directional graph} the sequence $\big(Z^{(n)}_{v}, v \in V\big)_{n \in \mathbb{N}}$ is mathematically constructed using independent PPMs. In this construction, each population of trait $v$ is decomposed as the sum of subpopulations indexed by the walks in the graph that start from trait $0$ and lead to trait $v$. An exact definition will be given below. The idea is to apply the reasoning of Section \ref{Section: proof mono directional graph} to each walk $\gamma$ from trait $0$ to trait $v$, which will provide the first-order asymptotics for the subpopulation of cells of trait $v$ indexed by $\gamma$. By comparing the order of distinct walks, we can then conclude the first-order asymptotics of the size of the mutant subpopulation of trait $v$. However, this reasoning holds only if there are finitely many walks from trait $0$ to trait $v$. In particular, notice that due to cycles, there may be countably infinitely many walks from trait $0$ to trait $v$. Consequently, the proof requires additional steps, introducing an equivalence relation on the walks. We argue that there are only finitely many equivalent classes, and for each equivalent class, the result follows by adapting the reasoning from Section \ref{Section: proof mono directional graph}. For the equivalent class with infinitely many walks, we show that, with high probability, most of these walks do not asymptotically contribute.

Among wild-type individuals, we define the \textit{primary cell population}, denoted by \\
$\big(Z_{(0)}^{(n)}(t)\big)_{t \geq 0}$, as the set of all cells that have no mutants in their ancestry, tracing back to the initial cell. This corresponds to $Z_{0}^{(n)}$ in the case of the mono-directional graph. 
\begin{definition}[Walks and neighbors]
Define the set of all walks in the graph $V$ starting from trait $0$ as $\Gamma(V)$.  For a trait $v \in V$, the set of traits to which a cell of trait $v$ may mutate is defined as $N(v):=\{u \in V: (v,u) \in E\}.$ For a walk $\gamma=(0,\cdots,\gamma(k))\in \Gamma(V)$, denote the last trait $\gamma(k)$ visited by $\gamma$ as $\gamma_{end}:=\gamma(k)$, and the sub-walk that does not include this last trait as $\overset{\leftarrow}{\gamma}:=\left(0,\cdots,\gamma(k-1)\right).$ Introduce the sets of tuples of the walks in $V$ starting from trait $0$, associated with one or two neighbors of the last trait of $\gamma$, as 
\begin{align}
    &N_{\Gamma}:=\{(\gamma,v): \gamma \in \Gamma(V), v \in N(\gamma_{end})\},
\end{align}
and
\begin{align}
    M_{\Gamma}:=\{(\gamma,(v,u)): \gamma \in \Gamma(V), (v,u) \in N(\gamma_{end})\times N(\gamma_{end})\}. 
\end{align}
\end{definition}
We then introduce the birth, death and growth rates of any lineage of a cell of trait $v$ as 
\begin{align}
    &\alpha^{(n)}(v)=\alpha(v)\Big(1-\overline{\mu}^{(n)}(v)\Big)^{2} \text{ with } \overline{\mu}^{(n)}(v):=\sum_{u \in V: (v,u) \in E} \mu^{(n)}(v,u), \\
    &\beta^{(n)}(v)=\beta(v)+\alpha(v)\sum_{(u,w)\in N(v)\times N(v)}\mu^{(n)}(v,u) \mu^{(n)}(v,w),\\
    &\lambda^{(n)}(v)=\alpha^{(n)}(v)-\beta^{(n)}(v)=\lambda(v)-2\alpha(v)\overline{\mu}^{(n)}(v).
\end{align}
Let $Q_{(0)}^{b}(ds,d\theta)$, $Q_{(0)}^{d}(ds,d\theta)$, $\left(Q_{\gamma}(ds,d\theta) \right)_{\gamma\in \Gamma(V)}$, $\left(Q_{\gamma,v}(ds,d\theta) \right)_{(\gamma,v) \in N_{\Gamma}}$
and \\
$\left(Q_{\gamma, (v,u)}(ds,d\theta) \right)_{(\gamma,(v,u)) \in M_{\Gamma}}$
be independent PPMs with intensity $dsd\theta$. The subpopulation of primary cells is
\begin{align}
\label{Equ6}
    Z_{(0)}^{(n)}(t)&:=1+\int_{0}^{t}\int_{\mathbb{R}^{+}} \1_{\left\{\theta \leq \alpha^{(n)}(0)Z_{(0)}^{(n)}(s^{-})\right\}}Q_{(0)}^{b}(ds,d\theta)\\
    &-\int_{0}^{t}\int_{\mathbb{R}^{+}}\1_{\left\{\theta \leq \beta(0)Z_{(0)}^{(n)}(s^{-}) \right\}}Q_{(0)}^{d}(ds,d\theta)-\sum_{(v,u) \in N(0)\times N(0)} H_{(0),(v,u)}^{(n)}(t),
\end{align}
and for all $\gamma \in \Gamma(V)$, the subpopulation among the cells of trait $\gamma_{end}$ whose ancestry traces back to trait $0$ with mutations occurring exactly along the edges of $\gamma$ is 
\begin{align}
\label{Equ7}
    \hspace{1cm}&Z_{\gamma}^{(n)}(t):=\int_{0}^{t}\int_{\mathbb{R}^{+}} \Bigg{(} \1_{\left\{\theta \leq \alpha^{(n)}(\gamma_{end})Z_{\gamma}^{(n)}(s^{-})\right\}}\\
    &\hspace{2.5cm}-\1_{\left\{\alpha^{(n)}(\gamma_{end})Z_{\gamma}^{(n)}(s^{-}) \leq \theta \leq \left(\alpha^{(n)}(\gamma_{end})+\beta(\gamma_{end})\right)Z_{\gamma}^{(n)}(s^{-}) \right\}}\Bigg{)}Q_{\gamma}(ds,d\theta) \\ 
    &+K_{\overset{\leftarrow}{\gamma},\gamma_{end}}^{(n)}(t)+2 H^{(n)}_{\overset{\leftarrow}{\gamma},(\gamma_{end},\gamma_{end})}+\sum_{v \in N(\overline{\gamma}_{end}), v \neq \gamma_{end}}\left(H_{\overset{\leftarrow}{\gamma},(\gamma_{end},v)}^{(n)}+H_{\overset{\leftarrow}{\gamma},(v,\gamma_{end})}^{(n)}\right)(t) \\
    &-\sum_{(v,u)\in N(\gamma_{end})\times N(\gamma_{end})  }H_{\gamma, (v,u)}^{(n)}(t),
\end{align}
where for all $(\gamma,v) \in N_{\Gamma}$,
\begin{align}
\label{ModelConstructionK}
    & K_{\gamma,v}^{(n)}(t):=\int_{0}^{t}\int_{\mathbb{R}^{+}}\1_{\{\theta \leq 2\alpha(\gamma_{end})\mu^{(n)}(\gamma_{end},v)\left(1-\overline{\mu}^{(n)}(\gamma_{end})\right)Z_{\gamma}^{(n)}(s^{-})\}}Q_{\gamma,v}(ds,d\theta),
\end{align}
and for all $(\gamma,(v,u)) \in M_{\Gamma}$,
\begin{align}
\label{ModelConstuctionH}
    H_{\gamma,(v,u)}^{(n)}(t):=\int_{0}^{t}\int_{\mathbb{R}^{+}} \1_{\left\{ \theta \leq \alpha(\gamma_{end}) \mu^{(n)}(\gamma_{end},v) \mu^{(n)}(\gamma_{end},u)Z_{\gamma}^{(n)}(s^{-})\right\}} Q_{\gamma, (v,u)}(ds,d\theta).
\end{align}
The process $\big(K_{\gamma,v}^{(n)}(t)\big)_{t\in\mathbb{R}^{+}}$, resp. $\big(H_{\gamma, \{v,u\}}^{(n)}(t):=H_{\gamma, (v,u)}^{(n)}(t)+H_{\gamma, (u,v)}^{(n)}(t)\big)_{t\in\mathbb{R}^{+}}$, counts the number of mutations up to time $t$ from the subpopulation indexed by $\gamma$ that result in exactly one mutant daughter cell of trait $v$, resp. two mutant daughter cells of traits $\{v,u\}$. Hence the subpopulation of trait $v \in V$ is 
\begin{align}
\label{ModelConstructionZ}
    Z_{v}^{(n)}(t):=Z_{(0)}^{(n)}(t)\1_{\{v=0\}}+\sum_{\gamma\in P(v)} Z_{\gamma}^{(n)}(t),
\end{align} 
where $P(v)$, defined in Definition \ref{Def: Admissible path}, is the set of all walks from trait $0$ to trait $v$. 
\begin{definition}[Limiting birth-death branching process for the primary cell population]
Let $(Z_{(0)}(t))_{t \in \mathbb{R}^{+}}$ be the birth-death branching process with rates $\alpha(0)$ and $\beta(0)$ respectively, constructed as follows
\begin{align}
\label{def Z(0)}
    Z_{(0)}(t)&=1+\int_{0}^{t}\int_{\mathbb{R}^{+}} \1_{\left\{\theta \leq \alpha(0)Z_{(0)}(s^{-})\right\}}Q_{(0)}^{b}(ds,d\theta)\\
    &-\int_{0}^{t}\int_{\mathbb{R}^{+}} \1_{\left\{ \theta \leq \beta(0)Z_{(0)}(s^{-})\right\}} Q_{(0)}^{d}(ds,d\theta).
\end{align}
Notice that with this construction, the monotone coupling $$\forall t \geq 0, Z_{(0)}^{(n)}(t)\leq Z_{(0)}(t), a.s.$$ immediately follows. \\
Introduce the almost sure limit of the positive martingale  $\left( e^{-\lambda(0) t}Z_{(0)}(t)\right)_{t \in \mathbb{R}^{+}}$ as 
\begin{align}
\label{W def general finite graph}
    W:=\lim_{t \to \infty} e^{-\lambda(0) t}Z_{(0)}(t),
\end{align}
whose law is $W\overset{law}{=}Ber\left(\frac{\lambda(0)}{\alpha(0)}\right)\otimes Exp \left( \frac{\lambda(0)}{\alpha(0)}\right),$ see \cite[Section 1.1]{durrett2010evolution}, or \cite[Theorem 1]{durrett2015branching}. 
\end{definition}

\begin{lemma}
\label{Lem:control primary pop general finite graph}
There exist $C(\alpha(0),\lambda(0))>0$ and $N \in \mathbb{N}$ such that for all $\varepsilon>0$ and $n \geq N$, 
\begin{align}
    \mathbb{P}\Big( \sup_{t \in \mathbb{R}^{+}} \Big{\vert} e^{-\lambda(0)t}Z_{(0)}(t)-e^{-\lambda^{(n)}(0)t}Z_{(0)}^{(n)}(t)\Big{\vert} \geq \varepsilon\Big) \leq \frac{ C(\alpha(0),\lambda(0))}{\varepsilon^{2}} \overline{\mu}^{(n)}(0) \underset{n \to \infty}{\longrightarrow}0.
\end{align}
\end{lemma}
\begin{proof}
    The result is derived by adapting the proof of Lemma \ref{Lem:control primary pop} with $\mu^{(n)}_{0}$ replaced by $\overline{\mu}^{(n)}(0)$. 
\end{proof}
Introduce the stopping time of the first time that the primary cell population reaches the size $n^{t}$ as
\begin{align}
    \tau_{t}^{(n)}:=\inf \big{\{}u \in \mathbb{R}^{+}: Z_{(0)}^{(n)}(u) \geq n^{t}\big{\}}.
\end{align}
\begin{lemma}
\label{Lem: deterministic approx tau general finite graph}
For all $\varepsilon>0$, $(T_1,T_2) \in \mathbb{R}^{+}$ and $\varphi_n$ such that  $\log(n)=o(\varphi_n)$ and $\varphi_n=o\big{(}n^{\underset{v \in N(0)}{\min}\ell(0,v)}\big{)}$, we have 
\begin{equation}
    \mathbb{P} \Big{(}\sup_{t \in \left[T_{1},T_{2}\frac{\varphi_n}{\log(n)}\right]} \Big{\vert}\tau_{t}^{(n)}-\Big{(}\mathfrak{t}^{(n)}_{t}-\frac{\log(W)}{\lambda(0)}\Big{)}\Big{\vert} \geq \varepsilon \Big{\vert} W>0\Big{)}\underset{n \to \infty}{\longrightarrow} 0.
\end{equation}
\end{lemma}
\begin{proof}
    By following the proof of Lemma \ref{convergence temps arret}, with $Z_{0}^{(n)}$ and $\eta_{t}^{(n)}$ replaced by $Z_{(0)}^{(n)}$ and $\tau_{t}^{(n)}$, respectively, we obtain the result. 
\end{proof}
In the next definition, we introduce an equivalence relation on $\Gamma(V)$. Two walks are considered equivalent if they are identical up to cycles (including cycles formed by backward mutations). More precisely, two walks are equivalent if there exists a minimal walk such that both walks use all the edges of this minimal walk, possible along with additional edges that form cycles. The purpose of this equivalence relation is to establish that, within a class of equivalence, only the walk with the minimal length contributes to the asymptotic size of the mutant subpopulation. In particular, this minimal walk is actually a path, since only distinct vertices are visited. 

\begin{definition}[Equivalence relation on $\Gamma(V)$]\label{Def: equivalent relation}
We say that two walks $\gamma_1$ and  $\gamma_2$ in $ \Gamma(V)\times \Gamma(V)$ are equivalent, denoted by $\gamma_1 \sim \gamma_2$, if and only if there exists $\gamma \in \Gamma(V)$, and for all $j \in \{1,2\}$ there exists 
\begin{align}
    \sigma_j: \{0, \cdots \vert \gamma &\vert-1\} \to \{ 0, \cdots,\vert \gamma_j \vert-1 \}^2 \\
    & i \mapsto (\underline{\sigma}_{j}(i), \overline{\sigma}_{j}(i))
\end{align}
satisfying :
\begin{itemize}
    \item $\forall j \in \{1,2\}, \underline{\sigma}_{j}(0)=0$, and $\overline{\sigma}_{j}(\vert \gamma \vert -1)=\vert \gamma_j \vert -1$,
    \item $\forall i \in \{ 0, \cdots \vert \gamma \vert -1 \}, \forall j \in \{1,2\}, \underline{\sigma}_{j}(i) \leq \overline{\sigma}_{j}(i)$ and $\overline{\sigma}_{j}(i)+1=\underline{\sigma}_{j}(i+1)$,
    \item $\forall i \in \{ 0, \cdots, \vert \gamma \vert -1 \}, \forall j \in \{1,2\}, \gamma(i)=\gamma_{j}(\underline{\sigma}_{j}(i))=\gamma_{j}(\overline{\sigma}_{j}(i))$.
\end{itemize} 
Since the graph is finite, there are only a finite number of equivalence classes. For each walk $\gamma \in \Gamma(V)$, denote by $[\gamma]$ its equivalence class. In each class of equivalence, there is a natural representative candidate which is the walk with the minimum length; we will denote this walk by $\widetilde{\gamma}$. For each $v \in V$, denote by $C(v)$ the set of representative candidates for the walks in $P(v)$. Note that $\vert C(v) \vert<\infty$. An illustration of this definition can be found in Figure~\ref{Figure: Equivalent paths}.
\end{definition}
\begin{figure}
\centering
\includegraphics[width=0.95 \linewidth]{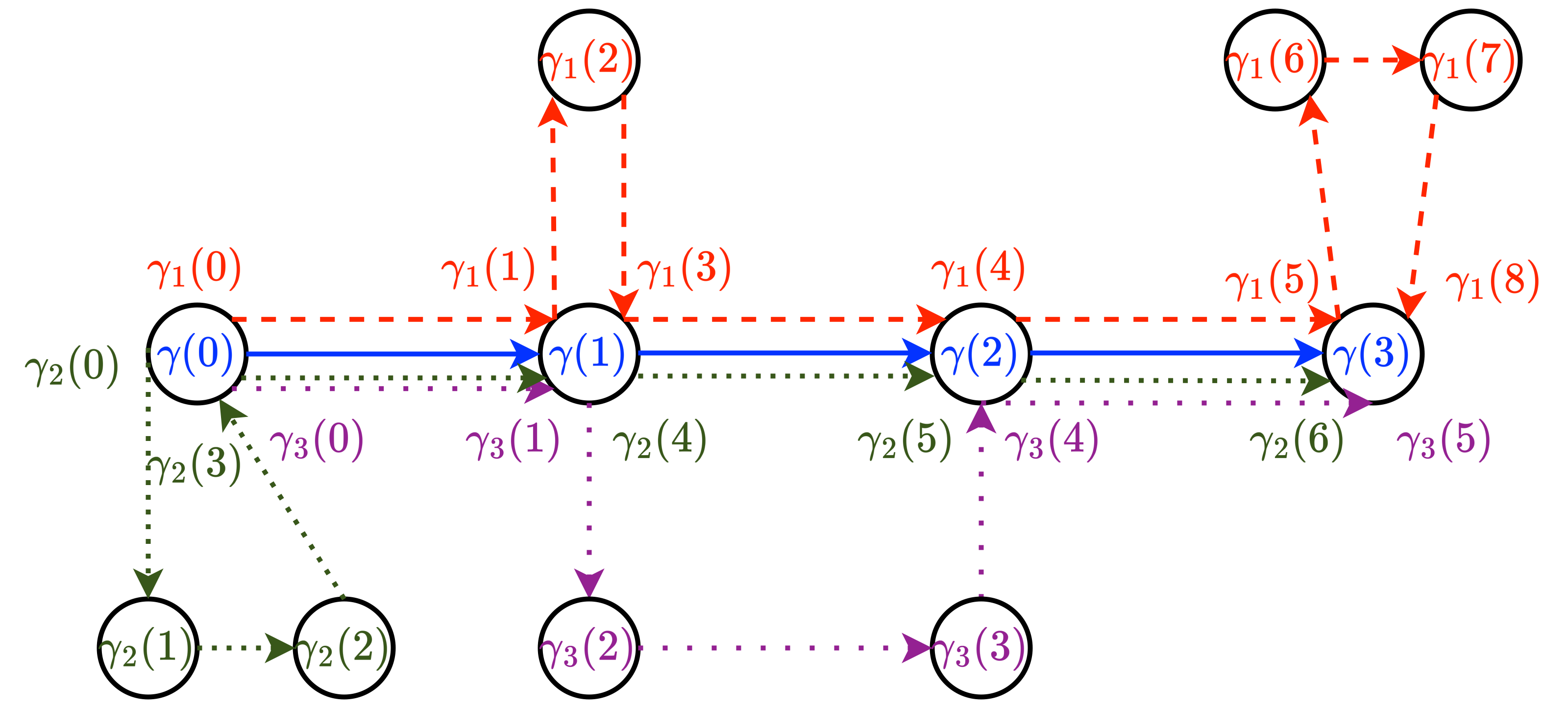}\caption{Example of Definition \ref{Def: equivalent relation}: here the walks $\gamma, \gamma_1$ and $\gamma_2$, represented respectively in plain blue, dashed red and dense dotted green, respectively, are equivalent. However, the walk $\gamma_3$, represented in sparse dotted purple, is not equivalent to any of the other walks. In particular, it is not possible to construct a function $\sigma$ satisfying condition (ii) of Definition \ref{Def: equivalent relation} for the walk $\gamma_3$. We have $\vert \gamma \vert=4$, $\vert \gamma_{1} \vert=9$, $ \vert \gamma_2 \vert=7 $, and $\sigma_1(0)=(0,0)$, $\sigma_1(1)=(1,3)$, $\sigma_1(2)=(4,4)$, $\sigma_{1}(3)=(5,8)$, $\sigma_{2}(0)=(0,3)$, $\sigma_{2}(1)=(4,4)$, $\sigma_{2}(2)=(5,5)$ and $\sigma_{2}(3)=(6,6)$.} 
\label{Figure: Equivalent paths}
\end{figure}
We introduce the notion of the mono-directional graph associated to a walk $\gamma$ in the following definition. 
\begin{definition}
    The mono-directional graph associated to a walk \\
    $\gamma=(0,\gamma(1),\cdots,\gamma(k))$ is the graph $(V_{\gamma},E_{\gamma})$ where
    \begin{align*}
        &V_{\gamma}:=\{0, \gamma(1),\cdots,\gamma(k)\}, \\
        &E_{\gamma}:=\{(0,\gamma(1)), (\gamma(1),\gamma(2)),\cdots,(\gamma(k-1),\gamma(k))\}.
    \end{align*}
    In other words, it is the graph composed of the successive subpopulations $$\big{(} Z_{(0)}^{(n)}, Z_{(0,\gamma(1))}^{(n)}, \cdots, Z_{\gamma}^{(n)}\big{)}.$$
\end{definition}
Now we have all the preliminary results and definitions necessary to prove Theorem \ref{Theorem: non increasin growth rate graph}. 
\begin{proof}[Proof of Theorem \ref{Theorem: non increasin growth rate graph}]
We prove Equations \eqref{Equation: conv neutral case} and \eqref{Equation: conv neutral case random scale}. The proofs of Equations \eqref{Equation: convergence deleterious case} and \eqref{Equation: convergence deleterious case in random scale} are similar and are left to the reader. 

\textbf{Step 1:} Let $\widetilde{\gamma}$ be a representative candidate of an equivalence class. Our first step is to prove, using the results of Section \ref{Section: proof mono directional graph}, that for all $\varepsilon>0$
\begin{align}
\label{Equ4}
    \mP \Big{(}\sup_{s \in [-M,M]} \sup_{t \in [0,T]} \Big{\vert} \sum_{\gamma \in [\widetilde{\gamma}]}\frac{Z_{\gamma}^{(n)}\big{(}\mathfrak{t}_{t}^{(n)}+s\big{)}}{d_{\widetilde{\gamma}}^{(n)}(t,s)}-Ww_{\widetilde{\gamma}}(t)\Big{\vert} \geq \varepsilon\Big{)}\underset{n \to \infty}{\longrightarrow}0,
\end{align}
where for all $\gamma \in \Gamma(V)$,
\begin{align}
    d_{\gamma}^{(n)}(t,s):=&\1_{\left\{t \in [0,t(\gamma)-h_n^{-1}\right\}}+\1_{\left\{ t \in [t(\gamma)-h_{n}^{-1},t(\gamma))\right\}}\psi_n \log^{\theta(\gamma)-1}(n)\\
    &\hspace{2cm}+\1_{\left\{t \in [t(\gamma),\infty)\right\}}n^{t-t(\gamma)}\log^{\theta(\gamma)}(n)e^{\lambda(0) s},
\end{align}
and $w_{\gamma}$ is defined in \eqref{weight}. Notice that 
\begin{align}
    &\mP \Big{(} \sup_{s \in [-M,M]} \sup_{t \in [0,T]} \Big{\vert}\sum_{\gamma \in [\widetilde{\gamma}]}\frac{Z_{\gamma}^{(n)}\big{(}\mathfrak{t}_{t}^{(n)}+s\big{)}}{d_{\widetilde{\gamma}}^{(n)}(t,s)}-Ww_{\widetilde{\gamma}}(t)\Big{\vert} \geq \varepsilon\Big{)} \\
    &\hspace{1cm}\label{Eq: 41} \leq \mP \Big{(} \sup_{s \in [-M,M]} \sup_{t \in [0,T]} \Big{\vert}\frac{Z_{\widetilde{\gamma}}^{(n)}\big{(}\mathfrak{t}_{t}^{(n)}+s\big{)}}{d_{\widetilde{\gamma}}^{(n)}(t,s)}-Ww_{\widetilde{\gamma}}(t)\Big{\vert} \geq \varepsilon\Big{)} \\
    &\hspace{1cm}\label{Eq: 42}+\sum_{\gamma \in [\widetilde{\gamma}]\backslash \{\widetilde{\gamma}\}: t(\gamma)\leq T} \mP \Big{(} \sup_{s \in [-M,M]} \sup_{t \in [0,T]} \Big{\vert}\frac{Z_{\gamma}^{(n)}\big{(}\mathfrak{t}_{t}^{(n)}+s\big{)}}{d_{\widetilde{\gamma}}^{(n)}(t,s)}\Big{\vert} \geq \varepsilon\Big{)}\\ 
    &\hspace{1cm}\label{Eq: 43}+\mP \Big{(} \sup_{s \in [-M,M]} \sup_{t \in [0,T]} \sum_{\gamma \in [\widetilde{\gamma}]\backslash \{\widetilde{\gamma}\}: t(\gamma)> T}\Big{\vert} \frac{Z_{\gamma}^{(n)}\big{(}\mathfrak{t}_{t}^{(n)}+s\big{)}}{d_{\widetilde{\gamma}}^{(n)}(t,s)}\Big{\vert} \geq \varepsilon\Big{)}. 
\end{align}
The term in \eqref{Eq: 41} converges to 0 by applying Equation \eqref{Equation: conv neutral case} to the mono-directional graph defined by the walk $\widetilde{\gamma}$, as proven in Section \ref{Section: proof mono directional graph}. The term in \eqref{Eq: 42} also converges to $0$ since:
\begin{itemize}
    \item the sum is over a finite set, as we are considering a finite graph with positive labels on the edges ,
    \item for each $\gamma \in [\widetilde{\gamma}] \backslash \{\widetilde{\gamma}\}$, we have $t(\gamma)>t(\widetilde{\gamma})$ by definition of the representative (see Definition \ref{Def: equivalent relation}). This implies, by applying Equation \eqref{Equation: conv neutral case} to the mono-directional graph defined by $\gamma$, that 
    \begin{align}
        \mP \Big{(} \sup_{s \in [-M,M]} \sup_{t \in [0,T]} \Big{\vert} \frac{Z_{\gamma}^{(n)}\big{(}\mathfrak{t}_{t}^{(n)}+s\big{)}}{d^{(n)}_{\widetilde{\gamma}}(t,s)}\Big{\vert} \geq \varepsilon \Big{)}\underset{n \to \infty}{\longrightarrow}0.
    \end{align}
\end{itemize}
The term in \eqref{Eq: 43} converges to $0$ because
\begin{align} \label{Eq: 101}
    \mP  \Big{(} \sup_{s \in [-M,M]} \sup_{t \in [0,T]} \sum_{\gamma \in [\widetilde{\gamma}]\backslash \{\widetilde{\gamma}\}: t(\gamma)> T} Z_{\gamma}^{(n)}\big{(}\mathfrak{t}_{t}^{(n)}+s \big{)} =0\Big{)} \underset{n \to \infty}{\to}1.
\end{align}
Indeed, for each $\gamma \in [\widetilde{\gamma}] \backslash \{\widetilde{\gamma}\}$ satisfying $t(\gamma)>T$, we have
\begin{align} \label{Eq: 100}
    \mP \Big{(}\sup_{s \in [-M,M]}\sup_{t \in [0,T]} Z_{\gamma}^{(n)}\big{(} \mathfrak{t}_{t}^{(n)}+s\big{)}=0 \Big{)} \underset{n \to \infty}{\longrightarrow}1,
\end{align}
by applying Lemma \ref{Lem: no mutant cell} to the mono-directional graph given by $\gamma$. It remains to handle the sum over the set $A_{\widetilde{\gamma}}(T):= \left\{ \gamma \in [\widetilde{\gamma}] \backslash \{\widetilde{\gamma}\}: t(\gamma)>T \right\}$. The easiest situation occurs when $\vert A_{\widetilde{\gamma}}(T) \vert < \infty$, as the result follows directly in this case. This situation corresponds exactly to the case where there is no cycle in the graph structure $(V,E)$ for the vertices of $\widetilde{\gamma}$. Now, consider the case $\vert A_{\widetilde{\gamma}}(T) \vert =\infty$. In this case, even though Equation \eqref{Eq: 100} holds for all $\gamma \in A_{\widetilde{\gamma}}(T)$, it does not necessary imply that Equation \eqref{Eq: 101} is automatically satisfied. The result follows if one can show that there exists a finite subset $B_{\widetilde{\gamma}}(T) \subset A_{\widetilde{\gamma}}(T)$ such that 
\begin{align} \label{Eq: 102}
    &\mP \Big{(}\sup_{s \in [-M,M]} \sup_{t \in [0,T]} \sum_{\gamma \in A_{\widetilde{\gamma}}(T) \backslash B_{\widetilde{\gamma}}(T)} Z_{\gamma}^{(n)}\big{(} \mathfrak{t}_{t}^{(n)}+s\big{)}=0 \Big{\vert} E^{(n)}_{\widetilde{\gamma}} \Big{)}=1,
\end{align}
where $E^{(n)}_{\widetilde{\gamma}}:= \big{\{}\sup_{s \in [-M,M]} \sup_{t \in [0,T]}\sum_{\gamma \in B_{\widetilde{\gamma}}(T)} Z_{\gamma}^{(n)}\big{(} \mathfrak{t}_{t}^{(n)}+s\big{)}=0 \big{\}}.$
We will now show that $B_{\widetilde{\gamma}}(T)$ exists. The set $[\widetilde{\gamma}]$ consists of walks where, for each vertex $v$ visited by $\widetilde{\gamma}$, there may be a cycle going back to $v$. Since there are only a finite number of vertices visited by $\widetilde{\gamma}$, and the labels on the vertices are positive, it follows that the number of walks $\gamma \in A_{\widetilde{\gamma}}(T)$ for which we need to control the event that they do not have any cells up to time $\mathfrak{t}_{T}^{(n)}+M$ is actually finite, and we denote this set by $B_{\widetilde{\gamma}}(T)$. Indeed, for all walks $\gamma \in A_{\widetilde{\gamma}}(T) \backslash B_{\widetilde{\gamma}}(T)$, there exists a walk $\gamma_1 \in B_{\widetilde{\gamma}}(T)$ such that cells in the subpopulation $Z_{\gamma}^{(n)}$ result from (potentially many) mutations of cells in the subpopulation $Z_{\gamma_{1}}^{(n)}$. Therefore, if one controls with high probability that no cells are generated up to time $\mathfrak{t}_{T}^{(n)}+M$ for the subpopulations indexed by $\gamma \in B_{\widetilde{\gamma}}^{(n)}$-which is feasible since $B_{\widetilde{\gamma}}^{(n)}$ is finite-it automatically implies by the mechanistic construction of the process that, under such an event, there are almost surely no cells in the subpopulations indexed by $\gamma \in A_{\widetilde{\gamma}}^{(n)} \backslash B_{\widetilde{\gamma}}^{(n)}$. This is precisely the statement of Equation \eqref{Eq: 102}. 

\textbf{Step 2:} In this step, Equation \eqref{Equation: conv neutral case} is proven. Notice that for $\gamma \in A(v)$, where $A(v)$ is defined in Definition \ref{Def: Admissible path}, we have $d_{\gamma}^{(n)}(t,s)=d_{v}^{(n)}(t,s)$, and also that $\gamma$ is the representative candidate $\widetilde{\gamma}$ of its equivalence class. In particular, this means that $\sum_{\gamma \in A(v)}w_{\gamma}(t)=\sum_{\widetilde{\gamma} \in C(v): \widetilde{\gamma} \in A(v)}w_{\widetilde{\gamma}}(t)$, where $C(v)$ is defined in Definition \ref{Def: equivalent relation}. The proof is obtained by noting that 
\begin{align}
    &\mP \Big{(} \sup_{s \in [-M,M]} \sup_{t \in [0,T]} \Big{\vert} \frac{Z_{v}^{(n)}\big{(}\mathfrak{t}_{t}^{(n)}+s\big{)}}{d_{v}^{(n)}(t,s)}-W \sum_{\gamma \in A(v)}w_{\gamma}(t)\Big{\vert} \geq \varepsilon\Big{)} \\
    & \hspace{1cm}\label{Equati: 3}\leq \sum_{\widetilde{\gamma} \in C(v): \widetilde{\gamma} \in A(v)}\mP \Big{(} \sup_{s \in [-M,M]} \sup_{t \in [0,T]} \Big{\vert}  \sum_{\gamma \in [\widetilde{\gamma}]}\frac{Z_{\gamma}^{(n)}\big{(}\mathfrak{t}_{t}^{(n)}+s\big{)}}{d_{\widetilde{\gamma}}^{(n)}(t,s)}-Ww_{\widetilde{\gamma}}(t)\Big{\vert} \geq \varepsilon\Big{)} \\
    & \hspace{1cm}\label{Equati: 4}+\sum_{\widetilde{\gamma} \in C(v): \widetilde{\gamma} \notin A(v)}\mP \Big{(} \sup_{s \in [-M,M]} \sup_{t \in [0,T]} \Big{\vert} \sum_{\gamma \in [\widetilde{\gamma}]}\frac{Z_{\gamma}^{(n)}\big{(}\mathfrak{t}_{t}^{(n)}+s\big{)}}{d_{v}^{(n)}(t,s)}\Big{\vert} \geq \varepsilon\Big{)}.
\end{align}
Indeed, \eqref{Equati: 3} converges to 0 by applying Equation \eqref{Equ4} and because the sum is finite. Similarly, \eqref{Equati: 4} converges to 0 because the sum is finite and, for all $\widetilde{\gamma} \in C(v), \widetilde{\gamma} \notin A(v)$, we have either $t(\widetilde{\gamma})>t(v)$ or $\theta(\widetilde{\gamma})<\theta(v)$. 

\textbf{Step 3:} In this step, we are going to prove Equation \eqref{Equation: conv neutral case random scale}. By following the proof of Proposition \ref{Prop: conv diff random time}, replacing $\eta_{t}^{(n)}$ with $\tau_{t}^{(n)}$, and defining $W$ as in \eqref{W def general finite graph} instead of \eqref{W def}, we obtain that for all $0<T_{1}<T_{2}$ and for all $\varepsilon>0$,
\begin{equation}
    \mP \Big{(}\sup_{t \in [T_1,T_2]} \big{(}\tau_{t}^{(n)}-\sigma_{t}^{(n)} \big{)} \leq \varepsilon \Big{\vert} W>0\Big{)} \underset{n \to \infty}{\longrightarrow}1.
\end{equation}
Indeed, because the number of vertices in the graph is finite, and due to Step 2, we have shown that the total number of mutant cells $\sum_{v \in V \backslash \{0\}}Z_{v}^{(n)} \big{(}\mathfrak{t}^{(n)}_{t}+s\big{)}$ is negligible compared to the number of wild-type cells $Z_{(0)}^{(n)} \big{(}\mathfrak{t}^{(n)}_{t}+s\big{)}$ for any time interval $[T_1,T_2]$. This allows us to apply the reasoning from \eqref{Eq: mutant pop at least certain size} and \eqref{Eq: control mutant pop proof convergence stopping times}, leading to a straightforward adaptation of the proof of Proposition \ref{Prop: conv diff random time}. By adapting the different proofs from Subsection \ref{Subsection: monodirectional graph random time}, we obtain that for all $0<T_{1}<T_{2}$, $M>0$ and $\varepsilon>0$, 
\begin{align}
    &\mP \Big{(}\sup_{s \in [-M,M]}\sup_{t \in [T_1,T_2]} \Big{\vert} \frac{Z_{v}^{(n)}\big{(}\rho_{t}^{(n)}+s\big{)}}{d_v^{(n)}(t,s)}-\1_{ \{W>0\}}w_{v}(t)\Big{\vert} \geq \varepsilon\Big{)}\underset{n \to \infty}{\longrightarrow} 0.
\end{align}
\end{proof}

\begin{appendix} \label{appn}

\section*{}

\begin{proof}[Proof of Lemma \ref{Lem martingale type 1 pop}]
For all $t \geq 0$ let $\mathcal{F}_{i,t}^{(n)}$ the $\sigma$-field generated by $Z^{(n)}_{j}(s)$ for all $0 \leq j\leq i$ and $0 \leq s \leq t$. For all $h \geq 0$, we have
\begin{align} \label{E15}
    &\E\Big[M_{i}^{(n)}(t+h)-M_{i}^{(n)}(t) \vert\mathcal{F}_{i,t}^{(n)}\Big]=\E\Big[Z^{(n)}_{i}(t+h)\Big{\lvert} \mathcal{F}_{i,t}^{(n)} \Big] e^{-\lambda_{i}^{(n)}(t+h)}\\
    &\hspace{2cm}-Z^{(n)}_{i}(t)e^{-\lambda_{i}^{(n)} t}-\int_{t}^{t+h}2\alpha_{i-1} \mu_{i-1}^{(n)} e^{-\lambda_{i}^{(n)} s}\E\Big[Z^{(n)}_{i-1}(s) \Big{\lvert} \mathcal{F}_{i,t}^{(n)} \Big]ds.
\end{align}

The forward Chapman-Kolmogorov equation gives the time-differential equation
\begin{equation}
\label{time-differential equation of E(B)}
    \begin{split}
        \dfrac{d\E\Big[Z^{(n)}_{i}(t)\Big]}{dt}=\lambda_{i}^{(n)} \E\Big[Z^{(n)}_{i}(t)\Big]+2\alpha_{i-1}\mu_{i-1}^{(n)} \E\Big[Z^{(n)}_{i-1}(t)\Big],
    \end{split}
\end{equation}
which leads to
\begin{equation}
\label{esperance B}
    \E_{Z_{i}^{(n)}(0)}\Big[Z^{(n)}_{i}(t)\Big]=Z^{(n)}_{i}(0)e^{\lambda_{i}^{(n)} t}+\int_{0}^{t}2\alpha_{i-1}\mu_{i-1}^{(n)} \E_{Z_{i}^{(n)}(0)}\Big[Z^{(n)}_{i-1}(s)\Big]e^{\lambda_{i}^{(n)}(t-s)}ds.
\end{equation}
In particular, by using the Markov property we obtain that
\begin{align} \label{E16}
    \E\Big[Z^{(n)}_{i}(t+h)\Big{\lvert} \mathcal{F}_{i,t}^{(n)} \Big] =Z^{(n)}_{i}(t)e^{\lambda_{i}^{(n)} h}+\int_{t}^{t+h}2\alpha_{i-1}\mu_{i-1}^{(n)} \E\Big[ Z^{(n)}_{i-1}(s)\lvert \mathcal{F}_{i,t}^{(n)}\Big] e^{\lambda_{i}^{(n)} (t+h-s)}ds.
\end{align}
Combining \eqref{E15} and \eqref{E16}, it follows that $\big(M_{i}^{(n)}(t)\big)_{t \in \mathbb{R}^{+}}$ is a martingale. Let $$F^{(n)}(t,x,y):=(e^{-\lambda_{i}^{(n)} t}x-y)^2.$$ Then, we have $$\frac{\partial F^{(n)}}{\partial t}(t,x,y)=-2\lambda_{i}^{(n)}x e^{-\lambda_{i}^{(n)} t}\sqrt{F^{(n)}(t,x,y)} \text{ and } \frac{\partial F^{(n)}}{\partial y}(t,x,y)=-2\sqrt{F^{(n)}}.$$ Applying Itô's formula with $x=Z_{i}^{(n)}(t)$ and $y=\int_0^t2\alpha_{i-1} \mu_{i-1}^{(n)} e^{-\lambda_{i}^{(n)} s} Z^{(n)}_{i-1}(s) ds$, we obtain 
\begin{align}
        &\Big(M^{(n)}_{i}(t)\Big)^2=F^{(n)}\Big(t,Z_{i}^{(n)}(t),\int_0^t2\alpha_{i-1} \mu_{i-1}^{(n)} e^{-\lambda_{i}^{(n)}s} Z^{(n)}_{i-1}(s) ds\Big) \\ 
        &=F^{(n)}(0,0,0)-2\int_0^t 2\alpha_{i-1} \mu_{i-1}^{(n)} e^{-\lambda_{i}^{(n)} s}Z^{(n)}_{i-1}(s) M^{(n)}_{i}(s)  ds\\
        &-2 \lambda_{i}^{(n)} \int_0^t e^{-\lambda_{i}^{(n)} s} Z^{(n)}_{i}(s) M^{(n)}_{i}(s)ds \\ 
        &+\int_0^t \int_{\mathbb{R}^{+}} \Big{[} \Big{(}M_{i}^{(n)}(s^{-})\\
        &\hspace{2cm}+e^{-\lambda_{i}^{(n)} s}\Big{\{}\1_{\left\{ \theta \leq \alpha_{i}^{(n)}Z^{(n)}_{i}(s^{-})\right\}}-\1_{\left\{\alpha_{i}^{(n)}Z_{i}^{(n)}(s^{-}) \leq \theta \leq \left(\alpha_{i}^{(n)}+\beta_{i}\right)Z_{i}^{(n)}(s^{-}) \right\}}\Big{\}}\Big{)}^2 \\ 
        & \hspace{8cm}-\Big(M_{i}^{(n)}(s^{-}) \Big)^2 \Big{]} Q_{i}(ds,d\theta) \\ 
        &+\int_0^t \int_{\mathbb{R}^{+}} \Big{[} \Big{(}M_{i}^{(n)}(s^{-})+e^{-\lambda_{i}^{(n)} s}\1_{\left\{\theta \leq 2\alpha_{i-1} \mu_{i-1}^{(n)}\left(1-\mu_{i-1}^{(n)}\right)Z_{i-1}^{(n)}(s^{-})\right\}} \Big{)}^{2}\\
        &\hspace{8cm}-\Big(M_{i}^{(n)}(s^{-}) \Big)^2 \Big{]}N_{i-1}(ds,d\theta)\\
        &+\int_0^t \int_{\mathbb{R}^{+}} \Big{[} \Big{(}M_{i}^{(n)}(s^{-})+e^{-\lambda_{i}^{(n)} s}2\1_{\left\{ \theta \leq \alpha_{i-1} \left( \mu_{i-1}^{(n)}\right)^{2}Z_{i-1}^{(n)}(s^{-})\right\}} \Big{)}^{2}\\
        &\hspace{8cm}-\Big(M_{i}^{(n)}(s^{-}) \Big)^2 \Big{]}Q_{i-1}^{m}(ds,d\theta)\\
        &+\int_0^t \int_{\mathbb{R}^{+}} \Big{[} \Big{(}M_{i}^{(n)}(s^{-})-e^{-\lambda_{i}^{(n)} s}\1_{\left\{ \theta \leq \alpha_{i} \left( \mu_{i}^{(n)}\right)^{2}Z_{i}^{(n)}(s^{-})\right\}} \Big{)}^{2}-\Big(M_{i}^{(n)} \Big)^2 \Big{]}Q_{i}^{m}(ds,d\theta)\\
        &=-2 \int_0^t  \left(2\alpha_{i-1} \mu_{i-1}^{(n)}Z^{(n)}_{i-1}(s) +\lambda_{i}^{(n)} Z^{(n)}_{i}(s) \right)e^{-\lambda_{i}^{(n)} s}M^{(n)}_{i}(s) ds\\
        &+ 2  \int_0^t \left(2\alpha_{i-1} \mu_{i-1}^{(n)}Z^{(n)}_{i-1}(s)+\lambda_{i}^{(n)}Z^{(n)}_{i}(s) \right)e^{-\lambda_{i}^{(n)} s} M^{(n)}_{i}ds \\
        &+\int_0^t \left[ 2\alpha_{i-1} \mu_{i-1}^{(n)} Z^{(n)}_{i-1}(s)+\left(\alpha_{i}^{(n)}+\beta_{i}^{(n)}\right) Z^{(n)}_{i}(s)\right]e^{-2 \lambda_{i}^{(n)} s}ds +\widetilde{M}^{(n)}_{i}(t) \\ 
        &=\widetilde{M}^{(n)}_{i}(t)+\int_0^t 2\alpha_{i-1}\mu_{i-1}^{(n)}e^{-2\lambda_{i}^{(n)} s}Z_{i-1}^{(n)}(s) ds+\left(\alpha_{i}^{(n)}+\beta_{i}^{(n)}\right)\int_0^t e^{-2 \lambda_{i}^{(n)} s}Z_{i}^{(n)}(s) ds,
\end{align}
where $(\widetilde{M}^{(n)}_{i}(t))_{t \geq 0}$ is a martingale. Finally, we obtain
\begin{equation}
    \Big\langle M^{(n)}_{i}\Big\rangle_t=\int_0^t 2\alpha_{i-1}\mu_{i-1}^{(n)}e^{-2\lambda_{i}^{(n)} s}Z_{i}^{(n)}(s) ds+\Big(\alpha_{i}^{(n)}+\beta_{i}^{(n)}\Big)\int_0^t e^{-2 \lambda_{i}^{(n)} s}Z_{i}^{(n)}(s) ds.
\end{equation}
\end{proof}

\begin{proof}[Proof of Lemma \ref{Lem:upper bound expectation subpop}]
First, we have that $\E\big[Z_{0}^{(n)}(u)\big]=e^{\lambda_{0}^{(n)}u} \leq e^{\lambda_{0}u}$, which is exactly the result for $i=0$. Then, for $i \in \mathbb{N}$ assume that the result is true for $i-1$. Taking the expected value of the martingale $M_{i}^{(n)}$ defined in \eqref{Martingale definition} at time $u$ and using the induction assumption we obtain the following
\begin{align}
    \E\Big[Z_{i}^{(n)}(u)\Big]&\leq e^{\lambda_i u} \int_0^u 2\alpha_{i-1}\mu_{i-1}^{(n)} e^{-\lambda_i s}\E \Big[Z_{i-1}^{(n)}(s)\Big]ds\\ 
    &\leq C_{i-1}\mu_{\otimes,i}^{(n)} 2\alpha_{i-1}\int_{0}^{u}e^{(\lambda_{0}-\lambda_{i})s}ds u^{\theta(i-1)} e^{\lambda_i u}\\
    & \leq C_{i-1} \mu_{\otimes,i}^{(n)}2\alpha_{i-1} \Big(\1_{\{\lambda_i=\lambda_0\}}u+\1_{\{\lambda_i<\lambda_0\}}\frac{1}{\lambda_0-\lambda_i}e^{\left(\lambda_0-\lambda_i\right)u}\Big) u^{\theta(i-1)} e^{\lambda_i u} \\
    &=C_{i-1}\mu_{\otimes,i}^{(n)} 2\alpha_{i-1} \Big(\1_{\{\lambda_i=\lambda_0\}}+\1_{\{\lambda_i<\lambda_0\}}\frac{1}{\lambda_0-\lambda_i}\Big) u^{\theta(i)} e^{\lambda_0 u},
\end{align}
which concludes the proof by induction. 
\end{proof}

\begin{proof}[Proof of Lemma \ref{Lem: control quadratic variation}]
In the proof, $C$ represents a positive constant that may change from line to line. 

\textbf{Neutral case:} Assume that $\lambda_i=\lambda_0$. Applying Lemma \ref{Lem:upper bound expectation subpop}, recalling that $\lambda_{i}^{(n)}=\lambda_0-2\alpha_{i}\mu_{i}^{(n)}$, and noting that there exists $N_1\in \mathbb{N}$ such that for all $n \geq N_1$, we have that $e^{4\alpha_i\mu_i^{(n)}\left(\mathfrak{t}^{(n)}_{t_2}+s\right)}\leq 2$, we obtain 
\begin{align}
\label{Eq:12}           \int_{\mathfrak{t}^{(n)}_{t_1^{(n)}}+s}^{\mathfrak{t}^{(n)}_{t_2}+s}e^{-2\lambda_{i}^{(n)}u}\E\Big[Z_{i}^{(n)}(u)\Big]du &\leq  C\mu_{\otimes,i}^{(n)}\int_{\mathfrak{t}^{(n)}_{t_1^{(n)}}+s}^{\mathfrak{t}^{(n)}_{t_2}+s} u^{\theta(i)}e^{-\lambda_0u}du.
\end{align}
Using integration by parts, we obtain 
\begin{align}
\int_{\mathfrak{t}^{(n)}_{t_1^{(n)}}+s}^{\mathfrak{t}^{(n)}_{t_2}+s} u^{\theta(i)}e^{-\lambda_0u}du\leq \frac{1}{\lambda(0)}\big(\mathfrak{t}^{(n)}_{t^{(n)}_{1}}+s\big)^{\theta(i)} e^{-\lambda_0 \big(\mathfrak{t}^{(n)}_{t^{(n)}_{1}}+s\big)}+\frac{\theta(i)}{\lambda_0}\int_{\mathfrak{t}^{(n)}_{t_1^{(n)}}+s}^{\mathfrak{t}^{(n)}_{t_2}+s} u^{\theta(i)-1}e^{-\lambda_0u}du.
\end{align}
Then, using $\theta(i)$ integrations by parts, there exists $N_2 \in \mathbb{N}$ such that for $n \geq N_2$, we have
\begin{align}
\label{Eq: 31}
    \int_{\mathfrak{t}^{(n)}_{t_1^{(n)}}+s}^{\mathfrak{t}^{(n)}_{t_2}+s} u^{\theta(i)}e^{-\lambda_0u}du\leq C\frac{e^{-\lambda_0 s}}{ n^{t_1^{(n)}}}\big(\mathfrak{t}^{(n)}_{t_1^{(n)}}+s\big)^{\theta(i)}.
\end{align}
It follows that for $n \geq \max(N_1,N_2)$,
\begin{align}
\int_{\mathfrak{t}^{(n)}_{t_1^{(n)}}+s}^{\mathfrak{t}^{(n)}_{t_2}+s}e^{-2\lambda_{i}^{(n)}u}\E\Big[Z_{i}^{(n)}(u)\Big]du &\leq C\frac{e^{-\lambda_0 s}}{n^{t_1^{(n)}}}\mu_{\otimes,i}^{(n)}\big(\mathfrak{t}^{(n)}_{t_1^{(n)}}+s\big)^{\theta(i)}.
\end{align}
Since vertex $i$ is assumed to be neutral, we have $\theta(i-1)=\theta(i)-1$. Using similar computation as above, there exists $N_3 \in \mathbb{N}$ such that for $n \geq N_3$, we have 
\begin{align}
\label{Eq:15}
\int_{\mathfrak{t}^{(n)}_{t_1^{(n)}}+s}^{\mathfrak{t}^{(n)}_{t_2}+s}\mu_{i-1}^{(n)}e^{-2\lambda_{i}^{(n)} u}\E\Big[Z_{i-1}^{(n)}(u)\Big]du \leq C\frac{e^{-\lambda_0 s}}{n^{t_1^{(n)}}}\mu_{\otimes,i}^{(n)}\big(\mathfrak{t}^{(n)}_{t_1^{(n)}}+s\big)^{\theta(i)-1}.
\end{align}
It follows that for all $n \geq \max(N_1,N_2,N_3)$, we have 
\begin{align}
    \E \Big[\big\langle M_{i}^{(n)}\big\rangle_{\mathfrak{t}^{(n)}_{t_2}+s}-\big\langle M_{i}^{(n)}\big\rangle_{\mathfrak{t}^{(n)}_{t_1^{(n)}}+s}\Big] &\leq C\frac{e^{-\lambda_0 s}}{n^{t_1^{(n)}}}\mu_{\otimes,i}^{(n)}\big(\mathfrak{t}^{(n)}_{t_1^{(n)}}+s\big)^{\theta(i)} . 
\end{align}

\textbf{Deleterious case:} We now address the case $\lambda_i<\lambda_0$ by applying the same strategy. We obtain 
\begin{align}
\label{Eq: eq 10}
&\int_{\mathfrak{t}^{(n)}_{t_1^{(n)}}+s}^{\mathfrak{t}^{(n)}_{t_2}+s}e^{-2\lambda_{i}^{(n)} u}\E\Big[Z_{i}^{(n)}(u)\Big]du \leq C\mu_{\otimes,i}^{(n)} \int_{\mathfrak{t}^{(n)}_{t_1^{(n)}}+s}^{\mathfrak{t}^{(n)}_{t_2}+s}u^{\theta(i)}e^{\left(\lambda_0-2\lambda_i\right)u}du \\
    &\hspace{2cm}\leq C\mu_{\otimes,i}^{(n)} \left(\mathfrak{t}^{(n)}_{t_2}+s\right)^{\theta(i)} \Big{[}\1_{\{\lambda_0>2\lambda_i\}} e^{\left(\lambda_0-2\lambda_i\right)\left( \mathfrak{t}^{(n)}_{t_2}+s\right)}\\
    &\hspace{4cm}+\1_{\{\lambda_0=2\lambda_i\}}\big(\mathfrak{t}^{(n)}_{t_2}+s\big)+\1_{\{\lambda_i<\lambda_0<2\lambda_i\}}e^{-\left(2\lambda_i-\lambda_0\right)\big(\mathfrak{t}^{(n)}_{t_1^{(n)}}+s\big)}\Big{]}.
\end{align}
Recalling that $\theta(i-1)=\theta(i)$, we find 
\begin{align}
\label{Eq: eq 11}
    &\int_{\mathfrak{t}^{(n)}_{t_1^{(n)}}+s}^{\mathfrak{t}^{(n)}_{t_2}+s}\mu_{i-1}^{(n)}e^{-2\lambda_{i}^{(n)} u}\E\Big[Z_{i-1}^{(n)}(u)\Big]du \leq C\mu_{\otimes,i}^{(n)} \big(\mathfrak{t}^{(n)}_{t_2}+s\big)^{\theta(i)}\Big{[}\1_{\{\lambda_0>2\lambda_i\}}e^{\left(\lambda_0-2\lambda_i\right)(\mathfrak{t}^{(n)}_{t_2}+s)}\\
    &\hspace{4cm}+\1_{\{\lambda_0=2\lambda_i\}}\big(\mathfrak{t}^{(n)}_{t_2}+s\big)+\1_{\{\lambda_i<\lambda_0<2\lambda_i\}}e^{-\left(2\lambda_i-\lambda_0\right)\big( \mathfrak{t}^{(n)}_{t_1^{(n)}}+s\big)}\Big{]}.
\end{align}
Finally, we have 
\begin{align}
    &\E \Big[\big\langle M_{i}^{(n)}\big\rangle_{\mathfrak{t}^{(n)}_{t_2}+s}-\big\langle M_{i}^{(n)}\big\rangle_{\mathfrak{t}^{(n)}_{t_1^{(n)}}+s}\Big] \leq C \mu_{\otimes,i}^{(n)} \big(\mathfrak{t}^{(n)}_{t_2}+s\big)^{\theta(i)}\cdot\Big[\1_{\{\lambda_0>2\lambda_i\}}e^{\left(\lambda_0-2\lambda_i\right)(\mathfrak{t}^{(n)}_{t_2}+s)}\\
    &\hspace{2cm}+\1_{\{\lambda_0=2\lambda_i\}}\big(\mathfrak{t}^{(n)}_{t_2}+s\big)+\1_{\{\lambda_i<\lambda_0<2\lambda_i\}}e^{-\left(2\lambda_i-\lambda_0\right)\big(\mathfrak{t}^{(n)}_{t_1^{(n)}}+s\big)}\Big].
\end{align}
\end{proof}

\end{appendix}

\paragraph*{Acknowledgements}
The author would like to thank Hélène Leman for inspiring and helpful discussions as well as valuable feedback. The author would like to thank the Chair "Modélisation Mathématique et Biodiversité" of VEOLIA-Ecole Polytechnique-MNHN-F.X.

\paragraph*{Funding}
This research was led with financial support from ITMO Cancer of AVIESAN (Alliance Nationale pour les Sciences de la Vie et de la Santé, National Alliance for Life Sciences Health) within the framework of the Cancer Plan.

\end{document}